\documentclass[11pt,twoside,leqno]{article}
\usepackage{amssymb,amsmath,amsthm}
\usepackage[dvipdfmx]{graphicx}
\numberwithin{equation}{section}
\usepackage{ytableau}

\theoremstyle{plain}
\newtheorem{thm}{Theorem}[section]
\newtheorem{prop}[thm]{Proposition}
\newtheorem{lemma}[thm]{Lemma}
\newtheorem{cor}[thm]{Corollary}

\theoremstyle{definition}

\newtheorem{rmk}[thm]{Remark}
\newtheorem{exam}[thm]{Example}

\newenvironment{MSC}{%
\smallbreak
\noindent \textbf{2010\ Mathematics Subject Classification\,:}}

\newenvironment{keywords}{%
\noindent\textbf{Key words and phrases\,:}\itshape}

\title{On Schur multiple zeta functions: \\
 A combinatoric generalization of multiple zeta functions}
\author{
Maki Nakasuji\thanks{Supported by Sophia Lecturing-Research Grants 2015, Grant-in-Aid for Young Scientists (B) 15K17519 and JSPS Joint Research Project with CNRS.}, 
Ouamporn Phuksuwan and 
Yoshinori Yamasaki\thanks{Supported by Grant-in-Aid for Scientific Research (C) 15K04785.} 
} 

\date{\today}

\begin{document}

\setlength{\baselineskip}{15pt}
\maketitle

\begin{abstract} 
 We introduce Schur multiple zeta functions
 which interpolate both the multiple zeta and multiple zeta-star functions of the Euler-Zagier type combinatorially.
 We first study their basic properties including a region of absolute convergence and the case where all variables are the same.
 Then, under an assumption on variables, some determinant formulas coming from theory of Schur functions
 such as the Jacobi-Trudi, Giambelli and dual Cauchy formula are established 
 with the help of Macdonald's ninth variation of Schur functions. 
 Moreover, we investigate the quasi-symmetric functions corresponding to the Schur multiple zeta functions.
 We obtain the similar results as above for them and, furthermore,  
 describe the images of them by the antipode of the Hopf algebra of quasi-symmetric functions explicitly. 
 Finally, we establish iterated integral representations of the Schur multiple zeta values of ribbon type,
 which yield a duality for them in some cases.     
\begin{MSC}
 11M41, 
 05E05. 
\end{MSC} 
\end{abstract}

\tableofcontents


\section{Introduction}
\label{sec:Introduction}

 The multiple zeta function and the multiple zeta-star function (MZF and MZSF for short)
 of the Euler-Zagier type are respectively defined by the series
\begin{align*}
 \zeta({\pmb s})
=\sum_{m_1<\cdots <m_n}\frac{1}{m_1^{s_1}\cdots m_n^{s_n}}, \quad 
 \zeta^{\star}({\pmb s})
=\sum_{m_1\le \cdots \le m_n}\frac{1}{m_1^{s_1}\cdots m_n^{s_n}},
\end{align*}
 where ${\pmb s}=(s_1,\ldots,s_n)\in\mathbb{C}^n$.
 These series converge absolutely for $\Re(s_1),\ldots,\Re(s_{n-1})\ge 1$ and $\Re(s_n)>1$
 (see, e.g., \cite{Mat} for more precise description about the region of absolute convergence).
 One easily sees that a MZSF can be expressed as a linear combination of MZFs, and vice versa.
 For instance,
\begin{align*}
 \zeta^{\star}(s_1,s_2)
&=\zeta(s_1,s_2)+\zeta(s_1+s_2),\\
 \zeta(s_1,s_2)
&=\zeta^{\star}(s_1,s_2)-\zeta^{\star}(s_1+s_2),\\
 \zeta^{\star}(s_1,s_2,s_3)
&=\zeta(s_1,s_2,s_3)+\zeta(s_1+s_2,s_3)+\zeta(s_1,s_2+s_3)+\zeta(s_1+s_2+s_3),\\
 \zeta(s_1,s_2,s_3)
&=\zeta^{\star}(s_1,s_2,s_3)-\zeta^{\star}(s_1+s_2,s_3)-\zeta^{\star}(s_1,s_2+s_3)+\zeta^{\star}(s_1+s_2+s_3),
\end{align*}
 where $\zeta(s)=\zeta^{\star}(s)$ is the Riemann zeta function.
 More generally, we have  
\begin{equation}
\label{for:zsz}
 \zeta^{\star}({\pmb s})
=\sum_{{\pmb t} \,\preceq\, {\pmb s}}\zeta({\pmb t}), \quad 
 \zeta({\pmb s})
=\sum_{{\pmb t} \,\preceq\, {\pmb s}}(-1)^{n-\ell({\pmb t})}\zeta^{\star}({\pmb t}),
\end{equation}
 where, for ${\pmb t}=(t_1,\ldots,t_m)\in\mathbb{C}^m$, 
 $\ell({\pmb t})=m$ and
 ${\pmb t} \preceq {\pmb s}$ means that 
 ${\pmb t}$ is obtained from ${\pmb s}$ by combining some of its adjacent parts. 
 The special value of $\zeta(s_1,\ldots,s_n)$ and $\zeta^{\star}(s_1,\ldots,s_n)$ at positive integers
 were first introduced by Euler \cite{E} for $n=2$, and by Hoffman \cite{H1} and Zagier \cite{Za} for general $n$, independently.
 Many different types of relations among such values have been studied in references such as \cite{Zl,M,IKOO,OZ}. 

 The purpose of the present paper is
 to introduce a generalization of both MZF and MZSF,
 which we call a {\it Schur multiple zeta function},
 from the viewpoint of $n$-ple zeta functions.
 Indeed, it is defined similarly to the tableau expression of the Schur function as follows. 
 For a partition $\lambda$ of a positive integer $n$,
 let $T(\lambda,X)$ be the set of all Young tableaux of shape $\lambda$ over a set $X$ and,
 in particular, $\mathrm{SSYT}(\lambda)\subset T(\lambda,\mathbb{N})$ the set of all semi-standard Young tableaux of shape $\lambda$
 (see Section~\ref{sec:SMZ} for precise definitions).
 Recall that $M=(m_{ij})\in T(\lambda,\mathbb{N})$ is called semi-standard
 if $m_{i1}\le m_{i2}\le \cdots$ for all $i$ and $m_{1j}<m_{2j}<\cdots $ for all $j$.
 For ${\pmb s}=(s_{ij})\in T(\lambda,\mathbb{C})$,
 the Schur multiple zeta function (SMZF for short) associated with $\lambda$ is defined by the series 
\[
 \zeta_{\lambda}({\pmb s})
=\sum_{M\in \mathrm{SSYT}(\lambda)}\frac{1}{M^{\pmb s}}, 
\]
 where $M^{\pmb s}=\prod_{(i,j)\in D(\lambda)}m_{ij}^{s_{ij}}$ for $M=(m_{ij})\in \mathrm{SSYT}(\lambda)$
 with $D(\lambda)$ being the Young diagram of $\lambda$.
 It is shown in Lemma~\ref{lem:convergence} that
 the above series converges absolutely whenever ${\pmb s}\in W_{\lambda}$ where 
\[
 W_\lambda
=
\left\{{\pmb s}=(s_{ij})\in T(\lambda,\mathbb{C})\,\left|\,
\begin{array}{l}
 \text{$\Re(s_{ij})\ge 1$ for all $(i,j)\in D(\lambda) \setminus C(\lambda)$ } \\[3pt]
 \text{$\Re(s_{ij})>1$ for all $(i,j)\in C(\lambda)$}
\end{array}
\right.
\right\}
\]
 with $C(\lambda)$ being the set of all corners of $\lambda$.
 If $(1^n)$ and $(n)$ are denoted by the one column and one row partitions of $n$, 
 then it is clear that $\zeta_{(1^n)}({\pmb s})$ (${\pmb s} \in T((1^n),\mathbb{C})$) and 
 $\zeta_{(n)}({\pmb s})$ (${\pmb s} \in T((n),\mathbb{C})$) are nothing but MZF and MZSF, respectively.
 This shows that SMZFs actually interpolate both MZFs and MZSFs combinatorially.
 We remark that such interpolation multiple zeta functions were first mentioned in \cite{Y}
 from the study of the multiple Dirichlet $L$-values.
 
 In this paper, we study fundamental properties of SMZFs and establish some relations among them,
 which can be regarded as analogues of those for Schur functions.
 Indeed, we obtain the following Jacobi-Trudi formulas for SMZFs,
 which is one of the main results of our paper. 
 To describe the result, we need the set 
\[
 W^{\mathrm{diag}}_{\lambda}
=W_\lambda\cap T^{\mathrm{diag}}(\lambda,\mathbb{C}),
\]
 where, for a set $X$, $T^{\mathrm{diag}}(\lambda,X)=\{T=(t_{ij})\in T(\lambda,X)\,|\,\text{$t_{ij}=t_{kl}$ if $j-i=l-k$}\}$.
 For a tableau ${\pmb s}=(s_{ij})\in W^{\mathrm{diag}}_{\lambda}$, 
 we always write $a_k=s_{i,i+k}$ for $k\in\mathbb{Z}$ (and for any $i\in\mathbb{N}$). 
 For example, when $\lambda=(4,3,3,2)$,
 ${\pmb s}=(s_{ij})\in W^{\mathrm{diag}}_{(4,3,3,2)}$ implies that ${\pmb s}$ is of the form of  
\[
 {\pmb s}
=\,
\ytableausetup{boxsize=18pt,aligntableaux=center}
\begin{ytableau}
 s_{11} & s_{12} & s_{13} & s_{14} \\
 s_{21} & s_{22} & s_{23} \\
 s_{31} & s_{32} & s_{33} \\
 s_{41} & s_{42} 
\end{ytableau} 
\,
=
\,
\begin{ytableau}
 a_0    & a_1 & a_2 & a_3 \\
 a_{-1} & a_0 & a_1 \\
 a_{-2} & a_{-1} & a_0 \\
 a_{-3} & a_{-2} 
\end{ytableau} 
.
\]

\begin{thm}
\label{JT}
 Let $\lambda=(\lambda_1,\ldots,\lambda_r)$ be a partition and $\lambda'=(\lambda'_1,\ldots,\lambda'_s)$ the conjugate of $\lambda$. 
 Assume that ${\pmb s}=(s_{ij})\in W^{\mathrm{diag}}_{\lambda}$. 
\begin{itemize}
\item[$(1)$] 
 Assume further that $\Re(s_{i,\lambda_i})>1$ for all $1\le i \le r$.
 Then, we have
\begin{equation}
\label{Htype}
 \zeta_{\lambda}({\pmb s})
=\det\left[\zeta^{\star}(a_{-j+1},a_{-j+2},\ldots,a_{-j+(\lambda_{i}-i+j)})\right]_{1\le i,j\le r\,}.
\end{equation}
 Here, we understand that $\zeta^{\star}(\,\cdots)=1$ if $\lambda_{i}-i+j=0$ and $0$ if $\lambda_{i}-i+j<0$.
\item[$(2)$]
 Assume further that $\Re(s_{\lambda'_i,i})>1$ for all $1\le i \le s$.
 Then, we have
\begin{equation}
\label{Etype}
 \zeta_{\lambda}({\pmb s})
=\det\left[\zeta(a_{j-1},a_{j-2},\ldots,a_{j-(\lambda'_{i}-i+j)})\right]_{1\le i,j\le s}.
\end{equation}
 Here, we understand that $\zeta(\,\cdots)=1$ if $\lambda'_{i}-i+j=0$ and $0$ if $\lambda'_{i}-i+j<0$.
\end{itemize}
\end{thm}

 As in the case of Schur functions, 
 we call \eqref{Htype} and \eqref{Etype} of $H$-type and $E$-type, respectively.
 From these formulas, as corollaries,
 one can obtain many algebraic relations given by determinants among MZFs and MZSFs.
 For example, considering the case $\lambda=(1^n)$ and $\lambda=(n)$,
 we have the following identities.
 
\begin{cor} 
\label{cor:detformulas}
 For $s_1,\ldots,s_{n}\in\mathbb{C}$ with $\Re(s_1),\ldots,\Re(s_n)>1$,
 we have 
\begin{align*}
 \zeta(s_1,\ldots,s_n)
&=
\left|
\begin{array}{cccccc}
 \zeta^{\star}(s_1) & \zeta^{\star}(s_2,s_1) & & \cdots & \cdots & \zeta^{\star}(s_n,\ldots,s_2,s_1) \\
 1 & \zeta^{\star}(s_2) & & \cdots & \cdots & \zeta^{\star}(s_n,\ldots,s_2) \\
 & 1 & \ddots & & & \vdots \\
 & & \ddots & 1& \zeta^{\star}(s_{n-1}) & \zeta^{\star}(s_{n},s_{n-1}) \\
 \scalebox{2}{$0$} & & & & 1 & \zeta^{\star}(s_{n})
\end{array}
\right|,\\[5pt]
 \zeta^{\star}(s_1,\ldots,s_n)
&=
\left|
\begin{array}{cccccc}
 \zeta(s_1) & \zeta(s_2,s_1) & & \cdots & \cdots & \zeta(s_n,\ldots,s_2,s_1) \\
 1 & \zeta(s_2) & & \cdots & \cdots & \zeta(s_n,\ldots,s_2) \\
  & 1 & \ddots & & & \vdots \\
 & & \ddots & 1& \zeta(s_{n-1}) & \zeta(s_{n},s_{n-1}) \\
 \scalebox{2}{$0$} & & & & 1 & \zeta(s_{n})
\end{array}
\right|.
\end{align*}
\end{cor}

 Moreover,
 just combining \eqref{Htype} and \eqref{Etype},
 we obtain a family of relations among MZFs and MZSFs.
 For example, considering the cases $\lambda=(2,2,1)$ and its conjugate $\lambda'=(3,2)$,
 we have 
\begin{align*}
\ytableausetup{boxsize=normal,aligntableaux=center}
 \zeta_{\lambda}
\left(~
\begin{ytableau}
 a & b \\
 c & a \\
 d  
\end{ytableau} 
~\right)
&= 
\left|
\begin{array}{ccc}
 \zeta^{\star}(a,b) & \zeta^{\star}(c,a,b) & \zeta^{\star}(d,c,a,b) \\
 \zeta^{\star}(a) & \zeta^{\star}(c,a) & \zeta^{\star}(d,c,a) \\
 0 & 1 & \zeta^{\star}(d)
\end{array}
\right|
= 
\left|
\begin{array}{cc}
 \zeta(a,c,d) & \zeta(b,a,c,d)\\
 \zeta(a) & \zeta(b,a)
\end{array}
\right|
,\\
 \zeta_{\lambda'}
\left(~
\begin{ytableau}
 a & c & d \\
 b & a  
\end{ytableau} 
~\right)
&= 
\left|
\begin{array}{cc}
 \zeta^{\star}(a,c,d) & \zeta^{\star}(b,a,c,d)\\
 \zeta^{\star}(a) & \zeta^{\star}(b,a)
\end{array}
\right|
=
\left|
\begin{array}{ccc}
 \zeta(a,b) & \zeta(c,a,b) & \zeta(d,c,a,b) \\
 \zeta(a) & \zeta(c,a) & \zeta(d,c,a) \\
 0 & 1 & \zeta(d)
\end{array}
\right|
,
\end{align*}
 where $a,b,c,d\in\mathbb{C}$ with $\Re(a),\Re(b),\Re(d)>1$ and $\Re(c)\ge 1$.
 As you can see in the above examples and Corollary~\ref{cor:detformulas},
 these kind of relations hold even if we replace $\zeta$ with $\zeta^{\star}$ and vice versa.
 
 It is also worth mentioning that both \eqref{Htype} and \eqref{Etype} give meromorphic continuations of $\zeta_{\lambda}({\pmb s})$
 to $T^{\mathrm{diag}}(\lambda,\mathbb{C})$ ($\,=\mathbb{C}^{s+r-1}$ where $s=\lambda_1$ and $r=\lambda'_1$) 
 as a function of $a_k$ for $1-r\le k\le 1+s$
 because both MZFs and MZSFs admit meromorphic continuations to the whole complex space (see, e.g., \cite{AET}).
 
 The assumption on variables on the same diagonal lines is crucial.
 Actually, in Section~\ref{sec:Macdonald},
 we find out that our SMZF, which can be easily generalized to the skew type,
 with the assumption is realized
 as (the limit of) a specialization of Macdonald's ninth variation of Schur function studied by Nakagawa, Noumi, Shirakawa and Yamada \cite{NNSY}.
 Based on this fact, we present some results such as the Jacobi-Trudi formula of skew type,
 the Giambelli formula and the dual Cauchy formula for SMZFs.
 Notice that if we work for such formulas without the assumption,
 then we encounter extra terms (see Remark~\ref{rmk:ErrorTerms}),
 which will be clarified in a forthcoming work. 
 
 Furthermore, in Section~\ref{sec:Sqsf},
 we study SMZFs in a more general framework, that is,
 in the Hopf algebra $\mathrm{QSym}$ of quasi-symmetric functions studied by Gessel \cite{G}.
 For a skew Young diagram $\nu$,
 we define a special type of quasi-symmetric function $S_{\nu}({\pmb \alpha})$,
 which we call a {\it Schur type quasi-symmetric function}, similarly to SMZFs.
 (Note that there is a different type of quasi-symmetric functions,
 called {\it quasi-symmetric Schur functions} defined by Haglund, Mason, Luoto and Willigenburg \cite{HLMW}, 
 as a basis of $\mathrm{QSym}$,
 which arise from the combinatorics of Macdonald polynomials and actually refine Schur functions in a natural way.)
 Then, we also prove the Jacobi-Trudi formulas of both $H$-type and $E$-type for such quasi-symmetric functions under the same assumption as above. 
 Notice that the former corresponds to \eqref{Htype} with entries in the essential quasi-symmetric functions
 and the latter to \eqref{Etype} with in the monomial quasi-symmetric functions. 
 Remark that when $\nu$ is the one column and one row partitions, 
 the corresponding formulas can be also respectively obtained
 by calculating the images of the essential and monomial quasi-symmetric functions
 by the antipode $S$ of $\mathrm{QSym}$ in two different ways,  
 as shown by Hoffman (\cite[Theorem~3.1]{H2}).
 More generally, for any skew Young diagram $\nu$,
 we calculate the image of $S_{\nu}({\pmb \alpha})$ by $S$ and see that 
 it is essentially equal to the Schur type quasi-symmetric function again associated with $\nu^{\#}$,
 the anti-diagonal transpose of $\nu$.
 
 In the final section,
 we give iterated integral representations of Schur multiple zeta values of ribbon type
 by following the similar discussion performed in \cite{KanekoYamamoto,Yamamoto}.
 As is the case of the multiple zeta values,
 one can obtain a duality for Schur multiple zeta values 
 by just making a change of variables in the integral representation 
 if the dual of the value is again of ribbon type.
 

\section{Schur multiple zeta functions}
\label{sec:SMZ}

\subsection{Combinatorial settings}

 We first set up some notions of partitions.
 A partition $\lambda=(\lambda_1,\ldots,\lambda_r)$ of a positive integer $n$
 is a non-increasing sequence of positive integers such that $|\lambda|=\sum^{r}_{i=1}\lambda_i=n$.
 We call $|\lambda|$ and $\ell(\lambda)=r$ the weight and length of $\lambda$, respectively.
 If $|\lambda|=n$, then we write $\lambda\vdash n$.
 We sometimes express $\lambda\vdash n$ as $\lambda=(n^{m_n(\lambda)}\cdots 2^{m_2(\lambda)}1^{m_1(\lambda)})$
 where $m_i(\lambda)$ is the multiplicity of $i$ in $\lambda$.
 We identify $\lambda\vdash n$ with its Young diagram $D(\lambda)=\{(i,j)\in\mathbb{Z}^2\,|\,1\le i\le r,\ 1\le j\le \lambda_i\}$,
 depicted as a collection of $n$ square boxes with $\lambda_i$ boxes in the $i$th row.
 We say that $(i,j)\in D(\lambda)$ is a corner of $\lambda$ if $(i+1, j)\notin D(\lambda)$ and $(i, j+1)\notin D(\lambda)$
 and denote by $C(\lambda) \subset D(\lambda)$ the set of all corners of $\lambda$.
 For example, $C((4,3,3,2))=\{(1,4),(3,3),(4,2)\}$.
 The conjugate $\lambda'=(\lambda'_1,\ldots,\lambda'_s)$ of $\lambda$ is defined by $\lambda'_i=\#\{j\,|\,\lambda_j\ge i\}$.
 Namely, $\lambda'$ is the partition whose Young diagram is the transpose of that of $\lambda$.
 For example, $(4,3,3,2)'=(4,4,3,1)$.

 Let $X$ be a set.
 For a partition $\lambda$, a Young tableau $T=(t_{ij})$ of shape $\lambda$ over $X$ is a filling of $D(\lambda)$ obtained by putting $t_{ij}\in X$ into $(i,j)$ box of $D(\lambda)$.
 Similarly to the above,
 the conjugate tableau of $T$ is defined by $T'=(t_{ji})$ whose shape is $\lambda'$.
 We denote by $T(\lambda,X)$ the set of all Young tableaux of shape $\lambda$ over $X$,
 which is sometimes identified with $X^{|\lambda|}$.
 Moreover, we put
\[
 T^{\mathrm{diag}}(\lambda,X)
=\left\{\left.(t_{ij})\in T(\lambda,X)\,\right|\,\text{$t_{ij}=t_{kl}$ if $j-i=l-k$}\right\},
\]
 which is identified with $X^{\lambda_1+\ell(\lambda)-1}$. 
 By a semi-standard Young tableau,
 we mean a Young tableau over the set of positive integers $\mathbb{N}$
 such that the entries in each row are weakly increasing from left to right and those in each column are strictly increasing from top to bottom.
 We denote by $\mathrm{SSYT}(\lambda)$ the set of all semi-standard Young tableaux of shape $\lambda$.

\subsection{Definition of Schur multiple zeta functions}

 For ${\pmb s}=(s_{ij})\in T(\lambda,\mathbb{C})$, 
 define
\begin{equation}
\label{def:SMZ}
 \zeta_{\lambda}({\pmb s})
=\sum_{M\in\mathrm{SSYT}(\lambda)}\frac{1}{M^{\pmb s}},
\end{equation}
 where $M^{\pmb s}=\prod_{(i,j)\in D(\lambda)}m_{ij}^{s_{ij}}$ for $M=(m_{ij})\in\mathrm{SSYT}(\lambda)$.
 We also define $\zeta_{\lambda}=1$ for the empty partition $\lambda=\emptyset$.
 We call $\zeta_{\lambda}({\pmb s})$ a {\it{Schur multiple zeta function}} (SMZF for short) associated with $\lambda$
 and sometimes write it shortly as ${\pmb s}$ if there is no confusion.
 Clearly, this is an extension of both MZFs and MZSFs.
 Actually, one sees that 
\[
\ytableausetup{boxsize=normal,aligntableaux=center}
 \zeta(s_1,\ldots,s_n)
=\zeta_{(1^n)}
\left(~ 
\begin{ytableau}
 s_1 \\
 \lower2pt\vdots \\
 s_{n}
\end{ytableau} 
~\right)
=
\,
\begin{ytableau}
 s_1 \\
 \lower2pt\vdots \\
 s_{n}
\end{ytableau} 
\,, \quad
 \zeta^{\star}(s_1,\ldots,s_n)
=\zeta_{(n)}
\left(~
\begin{ytableau}
 s_1 & \cdots & s_{n}
\end{ytableau} 
~\right) 
=
\,
\begin{ytableau}
 s_1 & \cdots & s_{n}
\end{ytableau} 
\,.
\]
 We first discuss a region where the series \eqref{def:SMZ} is absolutely convergent.

\begin{lemma}
\label{lem:convergence}
 Let
\[
 W_{\lambda}
=
\left\{{\pmb s}=(s_{ij})\in T(\lambda,\mathbb{C})\,\left|\,
\begin{array}{l}
 \text{$\Re(s_{ij})\ge 1$ for all $(i,j)\in D(\lambda) \setminus C(\lambda)$} \\[3pt]
 \text{$\Re(s_{ij})>1$ for all $(i,j)\in C(\lambda)$}
\end{array}
\right.
\right\}.
\]
 Then, the series \eqref{def:SMZ} converges absolutely if ${\pmb s}\in W_\lambda$.
\end{lemma} 
\begin{proof}
 Write $C(\lambda)=\{(i_1,j_1),\ldots,(i_k,j_k)\}$ where $i_1<\cdots<i_k$ and $j_1>\cdots>j_k$.
 Then, it can be written as $\lambda=(j_1^{i'_1} j_2^{i'_2}\cdots j_k^{i'_k})$
 where $i'_l=i_l-i_{l-1}$ with $i_0=0$.
 Since $\Re(s_{ij})\ge 1$ for $(i,j) \in D(\lambda) \setminus C(\lambda)$,
 we have 
\begin{align*}
 \sum_{M \in\mathrm{SSYT}(\lambda)}\left|\frac{1}{M^{\pmb s}}\right|
&\le
\prod^{k}_{l=1}\sum_{(m_{ij})\in \mathrm{SSYT}(j_l^{i'_l})}
\prod^{i_{l}}_{i=1}\prod^{j_l}_{j=1}\frac{1}{m^{\Re(s_{ij})}_{ij}}\\
&\le \prod^{k}_{l=1}\sum^{\infty}_{N_l=1}\frac{C_{i'_l,j_l}(N_l)}{N^{\Re(s_{i_l,j_l})}_{l}},
\end{align*}
 where $C_{a,b}(N)$ is a finite sum defined by 
\[
 C_{a,b}(N)
=\sum_{(m_{ij})\in \mathrm{SSYT}(b^a) \atop m_{a,b}=N}\underset{(i,j)\ne (a,b)}{\prod^{a}_{i=1}\prod^{b}_{j=1}}\frac{1}{m_{ij}}.
\] 
 It is well known that, for any $\varepsilon>0$,
 there exists a constant $C_{\varepsilon}>0$, which is not dependent on $N$,
 such that $\sum^{N}_{m=1}\frac{1}{m}<C_{\varepsilon}N^{\varepsilon}$.
 Hence
\[
 \left|C_{a,b}(N)\right|
\le \underset{(i,j)\ne (a,b)}{\prod^{a}_{i=1}\prod^{b}_{j=1}}\sum^{N}_{m_{ij}=1}\frac{1}{m_{ij}}
< C_{\varepsilon}^{ab-1}N^{\varepsilon(ab-1)}
\] 
 and therefore
\begin{align*}
 \sum_{M \in\mathrm{SSYT}(\lambda)}\left|\frac{1}{M^{\pmb s}}\right|
&\le \prod^{k}_{l=1}\sum^{\infty}_{N_l=1}\frac{C_{\varepsilon}^{i'_lj_l-1}N_l^{\varepsilon(i'_lj_l-1)}}{N^{\Re(s_{i_l,j_l})}_{l}}\\
&=\prod^{k}_{l=1} C_{\varepsilon}^{i'_lj_l-1}\zeta\left(\Re(s_{i_l,j_l})-\varepsilon(i'_lj_l-1)\right).
\end{align*}
 This ends the proof
 because $\Re(s_{i_l,j_l})>1$ for $1\le l\le k$ and $\varepsilon$ can be taken sufficiently small.
\end{proof}

\begin{rmk}
 The condition ${\pmb s}\in W_\lambda$ is a sufficient condition 
 that the series \eqref{def:SMZ} converges absolutely.
 It seems to be interesting to determine the region of absolute convergence of \eqref{def:SMZ} with full description.
 See e.g., \cite{Mat} for the cases of $\lambda=(1^n)$ and $(n)$, that is, the cases of MZFs and MZSFs.
\end{rmk}

 It should be noted that a SMZF can be also written as a linear combination of MZFs or MZSFs.  
 In fact, for $\lambda\vdash n$,
 let $\mathcal{F}(\lambda)$ be the set of all bijections $f:D(\lambda)\to\{1,2,\ldots,n\}$
 satisfying the following two conditions:
\begin{itemize}
\item[(i)]
 for all $i$, $f((i,j))<f((i,j'))$ if and only if $j<j'$, 
\item[(ii)]
 for all $j$, $f((i,j))<f((i',j))$ if and only if $i<i'$.
\end{itemize} 
 Moreover, for $T=(t_{ij})\in T(\lambda,X)$, 
 put
\[
 V(T)=
\left\{\left.
\left(t_{f^{-1}(1)},t_{f^{-1}(2)},\ldots,t_{f^{-1}(n)}\right)\in X^{n}\,\right|\,
f\in \mathcal{F}(\lambda)
\right\}.
\] 
 Furthermore, when $X$ has an addition $+$,
 we write ${\pmb w} \preceq T$ for ${\pmb w}=(w_1,w_2,\ldots,w_m)\in X^m$
 if there exists $(v_1,v_2,\ldots,v_{n})\in V(T)$ satisfying the following:
 for all $1\le k\le m$, there exist $1\le h_k\le m$ and $l_k\ge 0$ such that  
\begin{itemize}
\item[(i)]
 $w_k=v_{h_k}+v_{h_k+1}+\cdots +v_{h_k+l_k}$,
\item[(ii)]
 there are no $i$ and $i'$ such that $i\ne i'$ and $t_{ij},t_{i'j}\in\{v_{h_k},v_{h_k+1},\ldots ,v_{h_k+l_k}\}$ for some $j$,
\item[(iii)]
 $\bigsqcup^{m}_{k=1}\{h_k,h_k+1,\ldots,h_k+l_k\}=\{1,2,\ldots,n\}$.
\end{itemize}
 Then, by the definition, we have 
\begin{equation}
\label{for:SchurtoMZV}
 \zeta_{\lambda}({\pmb s})
=\sum_{{\pmb t} \,\preceq\, {\pmb s}}\zeta({\pmb t}).
\end{equation} 
 This clearly includes the first equation in \eqref{for:zsz} as the case $\lambda=(n)$.
 Moreover, by an inclusion-exclusion argument,
 one can also obtain its "dual" expression 
\begin{equation}
\label{for:SchurtoMZSV}
 \zeta_{\lambda}({\pmb s})
=\sum_{{\pmb t} \,\preceq\, {\pmb s}'}(-1)^{n-\ell({\pmb t})}\zeta^{\star}({\pmb t}),
\end{equation} 
 which does the second one in \eqref{for:zsz} as the case $\lambda=(1^n)$.

\begin{exam}
\begin{itemize}
\item[(1)]
 For ${\pmb s}=(s_{ij})\in T((3,1),\mathbb{C})$, we have 
\begin{align*}
 V({\pmb s})
&=\{
(s_{11},s_{12},s_{13},s_{21}),
(s_{11},s_{12},s_{21},s_{13}),
(s_{11},s_{21},s_{12},s_{13})
\}.
\end{align*}
 One sees that ${\pmb t} \preceq {\pmb s}$ if and only if ${\pmb t}$ is one of the following:
\begin{align*}
& (s_{11},s_{12},s_{13},s_{21}),(s_{11}+s_{12},s_{13},s_{21}),
 (s_{11},s_{12}+s_{13},s_{21}),(s_{11},s_{12},s_{13}+s_{21}),\\
& (s_{11}+s_{12}+s_{13},s_{21}),(s_{11}+s_{12},s_{13}+s_{21}),(s_{11},s_{12}+s_{13}+s_{21}),
 (s_{11},s_{12},s_{21},s_{13}),\\
& (s_{11}+s_{12},s_{21},s_{13}),(s_{11},s_{12}+s_{21},s_{13}),
 (s_{11},s_{21},s_{12},s_{13}),(s_{11},s_{21},s_{12}+s_{13}).
\end{align*} 
 This shows that when ${\pmb s}\in W_{(3,1)}$
\begin{align*}
\ytableausetup{boxsize=normal,aligntableaux=center}
 \ytableaushort{{s_{11}}{s_{12}}{s_{13}},{s_{21}}}\, 
&=\zeta(s_{11},s_{12},s_{13},s_{21})+\zeta(s_{11}+s_{12},s_{13},s_{21})+\zeta(s_{11},s_{12}+s_{13},s_{21})\\
&\ \ \ +\zeta(s_{11},s_{12},s_{13}+s_{21})+\zeta(s_{11}+s_{12}+s_{13},s_{21})+\zeta(s_{11}+s_{12},s_{13}+s_{21})\\
&\ \ \ +\zeta(s_{11},s_{12}+s_{13}+s_{21})+\zeta(s_{11},s_{12},s_{21},s_{13})+\zeta(s_{11}+s_{12},s_{21},s_{13})\\
&\ \ \ +\zeta(s_{11},s_{12}+s_{21},s_{13})+\zeta(s_{11},s_{21},s_{12},s_{13})+\zeta(s_{11},s_{21},s_{12}+s_{13})\\
&=\zeta^{\star}(s_{11},s_{21},s_{12},s_{13})-\zeta^{\star}(s_{11}+s_{21},s_{12},s_{13})-\zeta^{\star}(s_{11},s_{21}+s_{12},s_{13}), \\
&\ \ \ +\zeta^{\star}(s_{11},s_{12},s_{21},s_{13})-\zeta^{\star}(s_{11},s_{12},s_{21}+s_{13})+\zeta^{\star}(s_{11},s_{12},s_{13},s_{21}).
\end{align*} 
 Notice that the second equality follows from the discussion in (2).
\item[(2)]
 For ${\pmb s}=(s_{ij})\in T((2,1,1),\mathbb{C})$, we have 
\begin{align*}
 V({\pmb s})
&=\{
(s_{11},s_{12},s_{21},s_{31}),
(s_{11},s_{21},s_{12},s_{31}),
(s_{11},s_{21},s_{31},s_{12})
\}.
\end{align*}
 One sees that ${\pmb t} \preceq {\pmb s}$ if and only if ${\pmb t}$ is one of the followings:
\begin{align*}
& (s_{11},s_{12},s_{21},s_{31}),(s_{11}+s_{12},s_{21},s_{31}),(s_{11},s_{12}+s_{21},s_{31}), \\
& (s_{11},s_{21},s_{12},s_{31}),(s_{11},s_{21},s_{12}+s_{31}),(s_{11},s_{21},s_{31},s_{12}).
\end{align*} 
 This shows that when ${\pmb s}\in W_{(2,1,1)}$
\begin{align*}
\ytableausetup{boxsize=normal,aligntableaux=center}
 \ytableaushort{{s_{11}}{s_{12}},{s_{21}},{s_{31}}}\, 
&=\zeta(s_{11},s_{12},s_{21},s_{31})+\zeta(s_{11}+s_{12},s_{21},s_{31})+\zeta(s_{11},s_{12}+s_{21},s_{31}), \\
&\ \ \ +\zeta(s_{11},s_{21},s_{12},s_{31})+\zeta(s_{11},s_{21},s_{12}+s_{31})+\zeta(s_{11},s_{21},s_{31},s_{12})\\
&=\zeta^{\star}(s_{11},s_{21},s_{31},s_{12})-\zeta^{\star}(s_{11}+s_{21},s_{31},s_{12})-\zeta^{\star}(s_{11},s_{21}+s_{31},s_{12})\\
&\ \ \ -\zeta^{\star}(s_{11},s_{21},s_{31}+s_{12})+\zeta^{\star}(s_{11}+s_{21}+s_{31},s_{12})+\zeta^{\star}(s_{11}+s_{21},s_{31}+s_{12})\\
&\ \ \ +\zeta^{\star}(s_{11},s_{21}+s_{31}+s_{12})+\zeta^{\star}(s_{11},s_{21},s_{12},s_{31})-\zeta^{\star}(s_{11}+s_{21},s_{12},s_{31})\\
&\ \ \ -\zeta^{\star}(s_{11},s_{21}+s_{12},s_{31})+\zeta^{\star}(s_{11},s_{12},s_{21},s_{31})-\zeta^{\star}(s_{11},s_{12},s_{21}+s_{31}).
\end{align*} 
 Notice that the second equality follows from the discussion in (1).
\end{itemize}
\end{exam}
 
\begin{rmk}
 By the definitions,
 it is clear that if ${\pmb t}=(t_1,t_2,\ldots,t_m) \preceq {\pmb s}\in T(\lambda,\mathbb{C})$,
 then $t_m$ is expressed as a sum of $s_{ij}$ where at least one of $(i,j)$ is in $C(\lambda)$.
 This together with the expression \eqref{for:SchurtoMZV} or \eqref{for:SchurtoMZSV} also leads
 to Lemma~\ref{lem:convergence}. 
\end{rmk}

\subsection{A special case}
\label{subsec:special}

 We now consider a special case of variables; ${\pmb s}=\{s\}^{\lambda}$ ($s\in\mathbb{C}$)
 where $\{s\}^{\lambda}=(s_{ij})\in T(\lambda,\mathbb{C})$ is the tableau given by $s_{ij}=s$ for all $(i,j)\in D(\lambda)$.
 In this case, one sees that our SMZF is realized as a specialization of the Schur function.
 Actually, for variables ${\pmb x}=(x_1,x_2,\ldots)$, let
\[
 s_{\lambda}
=s_{\lambda}({\pmb x})
=\sum_{(m_{ij})\in\mathrm{SSYT}(\lambda)}\prod_{(i,j)\in D(\lambda)}x_{m_{ij}}
\]
 be the Schur function associated with $\lambda$.
 Then, for $s\in\mathbb{C}$ with $\Re(s)>1$, we have
\[
 \zeta_{\lambda}(\{s\}^{\lambda})
=e^{(s)} s_{\lambda}
=s_{\lambda}(1^{-s},2^{-s},\ldots),
\]
 where $e^{(s)}$ is the function sending $x_i$ to $\frac{1}{i^s}$. 
 This means that $\zeta_{\lambda}(\{s\}^{\lambda})$ can be written as a polynomial in $\zeta(s),\zeta(2s),\ldots$.

\begin{prop}
 Let $\lambda\vdash n$. 
 Then, for $s\in\mathbb{C}$ with $\Re(s)>1$, we have
\begin{equation}
\label{for:expnsionsbyp}
 \zeta_{\lambda}(\{s\}^{\lambda})
=\sum_{\mu\vdash n}\frac{\chi^{\lambda}(\mu)}{z_{\mu}}\prod^{\ell(\mu)}_{i=1}\zeta(\mu_i s).
\end{equation}
 Here, $z_{\mu}=\prod_{i\ge 1}i^{m_i(\mu)}m_i(\mu)!$ and $\chi^{\lambda}(\mu)\in\mathbb{Z}$
 is the value of the character $\chi^{\lambda}$ attached to the irreducible representation of the symmetric group $S_n$ of degree $n$
 corresponding to $\lambda$ on the conjugacy class of $S_n$ of the cycle type $\mu\vdash n$.
\end{prop}
\begin{proof}
 For a partition $\mu$,
 let $p_{\mu}=p_{\mu}({\pmb x})$ be the power-sum symmetric function defined by
 $p_{\mu}=\prod^{\ell(\mu)}_{i=1}p_{\mu_i}$ where $p_r=p_r({\pmb x})=\sum^{\infty}_{i=1}x_i^r$.
 We know that the Schur function can be written as
 a linear combination of power-sum symmetric functions (see \cite{Mac}) as 
\[
 s_{\lambda}=\sum_{\mu\vdash n}\frac{\chi^{\lambda}(\mu)}{z_{\mu}}p_{\mu}.
\]
 Hence, one obtains the desired expression by noticing $e^{(s)}p_r=p_r(1^{-s},2^{-s},\ldots)=\zeta(r s)$.
\end{proof}

\begin{rmk}
 For variables ${\pmb x}=(x_1,x_2,\ldots)$,
 let $e_{n}=e_{n}({\pmb x})$ and $h_{n}=h_{n}({\pmb x})$ be
 the elementary and complete symmetric functions of degree $n$,
 which are respectively defined by 
\[
 e_{n}=\sum_{i_1<\cdots<i_n}x_{i_1}\cdots x_{i_n}, \quad 
 h_{n}=\sum_{i_1\le \cdots\le i_n}x_{i_1}\cdots x_{i_n}.
\] 
 Since $s_{(1^n)}=e_{n}$ and $s_{(n)}=h_n$ with 
 $\chi^{(1^n)}(\mu)=|\mu|-\ell(\mu)$ and $\chi^{(n)}(\mu)=1$, 
 we have from \eqref{for:expnsionsbyp}  
\begin{align*}
 \zeta(s,\ldots,s)
&=e^{(s)}e_{n}
=e_{n}(1^{-s},2^{-s},\ldots)
=\sum_{\mu\vdash n}\frac{(-1)^{n-\ell(\mu)}}{z_{\mu}}\prod^{\ell(\mu)}_{i=1}\zeta(\mu_i s), \\
 \zeta^{\star}(s,\ldots,s)
&=e^{(s)}h_{n}
=h_n(1^{-s},2^{-s},\ldots)
=\sum_{\mu\vdash n}\frac{1}{z_{\mu}}\prod^{\ell(\mu)}_{i=1}\zeta(\mu_i s).
\end{align*}
 These expression are respectively implied from Theorem~2.2 and 2.1 in \cite{H1}.
\end{rmk}

 It is shown in e.g., \cite{H1,Za,M} 
 that $\zeta(2k,\ldots,2k),\zeta^{\star}(2k,\ldots,2k)\in\mathbb{Q}\pi^{2kn}$.
 These can be generalized to the Schur multiple zeta $\lq\lq$values" as follows.
 
\begin{cor}
 It holds that $\zeta_{\lambda}(\{2k\}^{\lambda})\in\mathbb{Q}\pi^{2k|\lambda|}$ for $k\in\mathbb{N}$.
\end{cor}
\begin{proof}
 This is a direct consequence of the expression \eqref{for:expnsionsbyp}
 together with the fact $\zeta(2k)\in\mathbb{Q}\pi^{2k}$ obtained by Euler
 (and hence the rational part can be explicitly written in terms of the Bernoulli numbers).
\end{proof}

\begin{exam}
\label{ex:n3}
 When $n=3$, we have 
\begin{align*}
\ytableausetup{boxsize=normal,aligntableaux=center}
 \ytableaushort{sss}\,
&=\frac{1}{6}\zeta(s)^3+\frac{1}{2}\zeta(2s)\zeta(s)+\frac{1}{3}\zeta(3s)=\zeta^{\star}(s,s,s),\\
 \ytableaushort{ss,s}\,
&=\frac{2}{6}\zeta(s)^3+\frac{0}{2}\zeta(2s)\zeta(s)+\frac{-1}{3}\zeta(3s),\\
 \ytableaushort{s,s,s}\,
&=\frac{1}{6}\zeta(s)^3+\frac{-1}{2}\zeta(2s)\zeta(s)+\frac{1}{3}\zeta(3s)=\zeta(s,s,s).
\end{align*}
 Special values of $\zeta_{\lambda}(\{2k\}^{\lambda})$ for $\lambda \vdash 3$ with small $k$ are given as follows:
\begin{center}
\ytableausetup{boxsize=11pt,aligntableaux=center}
\begin{tabular}{c||c|c|c|c}
 $\zeta_{\lambda}(\{2k\}^{\lambda})$ & $k=1$ & $k=2$ & $k=3$ & $k=4$ \\[5pt]
\hline
\hline 
\shortstack{\\[2pt] 
\begin{ytableau}
 \scriptstyle 2k & \scriptstyle 2k & \scriptstyle 2k 
\end{ytableau}
\\[0pt]}
 & \shortstack{\\[3pt] $\frac{31\pi^6}{15120}$ \\ \ }
 & \shortstack{\\[3pt] $\frac{4009\pi^{12}}{3405402000}$ \\ \ }
 & \shortstack{\\[3pt] $\frac{223199\pi^{18}}{194896477400625}$ \\ \ }
 & \shortstack{\\[3pt] $\frac{2278383389\pi^{24}}{1938427890852062610000}$ \\ \ } 
\\
\hline
\shortstack{\\[2pt]
\begin{ytableau}
 \scriptstyle 2k & \scriptstyle 2k \\
 \scriptstyle 2k 
\end{ytableau}
\\[0pt]}
 & \shortstack{\\[3pt] $\frac{\pi^6}{840}$ \\[5pt] \ }
 & \shortstack{\\[3pt] $\frac{493\pi^{12}}{5108103000}$ \\[5pt] \ }
 & \shortstack{\\[3pt] $\frac{86\pi^{18}}{4331032831125}$ \\[5pt] \ }
 & \shortstack{\\[3pt] $\frac{116120483\pi^{24}}{24230348635650782625000}$ \\[5pt] \ }  
\\
\hline
\shortstack{\\[2pt]
\begin{ytableau}
 \scriptstyle 2k \\
 \scriptstyle 2k \\
 \scriptstyle 2k 
\end{ytableau}
\\[0pt]} 
 & \shortstack{\\[3pt] $\frac{\pi^6}{5040}$ \\[10pt] \ }
 & \shortstack{\\[3pt] $\frac{\pi^{12}}{681080400}$ \\[10pt] \ }
 & \shortstack{\\[3pt] $\frac{2\pi^{18}}{64965492466875}$ \\[10pt] \ }
 & \shortstack{\\[3pt] $\frac{38081\pi^{24}}{48460697271301565250000}$ \\[10pt] \ } 
\end{tabular}
\end{center}
\end{exam}

\begin{exam}
\label{ex:n4}
 When $n=4$, we have 
\begin{align*}
\ytableausetup{boxsize=normal,aligntableaux=center}
 \ytableaushort{ssss}\,
&=\frac{1}{24}\zeta(s)^4+\frac{1}{4}\zeta(2s)\zeta(s)^2+\frac{1}{8}\zeta(2s)^2+\frac{1}{3}\zeta(3s)\zeta(s)+\frac{1}{4}\zeta(4s)=\zeta^{\star}(s,s,s,s),\\
 \ytableaushort{sss,s}\,
&=\frac{3}{24}\zeta(s)^4+\frac{1}{4}\zeta(2s)\zeta(s)^2+\frac{-1}{8}\zeta(2s)^2+\frac{0}{3}\zeta(3s)\zeta(s)-\frac{1}{4}\zeta(4s),\\
 \ytableaushort{ss,ss}\,
&=\frac{2}{24}\zeta(s)^4+\frac{0}{4}\zeta(2s)\zeta(s)^2+\frac{2}{8}\zeta(2s)^2+\frac{-1}{3}\zeta(3s)\zeta(s)+\frac{0}{4}\zeta(4s),\\
 \ytableaushort{ss,s,s}\,
&=\frac{3}{24}\zeta(s)^4+\frac{-1}{4}\zeta(2s)\zeta(s)^2+\frac{-1}{8}\zeta(2s)^2+\frac{0}{3}\zeta(3s)\zeta(s)+\frac{1}{4}\zeta(4s),\\
 \ytableaushort{s,s,s,s}\,
&=\frac{1}{24}\zeta(s)^4+\frac{-1}{4}\zeta(2s)\zeta(s)^2+\frac{1}{8}\zeta(2s)^2+\frac{1}{3}\zeta(3s)\zeta(s)+\frac{-1}{4}\zeta(4s)=\zeta(s,s,s,s).
\end{align*}
 Special values of $\zeta_{\lambda}(\{2k\}^{\lambda})$ for $\lambda\vdash 4$ with small $k$ are given as follows:
\begin{center}
\ytableausetup{boxsize=11pt,aligntableaux=center}
\begin{tabular}{c||c|c|c|c}
 $\zeta_{\lambda}(\{2k\}^{\lambda})$ & $k=1$ & $k=2$ & $k=3$ & $k=4$ \\[5pt]
\hline
\hline 
\shortstack{\\[2pt] 
\begin{ytableau}
 \scriptstyle 2k & \scriptstyle 2k & \scriptstyle 2k & \scriptstyle 2k
\end{ytableau}
\\[0pt]}
 & \shortstack{\\[3pt] $\frac{127\pi^8}{604800}$ \\ \ }
 & \shortstack{\\[3pt] $\frac{13739\pi^{16}}{1136785104000}$ \\ \ } 
 & \shortstack{\\[3pt] $\frac{1202645051\pi^{24}}{1009597859818782609375}$ \\ \ }
 & \shortstack{\\[3pt] $\frac{3467913415992313\pi^{32}}{27995618815818008860855350000000}$ \\ \ }  
\\
\hline
\shortstack{\\[2pt] 
\begin{ytableau}
 \scriptstyle 2k & \scriptstyle 2k & \scriptstyle 2k \\
 \scriptstyle 2k 
\end{ytableau}
\\[0pt]} 
 & \shortstack{\\[3pt] $\frac{239\pi^8}{1814400}$ \\[5pt] \ }
 & \shortstack{\\[3pt] $\frac{62191\pi^{16}}{62523180720000}$ \\[5pt] \ }
 & \shortstack{\\[3pt] $\frac{62572402\pi^{24}}{3028793579456347828125}$ \\[5pt] \ }
 & \shortstack{\\[3pt] $\frac{2019988202341\pi^{32}}{3999374116545429837265050000000}$ \\[5pt] \ }  
\\
\hline
\shortstack{\\[2pt] 
\begin{ytableau}
 \scriptstyle 2k & \scriptstyle 2k \\
 \scriptstyle 2k & \scriptstyle 2k 
\end{ytableau}
\\[0pt]} 
 & \shortstack{\\[3pt] $\frac{11\pi^8}{302400}$ \\[10pt] \ }
 & \shortstack{\\[3pt] $\frac{113\pi^{16}}{1838917080000}$ \\[10pt] \ }
 & \shortstack{\\[3pt] $\frac{14074\pi^{24}}{43895559122555765625}$ \\[10pt] \ }
 & \shortstack{\\[3pt] $\frac{30650383\pi^{32}}{15570422033269192914825000000}$ \\[10pt] \ }  
\\
\hline
\shortstack{\\[2pt] 
\begin{ytableau}
 \scriptstyle 2k & \scriptstyle 2k \\
 \scriptstyle 2k \\
 \scriptstyle 2k 
\end{ytableau}
\\[0pt]} 
 & \shortstack{\\[3pt] $\frac{11\pi^8}{362880}$ \\[15pt] \ }
 & \shortstack{\\[3pt] $\frac{29\pi^{16}}{1786376592000}$ \\[15pt] \ }
 & \shortstack{\\[3pt] $\frac{98642\pi^{24}}{3028793579456347828125}$ \\[15pt] \ }
 & \shortstack{\\[3pt] $\frac{332561213\pi^{32}}{3999374116545429837265050000000}$ \\[15pt] \ }  
\\
\hline
\shortstack{\\[2pt] 
\begin{ytableau}
 \scriptstyle 2k \\
 \scriptstyle 2k \\
 \scriptstyle 2k \\ 
 \scriptstyle 2k 
\end{ytableau}
\\[0pt]} 
 & \shortstack{\\[3pt] $\frac{\pi^8}{362880}$ \\[20pt] \ }
 & \shortstack{\\[3pt] $\frac{\pi^{16}}{12504636144000}$ \\[20pt] \ }
 & \shortstack{\\[3pt] $\frac{4\pi^{24}}{432684797065192546875}$ \\[20pt] \ }
 & \shortstack{\\[3pt] $\frac{13067\pi^{32}}{9331872938606002953618450000000}$ \\[20pt] \ }  
\end{tabular}
\end{center}
\end{exam} 


\section{Jacobi-Trudi formulas}

 The aim of this section is to give a proof of Theorem~\ref{JT}.
 To do that,
 we need some basic concepts in combinatorial method.
 Namely, we try to understand SMZF as a sum of weights of patterns on the $\mathbb{Z}^2$ lattice,
 similarly to Schur functions (more precisely, see, e.g., \cite{LP,HG,Ste,Zi}).
 Now, we do not work on SMZFs themselves, but with a truncated version of those,
 which may correspond to the Schur polynomial in theory of Schur functions.
 For $N\in \mathbb{N}$, let $\mathrm{SSYT}_N(\lambda)$ be the set of all $(m_{ij})\in \mathrm{SSYT}(\lambda)$ such that $m_{ij}\le N$ for all $i,j$.
 Define
\[
 \zeta^{N}_{\lambda}({\pmb s})
=\sum_{M\in\mathrm{SSYT}_N(\lambda)}\frac{1}{M^{{\pmb s}}}.
\]
 In particular, put
\[
\ytableausetup{boxsize=normal,aligntableaux=center}
 \zeta^{N}(s_1,\ldots,s_n)
=\zeta^{N}_{(1^n)}
\left(~ 
\begin{ytableau}
 s_1 \\
 \lower2pt\vdots \\
 s_{n}
\end{ytableau} 
~\right)
, \quad
 \zeta^{N\star}(s_1,\ldots,s_n)
=\zeta^N_{(n)}\left(~
\begin{ytableau}
 s_1 & \cdots & s_{n}
\end{ytableau} 
~\right) 
.
\]
 Notice that
 $\displaystyle{\lim_{N \to \infty} \zeta^N_{\lambda}({\pmb s})=\zeta_{\lambda}({\pmb s})}$ when ${\pmb s}\in W_\lambda$.
 Similarly to \eqref{for:zsz},
 we have the expressions
\begin{equation}
\label{for:zsNzN}
 \zeta^{N\star}({\pmb s})
=\sum_{{\pmb t} \,\preceq\, {\pmb s}}\zeta^{N}({\pmb t}), \quad 
 \zeta^{N}({\pmb s})
=\sum_{{\pmb t} \,\preceq\, {\pmb s}}(-1)^{n-\ell({\pmb t})}\zeta^{N\star}({\pmb t}).
\end{equation}

\subsection{A proof of the Jacobi-Trudi formula of $H$-type}

\subsubsection{Rim decomposition of partition}

 A skew partition is a pair of partitions $(\lambda,\mu)$
 satisfying $\mu\subset \lambda$, that is $\mu_i\le \lambda_i$ for all $i$.
 The resulting skew shape is denoted by $\lambda/\mu$ and
 the corresponding Young diagram is by $D(\lambda/\mu)$.
 We often identify $\lambda/\mu$ with $D(\lambda/\mu)$.
 A skew Young diagram $\theta$ is called a {\it ribbon}
 if $\theta$ is connected and contains no $2\times 2$ block of boxes.

 Let $\lambda$ be a partition.
 A sequence $\Theta=(\theta_1,\ldots,\theta_t)$ of ribbons is called a
 {\it rim decomposition of $\lambda$}
 if $D(\theta_k)\subset D(\lambda)$ for $1\le k\le t$
 and $\theta_1\sqcup\cdots\sqcup\theta_k$,
 the gluing of $\theta_1,\ldots,\theta_k$,
 is (the Young diagram of) a partition $\lambda^{(k)}$ for $1\le k\le t$
 satisfying $\lambda^{(t)}=\lambda$.
 One can naturally identify a rim decomposition $\Theta=(\theta_1,\ldots,\theta_t)$ of $\lambda$
 with the Young tableau $T=(t_{ij})\in T(\lambda,\{1,\ldots,t\})$
 defined by $t_{ij}=k$ if $(i,j)\in D(\theta_k)$. 
\begin{exam}
\label{ex:rim}
 The following $\Theta=(\theta_1,\theta_2,\theta_3,\theta_4)$ is a rim decomposition of $\lambda=(4,3,3,2)$;
\[
\ytableausetup{boxsize=normal,aligntableaux=center}
 \Theta
=\,
\begin{ytableau}
 1 & 1 & 3 & 3 \\
 2 & 3 & 3 \\
 2 & 3 & 4 \\
 3 & 3 
\end{ytableau}
,
\]
 which means that \!\!\!\!\!
\ytableausetup{mathmode,boxsize=10pt,aligntableaux=center}
 $\theta_1=\ydiagram{2}$\,,
 $\theta_2=\ydiagram{0,1,1}$\,,
 $\theta_3=\ydiagram{2+2,1+2,1+1,2}$\, and
 $\theta_4=\ydiagram{1}$\,.
\end{exam}

 Write $\lambda=(\lambda_1,\ldots,\lambda_r)$.
 We call a rim decomposition $\Theta=(\theta_1,\ldots,\theta_r)$ of $\lambda$ an {\it $H$-rim decomposition} 
 if each $\theta_i$ starts from $(i,1)$ for all $1\le i\le r$.
 Here, we permit $\theta_i=\emptyset$. 
 We denote by $\mathrm{Rim}^{\lambda}_H$ the set of all $H$-rim decompositions of $\lambda$.

\begin{exam}
\label{ex:Hrim}
 The following $\Theta=(\theta_1,\theta_2,\theta_3,\theta_4)$ is an $H$-rim decomposition of $\lambda=(4,3,3,2)$; 
\[
\ytableausetup{boxsize=normal,aligntableaux=center}
 \Theta
=\,
\begin{ytableau}
 1 & 1 & 3 & 3 \\
 3 & 3 & 3 \\
 3 & 4 & 4 \\
 4 & 4 
\end{ytableau}
,
\]
 which means that \!\!\!\!\!
\ytableausetup{mathmode,boxsize=10pt,aligntableaux=center} 
 $\theta_1=\ydiagram{2}$\,,
 $\theta_2=\emptyset$\,, 
 $\theta_3=\ydiagram{2+2,3,1}$\, and 
 $\theta_4=\ydiagram{1+2,2}$\,.
 Note that the rim decomposition appearing in Example~\ref{ex:rim} is not an $H$-rim decomposition.
\end{exam}

\begin{rmk}
 The $H$-rim decompositions also appeared in \cite{EggeaLoehrbWarringtonc2010},
 where they are called the {\it flat special rim-hooks}.
 They are used to compute the coefficients of the linear expansion of a given symmetric function via Schur functions.  
\end{rmk}

\subsubsection{Patterns on the $\mathbb{Z}^2$ lattice}

 Fix $N\in\mathbb{N}$.
 For a partition $\lambda=(\lambda_1,\ldots,\lambda_r)$,
 let $A_i$ and $B_i$ be lattice points in $\mathbb{Z}^2$
 respectively given by $A_i=(r+1-i,1)$ and $B_i=(r+1-i+\lambda_i,N)$ for $1\le i\le r$.
 Put $A=(A_1,\ldots,A_r)$ and $B=(B_1,\ldots,B_r)$.
 An {\it $H$-pattern} corresponding to $\lambda$ is a tuple $L=(l_1,\ldots,l_r)$ of directed paths on $\mathbb{Z}^2$,
 whose directions are allowed only to go one to the right or one up,
 such that $l_i$ starts from $A_i$ and ends to $B_{\sigma(i)}$ for some $\sigma\in S_r$.
 We call such $\sigma\in S_r$ the {\it type} of $L$ and denote it by $\sigma=\mathrm{type}(L)$.
 Note that the type of an $H$-pattern does not depend on $N$
 in the sense that the number of horizontal edges of each directed path of the $H$-pattern is independent of $N$.
 The number of horizontal edges appearing in the path $l_i$ is called the {\it horizontal distance} of $l_i$ and is denoted by $\mathrm{hd}(l_i)$.
 When $\mathrm{type}(L)=\sigma$,
 we simply write $L:A\to B^{\sigma}$ where $B^{\sigma}=(B_{\sigma(1)},\ldots,B_{\sigma(r)})$ and $l_i:A_i\to B_{\sigma(i)}$. 
 It is easy to see that $\mathrm{hd}(l_i)=\lambda_{\sigma(i)}-\sigma(i)+i$ and $\sum^{r}_{i=1}\mathrm{hd}(l_i)=|\lambda|$.
 
 Let $\mathcal{H}^{N}_{\lambda}$ be the set of all $H$-patterns corresponding to $\lambda$ and
 $S^{\lambda}_H=\{\mathrm{type}(L)\in S_r\,|\,L\in \mathcal{H}^{N}_{\lambda}\}$.
 The following is a key lemma of our study, which is easily verified.
 
\begin{lemma}
\label{lem:rimtypeH}
 For $\Theta=(\theta_1,\ldots,\theta_r) \in \mathrm{Rim}^{\lambda}_H$,
 there exists $L=(l_1,\ldots,l_r)\in \mathcal{H}^{N}_{\lambda}$ 
 such that $\mathrm{hd}(l_i)=|\theta_i|$ for all $1\le i\le r$.
 Moreover, the map $\tau_H:\mathrm{Rim}^{\lambda}_H\to S^{\lambda}_H$ given by $\tau_H(\Theta)=\mathrm{type}(L)$ is a bijection.
\end{lemma}

\begin{exam}
 Let $\lambda=(4,3,3,2)$.
 Then, we have $\tau_H(\Theta)=(1243)\in S_4$ 
 where $\Theta$ is the $H$-rim decomposition of $\lambda$ appearing in Example~\ref{ex:Hrim}.
\end{exam}

\subsubsection{Weight of patterns}

 Fix ${\pmb s}=(s_{ij})\in T(\lambda,\mathbb{C})$.
 We next assign a weight to $L=(l_1,\ldots,l_r) \in \mathcal{H}^{N}_{\lambda}$ via a $H$-rim decomposition of $\lambda$ as follows.
 Take $\Theta=(\theta_1,\ldots,\theta_r)\in\mathrm{Rim}^{\lambda}_{H}$ such that $\tau_H(\Theta)=\mathrm{type}(L)$.
 Then, when the $k$th horizontal edge of $l_i$ is on the $j$th row,
 we weight it with $\frac{1}{j^{s_{pq}}}$ where $(p,q)\in D(\lambda)$ is the $k$th component of $\theta_i$.
 Now, the weight $w^N_{{\pmb s}}(l_i)$ of the path $l_i$ is defined to be the product of weights of all horizontal edges along $l_i$.
 Here, we understand that $w^{N}_{{\pmb s}}(l_i)=1$ if $\theta_i=\emptyset$.
 Moreover, we define the weight $w^N_{{\pmb s}}(L)$ of $L\in \mathcal{H}^{N}_{\lambda}$ by 
\[
 w^N_{{\pmb s}}(L)=\prod^{r}_{i=1}w^{N}_{{\pmb s}}(l_i).
\]

\begin{exam}
 Let $\lambda=(4,3,3,2)$.
 Consider the following $L=(l_1,l_2,l_3,l_4)\in \mathcal{H}^{4}_{(4,3,3,2)}$;
\begin{figure}[h]
\begin{center}
 \includegraphics[clip,width=85mm]{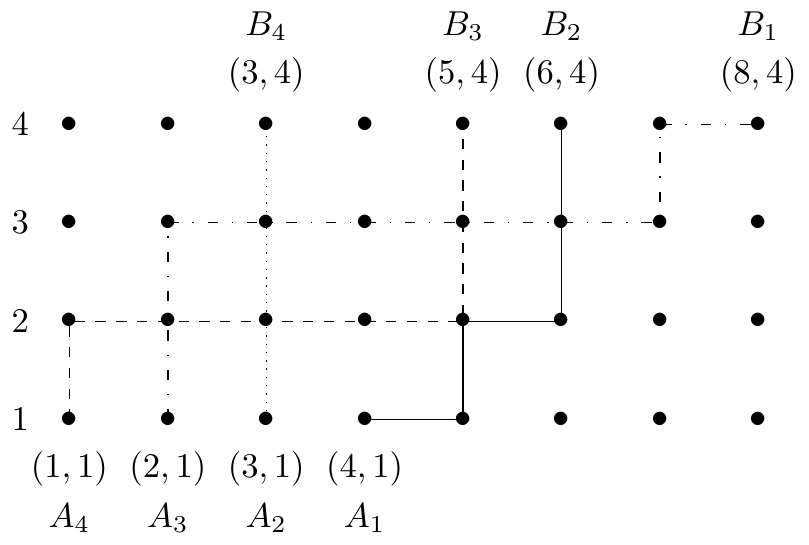}
\end{center}
\ \\[-40pt]
\caption{$L=(l_1,l_2,l_3,l_4)\in \mathcal{H}^{4}_{(4,3,3,2)}$}
\end{figure}

 Since $\mathrm{type}(L)=(1243)$,
 the corresponding $H$-rim decomposition of $\lambda$ is nothing but the one
 appearing in Example~\ref{ex:Hrim}. \\[-5pt]
 
 Let
$
\ytableausetup{boxsize=18pt,aligntableaux=center}
{\pmb s}=\,
\begin{ytableau}
 a & b & c & d \\
 e & f & g \\
 h & i & j \\
 k & l 
\end{ytableau}
\in T((4,3,3,2),\mathbb{C})
$. 
 Then, the weight of $l_i$ are given by  
\[
 w^{4}_{{\pmb s}}(l_1)
=\frac{1}{1^a2^b}, \quad
 w^{4}_{{\pmb s}}(l_2)
=1, \quad
 w^{4}_{{\pmb s}}(l_3)
=\frac{1}{3^h3^e3^f3^g3^c4^d}, \quad
 w^{4}_{{\pmb s}}(l_4)
=\frac{1}{2^k2^l2^i2^j}.
\]
 In particular, when 
$
\ytableausetup{boxsize=18pt,aligntableaux=center}
{\pmb s}=\,
\begin{ytableau}
 a_0 & a_1 & a_2 & a_3 \\
 a_{-1} & a_{0} & a_1 \\
 a_{-2} & a_{-1} & a_0 \\
 a_{-3} & a_{-2}
\end{ytableau}
\in T^{\mathrm{diag}}((4,3,3,2),\mathbb{C})$,
 these are equal to 
\[
 w^{4}_{{\pmb s}}(l_1)
=\frac{1}{1^{a_0}2^{a_1}}, \quad 
 w^{4}_{{\pmb s}}(l_2)
=1, \quad 
 w^{4}_{{\pmb s}}(l_3)
=\frac{1}{3^{a_{-2}}3^{a_{-1}}3^{a_0}3^{a_1}3^{a_2}4^{a_3}}, \quad 
 w^{4}_{{\pmb s}}(l_4)
=\frac{1}{2^{a_{-3}}2^{a_{-2}}2^{a_{-1}}2^{a_0}}.
\]
 Notice that, in this case, from the definition of the weight,
 the tuple of indices of the exponent of the denominator of $w^{4}_{{\pmb s}}(l_i)$ along $l_i$ should be equal to
 $(a_{1-i},a_{1-i+1},a_{1-i+2},\ldots)$ for all $i$. 
\end{exam}

\subsubsection{Proof}

 A proof of \eqref{Htype} is given by calculating the sum 
\[
 X^N_{\lambda}({\pmb s})
=\sum_{L\in \mathcal{H}^{N}_{\lambda}}\varepsilon_{\mathrm{type}(L)}w^{N}_{{\pmb s}}(L)
=\sum_{\sigma\in S^{\lambda}_H}
\varepsilon_{\sigma}
\sum_{L:A\to B^{\sigma}}
 w^{N}_{{\pmb s}}(L),
\]
 where $\varepsilon_{\sigma}$ is the signature of $\sigma\in S_r$.
 First, the inner sum can be calculated as follows.
 
\begin{lemma} 
\label{Hstep1}
 For $\sigma\in S^{\lambda}_H$,
 let $\Theta^{\sigma}=(\theta^{\sigma}_1,\ldots,\theta^{\sigma}_r)\in \mathrm{Rim}^{\lambda}_H$ be the $H$-rim decomposition
 such that $\tau_H(\Theta^{\sigma})=\sigma$.
 Then, we have  
\[ 
 \sum_{L:A\to B^{\sigma}}w^{N}_{{\pmb s}}(L)
=\prod^{r}_{i=1}\zeta^{N\star}\left(\theta^{\sigma}_i({\pmb s})\right).
\]
 Here, for $\Theta=(\theta_1,\ldots,\theta_r)\in\mathrm{Rim}^{\lambda}_H$, 
 $\theta_i({\pmb s})\in\mathbb{C}^{|\theta_i|}$ is the tuple
 obtained by reading contents of the shape restriction of ${\pmb s}$ to $\theta_i$ from the bottom left to the top right.
\end{lemma}
\begin{proof}
 Let $L=(l_1,\ldots,l_r)\in \mathcal{H}^{N}_{\lambda}$ be an $H$-pattern of type $\sigma$.
 Then $l_i$ is a path from $A_i$ to $B_{\sigma(i)}$ with $\mathrm{hd}(l_i)=\lambda_{\sigma(i)}-\sigma(i)+i=|\theta^{\sigma}_i|$.
 For simplicity, write $k_i=\lambda_{\sigma(i)}-\sigma(i)+i$ and $\theta^{\sigma}_i({\pmb s})=(s_{i,1},\ldots,s_{i,k_i})$.
 Suppose that $l_i$ has $n_j$ steps on the $j$th row for $1\le j\le N$.
 Then, from the definition of the weight, we have
\[
 w^{N}_{{\pmb s}}(l_i)
=\underbrace{\frac{1}{1^{s_{i,1}}}\cdots \frac{1}{1^{s_{i,n_1}}}}_{\text{$n_1$ terms}}
\underbrace{\frac{1}{2^{s_{i,n_1+1}}}\cdots \frac{1}{2^{s_{i,n_1+n_2}}}}_{\text{$n_2$ terms}}
\cdots 
\underbrace{\frac{1}{N^{s_{i,n_1+\cdots+n_{N-1}+1}}}\cdots \frac{1}{N^{s_{i,n_1+\cdots+n_{N}}}}}_{\text{$n_N$ terms}}
\] 
 with $n_1+\cdots+n_N=k_i$.
 This shows that  
\begin{align*}
 \sum_{L:A\to B^{\sigma}}w^{N}_{{\pmb s}}(L)
&=\prod^{r}_{i=1}\sum_{l_i:A_i\to B_{\sigma(i)}}w^{N}_{{\pmb s}}(l_i)\\
&=\prod^{r}_{i=1}\sum_{1\le m_1\le \cdots \le m_{k_i}\le N}\frac{1}{m_1^{s_{i,1}}\cdots m_{k_i}^{s_{i,k_i}}}\\
&=\prod^{r}_{i=1}\zeta^{N\star}(s_{i,1},\ldots,s_{i,k_i}).
\end{align*}
\end{proof}

 From Lemma~\ref{Hstep1}, we have 
\begin{equation}
\label{for:XN}
 X^N_{\lambda}({\pmb s})
=\sum_{\sigma\in S^{\lambda}_H}
\varepsilon_{\sigma}
 \prod^{r}_{i=1}\zeta^{N\star}\left(\theta^{\sigma}_i({\pmb s})\right).
\end{equation}
 Let $\mathcal{H}^{N}_{\lambda,0}$ be the set of all $L=(l_1,\ldots,l_r)\in \mathcal{H}^{N}_{\lambda}$ 
 such that any distinct pair of $l_i$ and $l_j$ has no intersection.
 Define 
\[ 
 X^{N}_{\lambda,0}({\pmb s})
=\sum_{L\in \mathcal{H}^{N}_{\lambda,0}}\varepsilon_{\mathrm{type}(L)}w^{N}_{{\pmb s}}(L), \quad
 X^{N}_{\lambda,1}({\pmb s})
=\sum_{L\in \mathcal{H}^{N}_{\lambda}\setminus \mathcal{H}^{N}_{\lambda,0}}\varepsilon_{\mathrm{type}(L)}w^{N}_{{\pmb s}}(L).
\]
 Clearly we have $X^N_{\lambda}({\pmb s})=X^N_{\lambda,0}({\pmb s})+X^N_{\lambda,1}({\pmb s})$.
 Moreover, since $\mathrm{type}(L)=\mathrm{id}$ for all $L\in \mathcal{H}^{N}_{\lambda,0}$
 where $\mathrm{id}$ is the identity element of $S_r$ and
 $\mathrm{id}$ corresponds to the trivial $H$-rim decomposition
 $(\theta_1,\ldots,\theta_r)=((\lambda_1),\ldots,(\lambda_r))$,
 employing the well-known identification between non-intersecting lattice paths and semi-standard Young tableaux
 (see, e.g., \cite{Sagan2001}),
 we have 
\[
 X^N_{\lambda,0}({\pmb s})
=\sum_{L\in \mathcal{H}^{N}_{\lambda,0}}w^{N}_{{\pmb s}}(L)
=\zeta^{N}_{\lambda}({\pmb s}).
\]
 Therefore, from \eqref{for:XN},
 we reach the expression  
\begin{align}
\label{for:Xkeyexpression}
 \zeta^{N}_{\lambda}({\pmb s})
&=\sum_{\sigma\in S^{\lambda}}
\varepsilon_{\sigma}
\prod^{r}_{i=1}
\zeta^{N\star}\left(\theta^{\sigma}_i({\pmb s})\right)
-X^N_{\lambda,1}({\pmb s}).
\end{align}

 Now, one can obtain \eqref{Htype} by taking the limit $N\to\infty$ of \eqref{for:Xkeyexpression}
 under suitable assumptions on ${\pmb s}$ described in Theorem~\ref{JT} together with the following proposition.

\begin{prop}
\label{prop:XX1}
 Assume that ${\pmb s}=(s_{ij})\in T^{\mathrm{diag}}(\lambda,\mathbb{C})$.
 Write $a_{k}=s_{i,i+k}$ for $k\in\mathbb{Z}$.
\begin{itemize}
\item[$(1)$] 
 We have 
\begin{equation}
\label{for:Xdet}
 X^{N}_{\lambda}({\pmb s})
=\det
\left[
 \zeta^{N\star}(a_{-j+1},a_{-j+2},\ldots,a_{-j+(\lambda_i-i+j)})
\right]_{1\le i,j\le r}.
\end{equation}
\item[$(2)$]
 It holds that
\begin{equation}
\label{for:X1}
 X^N_{\lambda,1}({\pmb s})=0.
\end{equation}
\end{itemize}
\end{prop}
\begin{proof}
 We first notice that,
 if ${\pmb s}=(s_{ij})\in T^{\mathrm{diag}}(\lambda,\mathbb{C})$,
 then we have
\[
 \theta^{\sigma}_i({\pmb s})
=(a_{1-i},a_{1-i+1},\ldots,a_{1-i+(\lambda_{\sigma(i)}-\sigma(i)+i)-1})
\]
 for all $1\le i\le r$.
 Therefore, understanding that $\zeta^{N}_{(k)}=0$ for $k<0$,
 from \eqref{for:XN}, we have
\begin{align*}
 X^{N}_{\lambda}({\pmb s})
&= \sum_{\sigma\in S^{\lambda}}\varepsilon_{\sigma}
\prod^{r}_{i=1}\zeta^{N\star}\left(\theta^{\sigma}_i({\pmb s})\right)\\
&=\sum_{\sigma\in S_r}
\varepsilon_{\sigma}
\prod^{r}_{i=1}
\zeta^{N\star}(a_{1-i},a_{1-i+1},\ldots,a_{1-i+(\lambda_{\sigma(i)}-\sigma(i)+i)-1})\\
&=\det
\left[
 \zeta^{N\star}(a_{1-i},a_{1-i+1},\ldots,a_{1-i+(\lambda_j-j+i)-1})
\right]_{1\le i,j\le r}\\
&=\det
\left[
 \zeta^{N\star}(a_{-i+1},a_{-j+2},\ldots,a_{-j+(\lambda_i-i+j)})
\right]_{1\le i,j\le r}.
\end{align*}
 Hence, we obtain \eqref{for:Xdet}. 

 We next show the second assertion.
 To do that,
 we employ the well-known involution $L\mapsto \overline{L}$ on $\mathcal{H}^{N}_{\lambda}\setminus \mathcal{H}^{N}_{\lambda,0}$ defined as follows.
 For $L=(l_1,\ldots,l_r)\in \mathcal{H}^{N}_{\lambda}\setminus \mathcal{H}^{N}_{\lambda,0}$ of type $\sigma$, 
 consider the first (rightmost) intersection point appearing in $L$,
 at which two paths say $l_i$ and $l_j$ cross.
 Let $\overline{L}$ be an $H$-pattern that contains every paths in $L$ except for $l_i$ and $l_j$ and
 two more paths $\overline{l_i}$ and $\overline{l_j}$.
 Here, $\overline{l_i}$ (resp. $\overline{l_j}$) follows $l_i$ (resp. $l_j$)
 until it meets the first intersection point and after that follows $l_j$ (resp. $l_i$) to the end. 
 Notice that, if ${\pmb s}=(s_{ij})\in T^{\mathrm{diag}}(\lambda,\mathbb{C})$,
 then we have $w^N_{{\pmb s}}(\overline{L})=w^N_{{\pmb s}}(L)$
 since there is no change of horizontal edges between $L$ and $\overline{L}$.
 Moreover, we have $\mathrm{type}(\overline{L})=-\mathrm{type}(L)$
 because the end points of $L$ and $\overline{L}$ are just switched.
 These imply that 
\begin{align*}
 X^N_{\lambda,1}({\pmb s})
&=\sum_{L\in \mathcal{H}^{N}_{\lambda}\setminus \mathcal{H}^{N}_{\lambda,0}}\varepsilon_{\mathrm{type}(\overline{L})}w^{N}_{{\pmb s}}(\overline{L})\\
&=-\sum_{L\in \mathcal{H}^{N}_{\lambda}\setminus \mathcal{H}^{N}_{\lambda,0}}\varepsilon_{\mathrm{type}(L)}w^{N}_{{\pmb s}}(L)\\
&=-X^N_{\lambda,1}({\pmb s})
\end{align*} 
 and therefore lead to \eqref{for:X1}.
\end{proof}

\begin{rmk}
 When ${\pmb s}\in T^{\mathrm{diag}}(\lambda,\mathbb{C})$,
 \eqref{for:Xkeyexpression} can be also written in terms of the $H$-rim decomposition as follows;
\begin{equation}
\label{for:keyexpressionHrim}
 \zeta^{N}_{\lambda}({\pmb s})
=\sum_{\Theta=(\theta_1,\theta_2,\ldots,\theta_r)\in \mathrm{Rim}^{\lambda}_H}
\varepsilon_H(\Theta)
\zeta^{N\star}\left(\theta_1({\pmb s})\right)\zeta^{N\star}\left(\theta_2({\pmb s})\right)\cdots \zeta^{N\star}\left(\theta_r({\pmb s})\right),
\end{equation}
 where $\varepsilon_H(\Theta)=\varepsilon_{\tau_H(\Theta)}$.
 Note that $\varepsilon(\Theta)=(-1)^{n-\#\{i\,|\,\theta_i\ne \emptyset\}}$ when $\lambda=(1^n)$. 
\end{rmk}

\begin{rmk}
\label{rmk:ErrorTerms}
 In some cases, 
 $X^{N}_{\lambda}({\pmb s})$ actually has a determinant expression without the assumption on variables;
\begin{align*}
\ytableausetup{boxsize=normal,aligntableaux=center}
 X^{N}_{(2,2)}
\left(~
\begin{ytableau}
 a & b \\
 c & d
\end{ytableau} 
~\right) 
&=\left|
\begin{array}{cc}
\zeta^{N\star}(a,b) & \zeta^{N\star}(c,d,b) \\
 \zeta^{N\star}(a) & \zeta^{N\star}(c,d)
\end{array}
\right|,\\[5pt]
 X^{N}_{(2,2,1)}
\left(~
\begin{ytableau}
 a & b \\
 c & d \\
 e
\end{ytableau} 
~\right) 
&=\left|
\begin{array}{ccc}
 \zeta^{N\star}(a,b) & \zeta^{N\star}(c,d,b) & \zeta^{N\star}(e,c,d,b) \\
 \zeta^{N\star}(a) & \zeta^{N\star}(c,d) & \zeta^{N\star}(e,c,d) \\
 0 & 1 & \zeta^{N\star}(e)
\end{array}
\right|.
\end{align*} 
 However, in general,
 $X^{N}_{\lambda}({\pmb s})$ can not be written as a determinant.
 For example, we have 
\begin{align*}
 X^{N}_{(2,2,2)}
\left(~
\begin{ytableau}
 a & b \\
 c & d \\
 e & f
\end{ytableau} 
~\right) 
&=\zeta^{N\star}(a,b)\zeta^{N\star}(c,d)\zeta^{N\star}(e,f)
-\zeta^{N\star}(a,b)\zeta^{N\star}(c)\zeta^{N\star}(e,f,d)\\
&\ \ \ -\zeta^{N\star}(c,a)\zeta^{N\star}(e,f,d,b) -\zeta^{N\star}(a)\zeta^{N\star}(c,d,b)\zeta^{N\star}(e,f)\\
&\ \ \ +\zeta^{N\star}(c,a,b)\zeta^{N\star}(e,f,d) +\zeta^{N\star}(a)\zeta^{N\star}(c)\zeta^{N\star}(e,f,d,b)
\end{align*}  
 and see that the righthand side does not seem to be expressed as a determinant
 (but is close to the determinant). 
 
 Similarly, $X^{N}_{\lambda,1}({\pmb s})$ does not vanish in general.
 For example,  
\begin{align*}
 X^{2}_{(2,2),1}
\left(~
\begin{ytableau}
 a & b \\
 c & d
\end{ytableau} 
~\right)
&=\left(\frac{1}{1^a1^b1^c1^d}+\frac{1}{1^a1^b1^c2^d}
+\frac{1}{1^a2^b1^c1^d}+\frac{1}{1^a2^b1^c2^d}\right.\\
&\ \ \ \ \ \ \left.\frac{1}{1^a2^b2^c2^d}+\frac{1}{2^a2^b1^c1^d}+\frac{1}{2^a2^b1^c2^d}+\frac{1}{2^a2^b2^c2^d}\right)\\
&\ \ \ -\left(\frac{1}{1^a1^b1^c1^d}+\frac{1}{1^a2^b1^c1^d}+\frac{1}{1^a2^b1^c2^d}+\frac{1}{1^a2^b2^c2^d}\right.\\
&\ \ \ \ \ \ \left.+\frac{1}{2^a1^b1^c1^d}+\frac{1}{2^a2^b1^c1^d}+\frac{1}{2^a2^b1^c2^d}+\frac{1}{2^a2^b2^c2^d}\right)\\
&=\frac{1}{1^a1^b1^c2^d}-\frac{1}{2^a1^b1^c1^d},
\end{align*}
 which actually vanishes when $a=d$.
\end{rmk}

\subsection{A proof of the Jacobi-Trudi formula of $E$-type}

 To prove \eqref{Etype},
 we need to consider another type of patterns on the $\mathbb{Z}^2$ lattice.
 Because the discussion is essentially the same as in the previous subsection,
 we omit all proofs of the results in this subsection.
 
 Let $\lambda=(\lambda_1,\ldots,\lambda_r)$ be a partition and $\lambda'=(\lambda'_1,\ldots,\lambda'_s)$ the conjugate of $\lambda$.
 A rim decomposition $\Theta=(\theta_1,\ldots,\theta_s)$ of $\lambda$ is called
 an {\it $E$-rim decomposition} if each $\theta_i$ starts from $(1,i)$ for all $1 \le i \le s$. 
 Here, we again permit $\theta_i=\emptyset$.
 We denote by $\mathrm{Rim}^{\lambda}_E$ the set of all $E$-rim decompositions of $\lambda$.

\begin{exam}
\label{ex:Erim}
 The following $\Theta=(\theta_1,\theta_2,\theta_3,\theta_4)$ is an $E$-rim decomposition of $\lambda=(4,3,3,2)$; 
\[
\ytableausetup{boxsize=normal,aligntableaux=center}
 \Theta
=\,
\begin{ytableau}
 1 & 2 & 3 & 4 \\
 1 & 2 & 3 \\
 1 & 3 & 3 \\
 3 & 3 
\end{ytableau}
,
\]
 which means that \!\!\!\!\!
\ytableausetup{mathmode,boxsize=10pt,aligntableaux=center} 
 $\theta_1=\ydiagram{1,1,1}$\,,
 $\theta_2=\ydiagram{1,1}$\,, 
 $\theta_3=\ydiagram{2+1,2+1,1+2,2}$\, and 
 $\theta_4=\ydiagram{1}$\,.
\end{exam}

 Fix $N\in\mathbb{N}$.
 Let $C_i$ and $D_i$ be lattice points in $\mathbb{Z}^2$
 respectively given by $C_i=(s+1-i,1)$ and $D_i=(s+1-i+\lambda'_i,N+1)$ for $1\le i\le s$.
 Put $C=(C_1,\ldots,C_s)$ and $D=(D_1,\ldots,D_s)$.
 An {\it $E$-pattern} corresponding to $\lambda$ is a tuple $L=(l_1,\ldots,l_s)$ of directed paths on $\mathbb{Z}^2$,
 whose directions are allowed only to go one to the northeast or one up,
 such that $l_i$ starts from $C_i$ and ends to $D_{\sigma(i)}$ for some $\sigma\in S_s$.
 We also call such $\sigma\in S_s$ the {\it type} of $L$ and denote it by $\sigma=\mathrm{type}(L)$.
 The number of northeast edges appearing in the path $l_i$ is called the {\it northeast distance} of $l_i$ and is denoted by $\mathrm{ned}(l_i)$.
 When $\mathrm{type}(L)=\sigma$,
 we simply write $L:C\to D^{\sigma}$ where $D^{\sigma}=(D_{\sigma(1)},\ldots,D_{\sigma(s)})$ and $l_i:C_i\to D_{\sigma(i)}$. 
 It is easy to see that $\mathrm{ned}(l_i)=\lambda'_{\sigma(i)}-\sigma(i)+i$ and $\sum^{s}_{i=1}\mathrm{ned}(l_i)=|\lambda|$.
 
 Let $\mathcal{E}^{N}_{\lambda}$ be the set of all $E$-patterns corresponding to $\lambda$
 and $S^{\lambda}_E=\{\mathrm{type}(L)\in S_s\,|\,L\in \mathcal{E}^{N}_{\lambda}\}$.

\begin{lemma}
\label{lem:rimtypeE}
 For $\Theta=(\theta_1,\ldots,\theta_s)\in \mathrm{Rim}^{\lambda}_E$,
 there exists $L=(l_1,\ldots,l_s)\in \mathcal{E}^{N}_{\lambda}$ 
 such that $\mathrm{ned}(l_i)=|\theta_i|$ for all $1\le i\le s$.
 Moreover, the map $\tau_E:\mathrm{Rim}^{\lambda}_E\to S^{\lambda}_E$ given by $\tau_E(\Theta)=\mathrm{type}(L)$ is a bijection.
\end{lemma}

 Fix ${\pmb s}=(s_{ij})\in T(\lambda,\mathbb{C})$.
 A weight on $L=(l_1,\ldots,l_s) \in \mathcal{E}^{N}_{\lambda}$ is similarly defined via the $E$-rim decomposition of $\lambda$ as follows.
 Take $\Theta=(\theta_1,\ldots,\theta_s)\in\mathrm{Rim}^{\lambda}_{E}$ such that $\tau_E(\Theta)=\mathrm{type}(L)$.
 Then, when the $k$th northeast edge of $l_i$ lies from the $j$th row to $(j+1)$th row,
 we weight it with $\frac{1}{j^{s_{pq}}}$ where $(p,q)\in D(\lambda)$ is the $k$th component of $\theta_i$.
 Now, the weight $w^N_{{\pmb s}}(l_i)$ of the path $l_i$ is defined to be the product of weights of all northeast edges along $l_i$.
 Here, we understand that $w^{N}_{{\pmb s}}(l_i)=1$ if $\theta_i=\emptyset$.
 Moreover, we define the weight $w^N_{{\pmb s}}(L)$ of $L\in \mathcal{E}^{N}_{\lambda}$ by 
\[
 w^N_{{\pmb s}}(L)=\prod^{s}_{i=1}w^{N}_{{\pmb s}}(l_i).
\]

\begin{exam}
 Let $\lambda=(4,3,3,2)$.
 Consider the $E$-rim decomposition $\Theta\in\mathrm{Rim}^{\lambda}_E$ of $\lambda$ appeared in Example~\ref{ex:Erim}.
 It is easy to see that $\tau_E(\Theta)=(123)\in S_4$ via the following $L=(l_1,l_2,l_3,l_4)\in \mathcal{E}^{6}_{(4,3,3,2)}$;
\begin{figure}[h]
\begin{center}
 \includegraphics[clip,width=85mm]{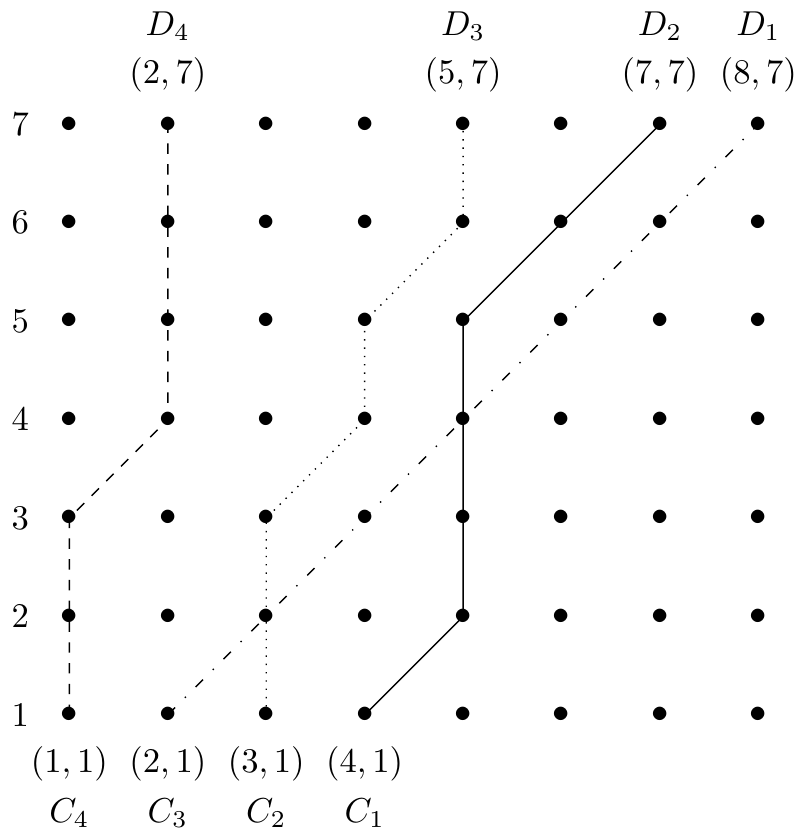}
\end{center}
\ \\[-40pt]
\caption{$L=(l_1,l_2,l_3,l_4)\in \mathcal{E}^{6}_{(4,3,3,2)}$}
\end{figure}

 Let
$
\ytableausetup{boxsize=18pt,aligntableaux=center}
{\pmb s}=\,
\begin{ytableau}
 a & b & c & d \\
 e & f & g \\
 h & i & j \\
 k & l 
\end{ytableau}
\in T((4,3,3,2),\mathbb{C})
$. 
 Then, the weight of $l_i$ are given by  
\[
 w^{4}_{{\pmb s}}(l_1)
=\frac{1}{1^a5^e6^h}, \quad
 w^{4}_{{\pmb s}}(l_2)
=\frac{1}{3^b5^f}, \quad
 w^{4}_{{\pmb s}}(l_3)
=\frac{1}{1^c2^g3^j4^i5^l6^k}, \quad
 w^{4}_{{\pmb s}}(l_4)
=\frac{1}{3^d}.
\]
 In particular, when 
$
\ytableausetup{boxsize=18pt,aligntableaux=center}
{\pmb s}=\,
\begin{ytableau}
 a_0 & a_1 & a_2 & a_3 \\
 a_{-1} & a_{0} & a_1 \\
 a_{-2} & a_{-1} & a_0 \\
 a_{-3} & a_{-2}
\end{ytableau}
\in T^{\mathrm{diag}}((4,3,3,2),\mathbb{C})$,
 these are equal to 
\[
 w^{4}_{{\pmb s}}(l_1)
=\frac{1}{1^{a_0}5^{a_{-1}}6^{a_{-2}}}, \quad
 w^{4}_{{\pmb s}}(l_2)
=\frac{1}{3^{a_1}5^{a_0}}, \quad
 w^{4}_{{\pmb s}}(l_3)
=\frac{1}{1^{a_2}2^{a_1}3^{a_0}4^{a_{-1}}5^{a_{-2}}6^{a_{-3}}}, \quad
 w^{4}_{{\pmb s}}(l_4)
=\frac{1}{3^{a_3}}.
\]
 Notice that, in this case, from the definition of the weight,
 the tuple of indexes of the exponent of the denominator of $w^{4}_{{\pmb s}}(l_i)$ along $l_i$ should be equal to
 $(a_{-1+i},a_{-1+i-1},a_{-1+i-2},\ldots)$ for all $i$. 
\end{exam}

 We similarly give a proof of \eqref{Etype} by calculating the sum 
\[
 Y^N_{\lambda}({\pmb s})
=\sum_{L\in \mathcal{E}^{N}_{\lambda}}\varepsilon_{\mathrm{type}(L)}w^{N}_{{\pmb s}}(L)
=\sum_{\sigma\in S^{\lambda}_E}
\varepsilon_{\sigma}
\sum_{L:C\to D^{\sigma}}
w^{N}_{{\pmb s}}(L),
\]

\begin{lemma} 
\label{Estep1}
 For $\sigma\in S^{\lambda}_E$,
 let $\Theta^{\sigma}=(\theta^{\sigma}_1,\ldots,\theta^{\sigma}_s)\in \mathrm{Rim}^{\lambda}_E$ be the $E$-rim decomposition
 such that $\tau_E(\Theta^{\sigma})=\sigma$.
 Then, we have  
\[ 
 \sum_{L:C\to D^{\sigma}}w^{N}_{{\pmb s}}(L)
=\prod^{s}_{i=1}\zeta^{N}\left(\theta^{\sigma}_i({\pmb s})\right).
\]
 Here, for $\Theta=(\theta_1,\ldots,\theta_s)\in\mathrm{Rim}^{\lambda}_E$, 
 $\theta_i({\pmb s})\in\mathbb{C}^{|\theta_i|}$ is the tuple
 obtained by reading contents of the shape restriction of ${\pmb s}$ to $\theta_i$ from the top right to the bottom left.
\end{lemma}

 From Lemma~\ref{Estep1}, we have 
\begin{equation}
\label{for:YN}
 Y^N_{\lambda}({\pmb s})
=\sum_{\sigma\in S^{\lambda}_E}
\varepsilon_{\sigma}
 \prod^{s}_{i=1}\zeta^{N}\left(\theta^{\sigma}_i({\pmb s})\right).
\end{equation}
 Define $\mathcal{E}^{N}_{\lambda,0}$ similarly to $\mathcal{H}^{N}_{\lambda,0}$
 and also $ Y^N_{\lambda,0}({\pmb s})$ and $Y^N_{\lambda,1}({\pmb s})$.
 It holds that 
\[
 Y^N_{\lambda,0}({\pmb s})
=\sum_{L\in \mathcal{E}^{N}_{\lambda,0}}w^{N}_{{\pmb s}}(L)
=\zeta^{N}_{\lambda}({\pmb s}).
\]
 Hence, from \eqref{for:YN},
 we reach the expression  
\begin{align}
\label{for:Ykeyexpression}
 \zeta^{N}_{\lambda}({\pmb s})
&=\sum_{\sigma\in S^{\lambda}_E}
\varepsilon_{\sigma}
\prod^{s}_{i=1}
\zeta^{N}\left(\theta^{\sigma}_i({\pmb s})\right)
-Y^N_{\lambda,1}({\pmb s}).
\end{align}

 Now, \eqref{Etype} is obtained by taking the limit $N\to\infty$ of \eqref{for:Ykeyexpression}
 under suitable assumptions on ${\pmb s}$ described in Theorem~\ref{JT} together with the following proposition.

\begin{prop}
\label{prop:YY1}
 Assume that ${\pmb s}=(s_{ij})\in T^{\mathrm{diag}}(\lambda,\mathbb{C})$.
 Write $a_{k}=s_{i,i+k}$ for $k\in\mathbb{Z}$.
\begin{itemize}
\item[$(1)$] 
 We have 
\[
 Y^{N}_{\lambda}({\pmb s})
=\det
\left[
 \zeta^{N}(a_{j-1},a_{j-2},\ldots,a_{j-(\lambda'_i-i+j)})
\right]_{1\le i,j\le s}.
\]
\item[$(2)$]
 It holds that
\[
 Y^N_{\lambda,1}({\pmb s})=0.
\]
\end{itemize}
\end{prop}

\begin{rmk}
 When ${\pmb s}\in T^{\mathrm{diag}}(\lambda,\mathbb{C})$,
 \eqref{for:Ykeyexpression} can be also written in terms of the $E$-rim decomposition as follows;
\begin{equation}
\label{for:keyexpressionErim}
 \zeta^{N}_{\lambda}({\pmb s})
=\sum_{\Theta=(\theta_1,\theta_2,\ldots,\theta_s)\in \mathrm{Rim}^{\lambda}_E}
\varepsilon_E(\Theta)
\zeta^{N}\left(\theta_1({\pmb s})\right)\zeta^{N}\left(\theta_2({\pmb s})\right)\cdots \zeta^{N}\left(\theta_s({\pmb s})\right),
\end{equation}
 where $\varepsilon_E(\Theta)=\varepsilon_{\tau_E(\Theta)}$.
 Note that $\varepsilon_E(\Theta)=(-1)^{n-\#\{i\,|\,\theta_i\ne \emptyset\}}$ when $\lambda=(n)$. 
\end{rmk}


\section{Schur multiple zeta functions as variations of Schur functions}
\label{sec:Macdonald}

\subsection{Schur multiple zeta functions of skew type}

 Our SMZFs are naturally extended to those of skew type as follows.
 Let $\lambda$ and $\mu$ be partitions satisfying $\mu\subset \lambda$.
 We use the same notations $T(\lambda/\mu,X),T^{\mathrm{diag}}(\lambda/\mu,X)$ for a set $X$,
 $\mathrm{SSYT}(\lambda/\mu)$ and $\mathrm{SSYT}_N(\lambda/\mu)$ for a positive integer $N\in\mathbb{N}$
 as the previous sections.

 Let ${\pmb s}=(s_{ij})\in T(\lambda/\mu,\mathbb{C})$.
 We define a skew SMZF associated with $\lambda/\mu$ by
\begin{equation}
\label{def:skewSMZ}
 \zeta_{\lambda/\mu}({\pmb s})
=\sum_{M\in \mathrm{SSYT}(\lambda/\mu)}\frac{1}{M^{\pmb s}}
\end{equation}
 and its truncated sum   
\[
 \zeta^{N}_{\lambda/\mu}({\pmb s})
=\sum_{M\in \mathrm{SSYT}_N(\lambda/\mu)}\frac{1}{M^{\pmb s}},
\]
 where $M^{\pmb s}=\prod_{(i,j)\in D(\lambda/\mu)}m_{ij}^{s_{ij}}$ for $M=(m_{ij})\in\mathrm{SSYT}(\lambda/\mu)$.
 As we have seen in Lemma~\ref{lem:convergence}, 
 the series \eqref{def:skewSMZ} converges absolutely if ${\pmb s}\in W_{\lambda/\mu}$ where $W_{\lambda/\mu}$ is also similarly defined as $W_{\lambda}$
 (note that $C(\lambda/\mu)\subset C(\lambda)$).
 We have again  
\begin{equation}
\label{for:SchurtoMZVskew}
 \zeta_{\lambda/\mu}({\pmb s})
=\sum_{{\pmb t} \,\preceq\, {\pmb s}}\zeta({\pmb t}), \quad
 \zeta_{\lambda/\mu}({\pmb s})
=\sum_{{\pmb t} \,\preceq\, {\pmb s}'}(-1)^{|\lambda/\mu|-\ell({\pmb t})}\zeta^{\star}({\pmb t}),
\end{equation} 
 where $\preceq$ is naturally generalized to the skew types.

\begin{exam}
\label{ex:skewSMZ}
\begin{itemize}
\item[(1)]
 For ${\pmb s}=(s_{ij})\in W_{(2,2,2)/(1,1)}$, we have 
\begin{align*}
\ytableausetup{boxsize=normal,aligntableaux=center}
\begin{ytableau}
 \none & s_{12} \\
 \none & s_{22} \\
 s_{31} & s_{32}
\end{ytableau}
\, 
&=\zeta(s_{31},s_{12},s_{22},s_{32})+\zeta(s_{31}+s_{12},s_{22},s_{32})+\zeta(s_{12},s_{31}+s_{22},s_{32}), \\
&\ \ \ +\zeta(s_{12},s_{31},s_{22},s_{32})+\zeta(s_{12},s_{22},s_{31}+s_{32})+\zeta(s_{12},s_{22},s_{31},s_{32})\\
&=\zeta^{\star}(s_{31},s_{12},s_{22},s_{32})-\zeta^{\star}(s_{31}+s_{12},s_{22},s_{32})-\zeta^{\star}(s_{31},s_{12}+s_{22},s_{32})\\
&\ \ \ -\zeta^{\star}(s_{31},s_{12},s_{22}+s_{32})+\zeta^{\star}(s_{31}+s_{12}+s_{22},s_{32})+\zeta^{\star}(s_{31}+s_{12},s_{22}+s_{32})\\
&\ \ \ +\zeta^{\star}(s_{31},s_{12}+s_{22}+s_{32})+\zeta^{\star}(s_{12},s_{31},s_{22},s_{32})-\zeta^{\star}(s_{12},s_{31},s_{22}+s_{32})\\
&\ \ \ +\zeta^{\star}(s_{12},s_{22},s_{31},s_{32})-\zeta^{\star}(s_{12}+s_{22},s_{31},s_{32})-\zeta^{\star}(s_{12},s_{22}+s_{31},s_{32}).
\end{align*} 
\item[(2)]
 For ${\pmb s}=(s_{ij})\in W_{(3,3)/(2)}$, we have 
\begin{align*}
\ytableausetup{boxsize=normal,aligntableaux=center}
\begin{ytableau}
 \none & \none & s_{13} \\
 s_{21} & s_{22} & s_{23} 
\end{ytableau}
\,
&=\zeta(s_{13},s_{21},s_{22},s_{23})+\zeta(s_{13}+s_{21},s_{22},s_{23})+\zeta(s_{13},s_{21}+s_{22},s_{23})\\
&\ \ \ +\zeta(s_{13},s_{21},s_{22}+s_{23})+\zeta(s_{13}+s_{21}+s_{22},s_{23})+\zeta(s_{13}+s_{21},s_{22}+s_{23})\\
&\ \ \ +\zeta(s_{13},s_{21}+s_{22}+s_{23})+\zeta(s_{21},s_{13},s_{22},s_{23})+\zeta(s_{21},s_{13}+s_{22},s_{23})\\
&\ \ \ +\zeta(s_{21},s_{13},s_{22}+s_{23})+\zeta(s_{21},s_{22},s_{13},s_{23})+\zeta(s_{21}+s_{22},s_{13},s_{23})\\
&=\zeta^{\star}(s_{13},s_{21},s_{22},s_{23})-\zeta^{\star}(s_{13}+s_{21},s_{22},s_{23})-\zeta^{\star}(s_{21},s_{13}+s_{22},s_{23}), \\
&\ \ \ +\zeta^{\star}(s_{21},s_{13},s_{22},s_{23})-\zeta^{\star}(s_{21},s_{22},s_{23}+s_{13})+\zeta^{\star}(s_{21},s_{22},s_{13},s_{23}).
\end{align*} 
\end{itemize}
\end{exam}

 As in Section~\ref{subsec:special},
 one sees that $\zeta_{\lambda/\mu}(\{s\}^{\lambda/\mu})=e^{(s)}s_{\lambda/\mu}=s_{\lambda/\mu}(1^{-s},2^{-s},\ldots)$ for $s\in\mathbb{C}$ with $\Re(s)>1$
 where $s_{\lambda/\mu}$ is the skew Schur function associated with $\lambda/\mu$ (see \cite{Mac}).
 In particular, since $s_{\lambda/\mu}$ is a symmetric function and hence can be expressed as a linear combination of the power-sum symmetric functions,
 we have $\zeta_{\lambda/\mu}(\{2k\}^{\lambda/\mu})\in\mathbb{Q}\pi^{2k(|\lambda|-|\mu|)}$ for $k\in\mathbb{N}$.
 Notice that it is shown in \cite{Sta} that
 $\zeta_{\lambda/\mu}(\{2k\}^{\lambda/\mu})$ for a special choice of $\lambda/\mu$ with $k=1,2,3$
 is involved with $f^{\lambda/\mu}$, the number of standard Young tableaux of shape $\lambda/\mu$.

\subsection{Macdonald's ninth variation of Schur functions}

 Let $W^{\mathrm{diag}}_{\lambda/\mu}=W_{\lambda/\mu}\cap T^{\mathrm{diag}}(\lambda/\mu,\mathbb{C})$.
 We now show that, when ${\pmb s}\in W^{\mathrm{diag}}_{\lambda/\mu}$,
 the skew SMZF $\zeta_{\lambda/\mu}({\pmb s})$ is realized
 as (the limit of) a specialization of the ninth variation of skew Schur functions
 studied by Nakagawa, Noumi, Shirakawa and Yamada \cite{NNSY}.
 As in the previous discussion,
 we write $a_{k}=s_{i,i+k}$ for $k\in\mathbb{Z}$ (and for any $i\in\mathbb{N}$)
 for ${\pmb s}=(s_{ij})\in W^{\mathrm{diag}}_{\lambda/\mu}$.

 Let $r$ and $s$ be positive integers.
 Put $\eta=r+s$. Let $\lambda=(\lambda_1,\ldots,\lambda_r)$ and $\mu=(\mu_1,\ldots,\mu_r)$
 be partitions satisfying $\mu\subset \lambda\subset (s^r)$ (we here allow $\lambda_i=0$ or $\mu_i=0$) and
 $J=\{j_1,j_2,\ldots,j_r\}$ with $j_{a}=\lambda_{r+1-a}+a$ and
 $I=\{i_1,i_2,\ldots,i_r\}$ with $i_{b}=\mu_{r+1-b}+b$
 the corresponding Maya diagrams, respectively.
 Notice that $I$ and $J$ are subsets of $\{1,2,\ldots,\eta\}$
 satisfying $j_1<j_2<\cdots<j_r$ and $i_1<i_2<\cdots<i_r$. 
 Then, Macdonald's ninth variation of skew Schur function $S^{(r)}_{\lambda/\mu}(X)$ 
 associated with a general matrix $X=[x_{ij}]_{1\le i,j\le \eta}$ of size $\eta$ is defined by
\[
 S^{(r)}_{\lambda/\mu}(X)=\xi^{I}_{J}(X_{+}).
\]
 Here, we have used the Gauss decomposition $X=X_{-}X_{0}X_{+}$ of $X$
 where $X_{-},X_0$ and $X_{+}$ are lower unitriangular, diagonal and upper unitriangular matrices, respectively,
 which are determined uniquely as matrices with entries in the field of rational functions
 in the variables $x_{ij}$ for $1\le i,j\le \eta$.
 Moreover,
 $\xi^{I}_{J}(X_{+})$ is the minor determinant of $X_{+}$ corresponding to $I$ and $J$.
 Put 
\begin{align*}
 e^{(r)}_n(X)
&=S^{(r)}_{(1^n)}(X)=\xi^{1,\ldots,r}_{1,\ldots,r-1,r+n}(X_{+}),\\
 h^{(r)}_n(X)
&=S^{(r)}_{(n)}(X)=\xi^{1,\ldots,r}_{1,\ldots,\widehat{r-n+1},r+1}(X_{+}),
\end{align*}
 which are variations of the elementary and complete symmetric polynomials, respectively.
 Here $\widehat{r-n+1}$ means that we ignore $r-n+1$.
 For convenience, we put $e^{(r)}_0(X)=h^{(r)}_0(X)=1$ and $e^{(r)}_n(X)=h^{(r)}_n(X)=0$ for $n<0$.
 
 For $N\in\mathbb{N}$,
 let $U=U^{(N)}$ be an upper unitriangular matrix of size $\eta$ defined by $U=U_1U_2\cdots U_N$
 where  
\[
 U_k
=\left(I_{\eta}+u^{(1)}_k E_{12}\right)\left(I_{\eta}+u^{(2)}_k E_{23}\right)
\cdots \left(I_{\eta}+u^{(\eta-1)}_k E_{\eta-1,\eta}\right).
\]
 Here, $u^{(i)}_k$ are variables for $1\le k\le N$ and $1\le i\le \eta-1$ and
 $I_{\eta}$ and $E_{ij}$ are the identity and unit matrix of size $\eta$, respectively. 
 The following is crucial in this section.
 
\begin{lemma}
 Let ${\pmb s}=(s_{ij})\in T^{\mathrm{diag}}(\lambda/\mu,\mathbb{C})$.
 Write $a_{k}=s_{i,i+k}$ for $k\in\mathbb{Z}$.
 If $u^{(i)}_k=k^{-a_{i-r}}$, then we have 
\begin{equation}
\label{for:sMZVisS}
 \zeta^{N}_{\lambda/\mu}({\pmb s})=S^{(r)}_{\lambda/\mu}(U).
\end{equation}
\end{lemma}
\begin{proof}
 It is shown in \cite{NNSY} that $S^{(r)}_{\lambda/\mu}(U)$ has a tableau representation
\begin{equation}
\label{for:TableauRepresentation}
 S^{(r)}_{\lambda/\mu}(U)
=\sum_{(m_{ij})\in \mathrm{SSYT}_N(\lambda/\mu)}\prod_{(i,j)\in D(\lambda/\mu)}u^{(r-i+j)}_{m_{ij}}.
\end{equation}
 Hence the claim immediately follows because $u^{(r-i+j)}_{m_{ij}}=m^{-a_{j-i}}_{ij}=m^{-s_{ij}}_{ij}$
 if $u^{(i)}_k=k^{-a_{i-r}}$.
\end{proof}

 As corollaries of the results in \cite{NNSY},
 we obtain the following formulas for skew SMZFs.

\subsection{Jacobi-Trudi formulas}

 It is shown in \cite{NNSY} that 
 $S^{(r)}_{\lambda/\mu}(X)$ satisfies the Jacobi-Trudi formulas
\begin{align}
\label{for:JTh}
 S^{(r)}_{\lambda/\mu}(X)
&=\det\left[h^{(\mu_j+r-j+1)}_{\lambda_i-\mu_j-i+j}(X)\right]_{1\le i,j\le r},\\
\label{for:JTe}
 S^{(r)}_{\lambda/\mu}(X)
&=\det\left[e^{(r-1-\mu'_j+j)}_{\lambda'_i-\mu'_j-i+j}(X)\right]_{1\le i,j\le s},
\end{align}
 where $\lambda'=(\lambda'_1,\ldots,\lambda'_s)$ and $\mu'=(\mu'_1,\ldots,\mu'_s)$ are the conjugates of $\lambda$ and $\mu$, respectively
 (we again allow $\lambda'_i=0$ or $\mu'_i=0$). 
 
\begin{thm}
 Retain the above notations.
 Assume that ${\pmb s}=(s_{ij})\in W^{\mathrm{diag}}_{\lambda/\mu}$. 
\begin{itemize}
\item[$(1)$] 
 Assume further that $\Re(s_{i, \lambda_i})>1$ for all $1\le i \le r$.
 Then, we have
\begin{equation}
\label{skewHtype}
 \zeta_{\lambda/\mu}({\pmb s})
=\det\left[\zeta^{\star}(a_{\mu_j-j+1},a_{\mu_j-j+2},\ldots,a_{\mu_j-j+(\lambda_{i}-\mu_j-i+j)})\right]_{1\le i,j\le r\,}.
\end{equation}
 Here, we understand that $\zeta^{\star}(\,\cdots)=1$ if $\lambda_i-\mu_j-i+j=0$ and $0$ if $\lambda_i-\mu_j-i+j<0$.
\item[$(2)$]
 Assume further that $\Re(s_{\lambda'_i, i})>1$ for all $1\le i \le s$.
 Then, we have
\begin{equation}
\label{skewEtype}
 \zeta_{\lambda/\mu}({\pmb s})
=\det\left[\zeta(a_{-\mu'_j+j-1},a_{-\mu'_j+j-2},\ldots,a_{-\mu'_j+j-(\lambda'_{i}-\mu'_j-i+j)})\right]_{1\le i,j\le s}.
\end{equation}
 Here, we understand that $\zeta(\,\cdots)=1$ if $\lambda'_i-\mu'_j-i+j=0$ and $0$ if $\lambda'_i-\mu'_j-i+j<0$.
\end{itemize}
\end{thm}
\begin{proof}
 From \eqref{for:TableauRepresentation}, we have
\begin{align*}
 e^{(r)}_{n}(U)
&=\sum_{m_{1}<m_2<\cdots < m_{n}\le N}u^{(r)}_{m_{1}}u^{(r-1)}_{m_{2}}\cdots u^{(r-n+1)}_{m_{n}},\\[5pt]
\label{for:hU}
 h^{(r)}_{n}(U)
&=\sum_{m_{1}\le m_2\le \cdots \le m_{n}\le N}u^{(r)}_{m_{1}}u^{(r+1)}_{m_{2}}\cdots u^{(r+n-1)}_{m_{n}}.
\end{align*}
 Now, write $r'=\mu_j+r-j+1$ and $k'=\lambda_i-\mu_j-i+j$ for simplicity. 
 Then, we have $u^{(r'+i-1)}_{m_{i}}=m_i^{-a_{\mu_j-j+i}}$ if $u^{(i)}_{k}=k^{-a_{i-r}}$
 and hence   
\begin{align*}
 h^{(\mu_j+r-j+1)}_{\lambda_i-\mu_j-i+j}(U)
&=\sum_{m_{1}\le m_2\le \cdots \le m_{k'}\le N}u^{(r')}_{m_{1}}u^{(r'+1)}_{m_{2}}\cdots u^{(r'+k'-1)}_{m_{k'}}\\
&=\sum_{m_{1}\le m_2\le \cdots \le m_{k'}\le N}m_{1}^{-a_{\mu_j-j+1}}m_{2}^{-a_{\mu_j-j+2}}\cdots m_{k'}^{-a_{\mu_j-j+k'}}\\
&=\zeta^{N\star}(a_{\mu_j-j+1},a_{\mu_j-j+2},\ldots,a_{\mu_j-j+(\lambda_i-\mu_j-i+j)}).
\end{align*}
 This shows that \eqref{skewHtype} follows from \eqref{for:JTh} by letting $N\to\infty$.
 Similarly, \eqref{for:JTe} is obtained from \eqref{skewEtype} via the expression 
\begin{align*}
 e^{(r-1-\mu'_j+j)}_{\lambda'_i-\mu'_j-i+j}(U)
&=\zeta^{N}(a_{-\mu'_j+j-1},a_{-\mu'_j+j-2},\ldots,a_{-\mu'_j+j-(\lambda'_{i}-\mu'_j-i+j)}).
\end{align*}
\end{proof}

\begin{exam}
 When $\lambda/\mu=(4,3,2)/(2,1)$, we have 
\begin{align*} 
\ytableausetup{boxsize=18pt,aligntableaux=center}
\begin{ytableau}
  \none & \none & a_2 & a_3 \\
  \none &   a_0 & a_1 \\
 a_{-2} & a_{-1}
\end{ytableau}
&\,= 
\left|
\begin{array}{ccc}
 \zeta^{\star}(a_{2},a_{3}) & \zeta^{\star}(a_{0},a_{1},a_{2},a_{3}) & \zeta^{\star}(a_{-2},a_{-1},a_{0},a_{1},a_{2},a_{3}) \\
 1 & \zeta^{\star}(a_{0},a_{1}) & \zeta^{\star}(a_{-2},a_{-1},a_{0},a_{1}) \\
 0 & 1 & \zeta^{\star}(a_{-2},a_{-1}) 
\end{array}
\right|,\\
\begin{ytableau}
  \none & \none & a_2 & a_3 \\
  \none &   a_0 & a_1 \\
 a_{-2} & a_{-1}
\end{ytableau}
&\,=
\left|
\begin{array}{cccc}
 \zeta(a_{-2}) & \zeta(a_{0},a_{-1},a_{-2}) & \zeta(a_{2},a_{1},a_{0},a_{-1},a_{-2}) & \zeta(a_{3},a_{2},a_{1},a_{0},a_{-1},a_{-2}) \\
 1 & \zeta(a_{0},a_{-1}) & \zeta(a_{2},a_{1},a_{0},a_{-1}) & \zeta(a_{3},a_{2},a_{1},a_{0},a_{-1}) \\
 0 & 1 & \zeta(a_{2},a_{1}) & \zeta(a_{3},a_{2},a_{1}) \\
 0 & 0 & 1 & \zeta(a_{3}) 
\end{array}
\right|.
\end{align*}
\end{exam}

\subsection{Giambelli formula}
\label{subsec:Giambelli}

 For a partition $\lambda$, 
 we define two sequences of indices $p_1,\ldots,p_t$ and $q_1,\ldots,q_t$ by $p_i=\lambda_i-i+1$ and $q_i=\lambda'_i-i$ for $1\le i\le t$
 where $t$ is the number of diagonal entries of $\lambda$.
 Notice that $p_1>p_2>\cdots >p_t>0$ and $q_1>q_2>\cdots >q_t\ge 0$
 and $\lambda=(p_1-1,\ldots,p_t-1\,|\,q_1,\ldots,q_t)$ is the Frobenius notation of $\lambda$.
 It is shown in \cite{NNSY} that $S^{(r)}_{\lambda}(X)$ satisfies the Giambelli formula
\begin{equation}
\label{for:G}
 S^{(r)}_{\lambda}(X)
=\det\left[S^{(r)}_{(p_i,1^{q_j})}(X)\right]_{1\le i,j \le t}.
\end{equation}
 
\begin{thm}
 Retain the above notations.
 Assume that ${\pmb s}=(s_{ij})\in W^{\mathrm{diag}}_{\lambda}$.
 Moreover, assume further that
 $\Re(s_{i,\lambda_i})=\Re(a_{p_i-1})>1$ and $\Re(s_{\lambda'_i,i})=\Re(a_{-q_i})>1$ for $1\le i \le t$. 
 Then, we have
\[
 \zeta_{\lambda}({\pmb s}) 
=\det\left[\zeta_{(p_i, 1^{q_j})}({\pmb s}_{i,j})\right]_{1 \le i,j \le t},
\]
 where \!
${\pmb s}_{i,j}=\,
\ytableausetup{boxsize=25pt,aligntableaux=center}
\begin{ytableau}
 a_0 & a_1 & a_2 & \cdots & a_{p_i-1} \\
 a_{-1} \\
 \vdots \\
 a_{-q_j} 
\end{ytableau}\in W_{(p_i, 1^{q_j})}
$.
\end{thm}
\begin{proof}
 Putting $u^{(i)}_{k}=k^{-a_{i-r}}$, from \eqref{for:sMZVisS} and \eqref{for:G}, we have  
\[
 \zeta^N_{\lambda}({\pmb s})
=S^{(r)}_{\lambda}(U)
=\det\left[S^{(r)}_{(p_i,1^{q_j})}(U)\right]_{1\le i,j \le t}\\
=\det\left[\zeta^{N}_{(p_i,1^{q_j})}({\pmb s}_{i,j})\right]_{1\le i,j \le t}.
\]
 This leads the desired equation by letting $N\to\infty$.
\end{proof}

\begin{exam}
 When $\lambda=(4,3,3,2)=(3,1,0\,|\,3,2,0)$, we have 
\[
\ytableausetup{boxsize=18pt,aligntableaux=center}
\begin{ytableau}
    a_0 &    a_1 & a_2 & a_3 \\
 a_{-1} &    a_0 & a_1 \\
 a_{-2} & a_{-1} & a_0 \\
 a_{-3} & a_{-2} 
\end{ytableau}
\,
=
\,
\left|~
\begin{array}{ccc}
\begin{ytableau}
 a_0 & a_1 & a_2 & a_3 \\
 a_{-1} \\
 a_{-2} \\
 a_{-3} 
\end{ytableau}
 &
\begin{ytableau}
 a_0 & a_1 & a_2 & a_3 \\
 a_{-1} \\
 a_{-2} 
\end{ytableau}
 &
\begin{ytableau}
 a_0 & a_1 & a_2 & a_3
\end{ytableau}
\\[38pt]
\begin{ytableau}
 a_0 & a_1 \\
 a_{-1} \\
 a_{-2} \\
 a_{-3} 
\end{ytableau}
 &
\begin{ytableau}
 a_0 & a_1 \\
 a_{-1} \\
 a_{-2} 
\end{ytableau}
 &
\begin{ytableau}
 a_0 & a_1 
\end{ytableau} 
\\[38pt]
\begin{ytableau}
 a_0 \\
 a_{-1} \\
 a_{-2} \\
 a_{-3} 
\end{ytableau} 
 &
\begin{ytableau}
 a_0  \\
 a_{-1} \\
 a_{-2}
\end{ytableau} 
 &
\begin{ytableau}
 a_0 
\end{ytableau}
\end{array}
~\right|.
\]
\end{exam}

\subsection{Dual Cauchy formula}

 It is shown in \cite{N1} (see also \cite{N2}) that the dual Cauchy formula
\begin{equation}
\label{for:DC}
 \sum_{\lambda\subset (s^r)}(-1)^{|\lambda|}S^{(r)}_{\lambda}(X)S^{(s)}_{\lambda^{\ast}}(Y)
=\Psi^{(r,s)}(X,Y)
\end{equation}
 holds for $X=[x_{ij}]_{1\le i,j\le \eta}$ and $Y=[y_{ij}]_{1\le i,j\le \eta}$.
 Here, for a partition $\lambda=(\lambda_1,\ldots,\lambda_r)\subset (s^r)$, 
 $\lambda^{\ast}=(r-\lambda'_s,\ldots,r-\lambda'_1)$.
 Moreover, $\Psi^{(r,s)}(X,Y)$ is the dual Cauchy kernel defined by
\[
 \Psi^{(r,s)}(X,Y)
=\frac{\xi^{1,\ldots,r+s}_{1,\ldots,r+s}(Z)}{\xi^{1,\ldots,r}_{1,\ldots,r}(X)\xi^{1,\ldots,s}_{1,\ldots,s}(Y)},\quad 
 Z=
\left[
\begin{array}{cccc}
 x_{11} & x_{12} & \cdots & x_{1\eta} \\
  \vdots & \vdots  & & \vdots \\
 x_{r1} & x_{r2} & \cdots & x_{r\eta} \\
\hline 
 y_{11} & y_{12} & \cdots & y_{1\eta} \\
  \vdots & \vdots  & & \vdots \\
 y_{s1} & y_{s2} & \cdots & y_{s\eta} 
\end{array}
\right] 
.
\]
 Remark that when both $X$ and $Y$ are unitriangular,
 we have $\Psi^{(r,s)}(X,Y)=\det(Z)$.

 We now show an analogue of \eqref{for:DC} for SMZFs.  
 To do that, we first simplify the formula \eqref{for:DC} in the case where $X=U$ and $Y=V$.
 Here, for $M\in\mathbb{N}$,
 $V=V^{(M)}$ is an upper unitriangular matrix of size $\eta$ similarly defined as $U$, that is,
 $V=V_1V_2\cdots V_M$ where 
\[
 V_k
=\left(I_{\eta}+v^{(1)}_k E_{12}\right)\left(I_{\eta}+v^{(2)}_k E_{23}\right)
\cdots \left(I_{\eta}+v^{(\eta-1)}_k E_{\eta-1,\eta}\right)
\]
 with $v^{(i)}_k$ being  variables for $1\le k\le M$ and $1\le i\le \eta-1$.
 
 Write $U=[u_{ij}]_{1\le i,j\le \eta}$ and $V=[v_{ij}]_{1\le i,j\le \eta}$.
 We first show that  
\begin{equation}
\label{for:UijVij}
 u_{ij}
=
\begin{cases}
 h^{(i)}_{j-i}(U) & i\le j, \\
 0 & i>j,
\end{cases}
\quad 
 v_{ij}
=
\begin{cases}
 h^{(i)}_{j-i}(V) & i\le j, \\
 0 & i>j.
\end{cases}
\end{equation}
 Since these are clearly equivalent, 
 let us show only the former.
 Because $U$ is an upper unitiangular matrix,
 we have $u_{ij}=0$ unless $i\le j$. 
 When $i\le j$, we have 
\begin{align*}
 u_{ij}
=\sum^{\eta}_{l_1,\ldots,l_{N-1}=1}(U_{1})_{i,l_1}(U_{2})_{l_1,l_2}\cdots (U_{N})_{l_{N-1},j}.
\end{align*} 
 Here, for a matrix $A$, 
 we denote by $(A)_{i,j}$ the $(i,j)$ entry of $A$.
 Since 
\[
 (U_{k})_{a,b}
=
\begin{cases}
 \displaystyle{\prod^{b-a-1}_{h=0}}u^{(a+h)}_k & a\le b, \\[5pt]
 0 & a>b,
\end{cases}
\] 
 we have 
\[
 u_{ij}
=\sum_{i\le l_1\le \cdots \le l_{N-1}\le j}
\left(\prod^{l_1-i-1}_{h_1=0}u^{(i+h_1)}_1\right)
\left(\prod^{l_2-l_1-1}_{h_2=0}u^{(l_1+h_2)}_2\right)\cdots 
\left(\prod^{j-l_{N-1}-1}_{h_{N}=0}u^{(l_{N-1}+h_N)}_{N}\right).
\]
 Furthermore, writing $j=i+p$, we have 
\begin{align*}
 u_{i,i+p}
&=\sum_{i\le l_1\le \cdots \le l_{N-1}\le i+p}
\left(\prod^{l_1-i-1}_{h_1=0}u^{(i+h_1)}_1\right)
\left(\prod^{l_2-l_1-1}_{h_2=0}u^{(l_1+h_2)}_2\right)\cdots 
\left(\prod^{i+p-l_{N-1}-1}_{h_{N}=0}u^{(l_{N-1}+h_N)}_{N}\right)\\
&=\sum_{1\le m_1\le \cdots\le m_{p}\le N}u^{(i)}_{m_1}u^{(i+1)}_{m_2}\cdots u^{(i+p-1)}_{m_p}\\
&=h^{(i)}_{p}(U),
\end{align*} 
 whence we obtain the claim. 

 When $X=U$ and $Y=V$,
 from \eqref{for:UijVij}, \eqref{for:DC} can be written as follows. 

\begin{cor}
 It holds that 
\begin{equation}
\label{for:DCspecial}
 \sum_{\lambda\subset (s^r)}(-1)^{|\lambda|}S^{(r)}_{\lambda}(U)S^{(s)}_{\lambda^{\ast}}(V)
=\det
\left[
\begin{array}{ccccccc}
 1 & h^{(1)}_{1}(U) & h^{(1)}_{2}(U) & \cdots &  h^{(1)}_{r}(U) & \cdots & h^{(1)}_{\eta-1}(U) \\
 0 & 1 & h^{(2)}_{1}(U) & \cdots & h^{(2)}_{r-1}(U) & \cdots & h^{(2)}_{\eta-2}(U) \\
 \vdots & \ddots & \ddots & \ddots & \vdots & & \vdots \\
 0 & \cdots & 0 & 1 & h^{(r)}_{1}(U) & \cdots & h^{(r)}_{\eta-r}(U) \\
\hline 
 1 & h^{(1)}_{1}(V) & h^{(1)}_{2}(V) & \cdots &  h^{(1)}_{s}(V) & \cdots & h^{(1)}_{\eta-1}(V) \\
 0 & 1 & h^{(2)}_{1}(V) & \cdots & h^{(2)}_{s-1}(V) & \cdots & h^{(2)}_{\eta-2}(V) \\
 \vdots & \ddots & \ddots & \ddots & \vdots & & \vdots \\
 0 & \cdots & 0 & 1 & h^{(s)}_{1}(V) & \cdots & h^{(s)}_{\eta-s}(V) 
\end{array}
\right].
\end{equation}
\end{cor} 

\begin{thm}
 Assume that ${\pmb s}=(s_{ij})\in W^{\mathrm{diag}}_{(s^r)}$ and ${\pmb t}=(t_{ij})\in W^{\mathrm{diag}}_{(r^s)}$
 with $a_{k}=s_{i,i+k}$ and $b_{k}=t_{i,i+k}$ for $k\in\mathbb{Z}$.
 Moreover, assume that $\Re(s_{ij})>1$ for all $1\le i\le r$, $1\le j\le s$ 
 and $\Re(t_{ij})>1$ for all $1\le i\le s$, $1\le j\le r$.    
 Then, we have
\begin{align}
\label{for:DCforzeta}
\ \\[-65pt]
 \sum_{\lambda\subset (s^r)}&(-1)^{|\lambda|}
\zeta_{\lambda}\left({\pmb s}|_{\lambda}\right)\zeta_{\lambda^{\ast}}\left({\pmb t}|_{\lambda^{\ast}}\right)\nonumber\\
&=\det
\left[
\begin{array}{ccccccc}
 1 & \zeta^{\star}(a_{1-r}) & \zeta^{\star}(a_{1-r},a_{2-r}) & \cdots & \zeta^{\star}(a_{1-r},\ldots,a_{0}) & \cdots & \zeta^{\star}(a_{1-r},\ldots,a_{\eta-1-r}) \\
 0 & 1 & \zeta^{\star}(a_{2-r}) & \cdots & \zeta^{\star}(a_{2-r},\ldots,a_{0}) & \cdots & \zeta^{\star}(a_{2-r},\ldots,a_{\eta-1-r}) \\
 \vdots & \ddots & \ddots & \ddots & \vdots & & \vdots \\
 0 & \cdots & 0 & 1 & \zeta^{\star}(a_0) & \cdots &  \zeta^{\star}(a_0, \ldots, a_{\eta-1-r}) \\
\hline 
 1 & \zeta^{\star}(b_{1-s}) & \zeta^{\star}(b_{1-s},b_{2-s}) & \cdots & \zeta^{\star}(b_{1-s},\ldots,b_{0}) & \cdots & \zeta^{\star}(b_{1-s},\ldots,b_{\eta-1-s}) \\
 0 & 1 & \zeta^{\star}(b_{2-s}) & \cdots & \zeta^{\star}(b_{2-s},\ldots,b_{0}) & \cdots & \zeta^{\star}(b_{2-s},\ldots,b_{\eta-1-s}) \\
 \vdots & \ddots & \ddots & \ddots & \vdots & & \vdots \\
 0 & \cdots & 0 & 1 & \zeta^{\star}(b_0) & \cdots & \zeta^{\star}(b_0, \ldots, b_{\eta-1-s}) 
\end{array}
\right].
\nonumber
\end{align}
 Here, ${\pmb s}|_{\lambda}\in W^{\mathrm{diag}}_{\lambda}$ and ${\pmb t}|_{\lambda^{\ast}}\in W^{\mathrm{diag}}_{\lambda^{\ast}}$
 are the shape restriction of ${\pmb s}$ and ${\pmb t}$ to $\lambda$ and $\lambda^{\ast}$, respectively.
\end{thm}
\begin{proof}
 Putting $u^{(i)}_{k}=k^{-a_{i-r}}$ and $v^{(i)}_{k}=k^{-b_{i-r}}$, we have 
\begin{align*}
 h^{(i)}_{k}(U)
&=\sum_{m_{1}\le \cdots \le m_{k}\le N}u^{(i)}_{m_{1}}u^{(i+1)}_{m_{2}}\cdots u^{(i+k-1)}_{m_{k}}\\
&=\sum_{m_{1}\le \cdots \le m_{k}\le N}m_{1}^{-a_{i-r}}m_{2}^{-a_{i+1-r}}\cdots m_{k}^{-a_{i+k-1-r}}\\
&=\zeta^{N\star}(a_{i-r},a_{i+1-r},\ldots,a_{i+k-1-r})
\end{align*}
 and similarly
\begin{align*}
 h^{(i)}_{k}(V)
&=\zeta^{M\star}(b_{i-s},b_{i+1-s},\ldots,b_{i+k-1-s}).
\end{align*}
 Therefore, \eqref{for:DCspecial} immediately yields \eqref{for:DCforzeta} by letting $N,M\to\infty$. 
\end{proof}

\begin{exam}
 When $r=2$ and $s=3$, we have
\begin{align*}
\ytableausetup{boxsize=18pt,aligntableaux=center}
 (\text{LHS of \eqref{for:DCforzeta}})
&=-\,
\begin{ytableau}
    a_0 & a_1 & a_2 \\
 a_{-1} & a_0 & a_1 
\end{ytableau} 
\ \cdot\ 1
-\,
\begin{ytableau}
    a_0 & a_1 & a_2 \\
 a_{-1} & a_0  
\end{ytableau} 
\ \cdot\ 
\begin{ytableau}
 b_0 
\end{ytableau} 
+\, 
\begin{ytableau}
    a_0 & a_1 & a_2 \\
 a_{-1} 
\end{ytableau} 
\ \cdot\ 
\begin{ytableau}
    b_0 \\
 b_{-1}  
\end{ytableau}
\\
&\ \ \ -\,
\begin{ytableau}
 a_0 & a_1 & a_2 \\
\end{ytableau} 
\ \cdot\ 
\begin{ytableau}
    b_0 \\
 b_{-1} \\
 b_{-2} 
\end{ytableau}
+\, 
\begin{ytableau}
    a_0 & a_1 \\
 a_{-1} & a_0  
\end{ytableau} 
\ \cdot\ 
\begin{ytableau}
 b_0 & b_1 
\end{ytableau}
-\, 
\begin{ytableau}
    a_0 & a_1 \\
 a_{-1}  
\end{ytableau} 
\ \cdot\ 
\begin{ytableau}
    b_0 & b_1 \\
 b_{-1}  
\end{ytableau} 
\\
&\ \ \ 
+\, 
\begin{ytableau}
 a_0 & a_1  
\end{ytableau}
\ \cdot\ 
\begin{ytableau}
    b_0 & b_1 \\
 b_{-1} \\
 b_{-2}  
\end{ytableau}
+\, 
\begin{ytableau}
    a_0 \\
 a_{-1}  
\end{ytableau} 
\ \cdot\ 
\begin{ytableau}
    b_0 & b_1 \\
 b_{-1} & b_0
\end{ytableau} 
-\, 
\begin{ytableau}
 a_0  
\end{ytableau} 
\ \cdot\ 
\begin{ytableau}
    b_0 & b_1 \\
 b_{-1} & b_0 \\
 b_{-2} 
\end{ytableau} 
+\,
1
\ \cdot\ 
\begin{ytableau}
    b_0 & b_1 \\
 b_{-1} & b_0 \\
 b_{-2} & b_{-1} 
\end{ytableau} 
\,.
\end{align*} 
 On the other hand, we have 
\begin{align*}
 (\text{RHS of \eqref{for:DCforzeta}})
&=\det
\left[
\begin{array}{ccccc}
 1 & \zeta^{\star}(a_{-1}) & \zeta^{\star}(a_{-1},a_{0}) & \zeta^{\star}(a_{-1},a_{0},a_1) & \zeta^{\star}(a_{-1},a_{0},a_1,a_2) \\[3pt]
 0 & 1 &  \zeta^{\star}(a_{0}) & \zeta^{\star}(a_{0},a_{1}) & \zeta^{\star}(a_{0},a_{1},a_2) \\[3pt]
\hline
 1 & \zeta^{\star}(b_{-2}) & \zeta^{\star}(b_{-2},b_{-1}) & \zeta^{\star}(b_{-2},b_{-1},b_0) & \zeta^{\star}(b_{-2},b_{-1},b_0,b_1) \\[3pt]
 0 & 1 & \zeta^{\star}(b_{-1}) & \zeta^{\star}(b_{-1},b_{0}) & \zeta^{\star}(b_{-1},b_{0},b_1) \\[3pt]
 0 & 0 & 1 & \zeta^{\star}(b_{0}) & \zeta^{\star}(b_{0},b_{1}) 
\end{array}
\right]
.
\end{align*}
\end{exam}


\section{Schur type quasi-symmetric functions}
\label{sec:Sqsf}

 We here investigate SMZFs
 from the view point of the quasi-symmetric functions introduced by Gessel \cite{G}.

\subsection{Quasi-symmetric functions}
 
 Let ${\pmb t}=(t_1,t_2,\ldots)$ be variables and
 $\mathfrak{P}$ a subalgebra of $\mathbb{Z}[\![t_1,t_2,\ldots\,]\!]$ consisting of all formal power series with integer coefficients of bounded degree.
 We call $p=p({\pmb t})\in \mathfrak{P}$ a {\it quasi-symmetric function}
 if the coefficient of $t^{\gamma_1}_{k_1}t^{\gamma_2}_{k_2}\cdots t^{\gamma_n}_{k_n}$ of $p$ is the same as
 that of $t^{\gamma_1}_{h_1}t^{\gamma_2}_{h_2}\cdots t^{\gamma_l}_{h_n}$ of $p$ whenever $k_1<k_2<\cdots <k_n$ and $h_1<h_2<\cdots <h_n$.
 The algebra of all quasi-symmetric functions is denoted by $\mathrm{Qsym}$. 
 For a composition ${\pmb \gamma}=(\gamma_1,\gamma_2,\ldots,\gamma_n)$ of a positive integer,
 define the {\it monomial quasi-symmetric function} $M_{\pmb \gamma}$ and 
 the {\it essential quasi-symmetric function} $E_{\pmb \gamma}$ respectively by
\[
 M_{\pmb \gamma}
=\sum_{m_1<m_2<\cdots<m_n}t_{m_1}^{\gamma_1}t_{m_2}^{\gamma_2}\cdots t^{\gamma_n}_{m_n},
\quad
 E_{\pmb \gamma}
=\sum_{m_1\le m_2\le \cdots \le m_n}t_{m_1}^{\gamma_1}t_{m_2}^{\gamma_2}\cdots t^{\gamma_n}_{m_n}.
\]
 We know that these respectively form integral basis of $\mathrm{Qsym}$.
 Notice that 
\begin{equation}
\label{for:EM}
 E_{\pmb \gamma}
=\sum_{{\pmb \delta} \,\preceq\, {\pmb \gamma}}M_{\pmb \delta}, \quad 
 M_{\pmb \gamma}
=\sum_{{\pmb \delta} \,\preceq\, {\pmb \gamma}}(-1)^{n-\ell({\pmb \delta})}E_{\pmb \delta}.
\end{equation}

\subsection{Relation between quasi-symmetric functions and multiple zeta values}
 
 A relation between the multiple zeta values and the quasi-symmetric functions is studied by Hoffman \cite{H2}
 (remark that the notations of MZF and MZSF in \cite{H2} are different from ours;
 they are $\zeta(s_n,s_{n-1},\ldots,s_1)$ and $\zeta^{\star}(s_n,s_{n-1},\ldots,s_1)$, respectively, in our notations).  
 Let $\mathfrak{H}=\mathbb{Z}\langle x,y\rangle$ be the noncommutative polynomial algebra over $\mathbb{Z}$.
 We can define a commutative and associative multiplication $\ast$, called a $\ast$-product, on $\mathfrak{H}$.
 We call $(\mathfrak{H},\ast)$ the (integral) harmonic algebra.   
 Let $\mathfrak{H}^{1}=\mathbb{Z}1+y\mathfrak{H}$, which is a subalgebra of $\mathfrak{H}$.
 Notice that every $w\in \mathfrak{H}^{1}$ can be written as 
 an integral linear combination of $z_{\gamma_1}z_{\gamma_2}\cdots z_{\gamma_n}$ where $z_{\gamma}=yx^{\gamma-1}$ for $\gamma\in\mathbb{N}$.
 For each $N\in \mathbb{N}$,
 define the homomorphism $\phi_N:\mathfrak{H}^{1}\to \mathbb{Z}[t_1,t_2,\ldots,t_N]$ by $\phi_N(1)=1$ and 
\[
 \phi_N(z_{\gamma_1}z_{\gamma_2}\cdots z_{\gamma_n})
=
\begin{cases}
 \displaystyle{\sum_{m_1<m_2<\cdots<m_n\le N}t_{m_1}^{\gamma_1}t_{m_2}^{\gamma_2}\cdots t^{\gamma_n}_{m_n}} & n\le N, \\
 0 & \text{otherwise},
\end{cases}
\]
 and extend it additively to $\mathfrak{H}^{1}$.
 There is a unique homomorphism $\phi:\mathfrak{H}^{1}\to \mathfrak{P}$ such that $\pi_N\phi=\phi_N$
 where $\pi_N$ is the natural projection from $\mathfrak{P}$ to $\mathbb{Z}[t_1,t_2,\ldots,t_N]$.
 We have $\phi(z_{\gamma_1}z_{\gamma_2}\cdots z_{\gamma_n})=M_{(\gamma_1,\gamma_2,\ldots,\gamma_n)}$.
 Moreover, as is described in \cite{H2}, $\phi$ is an isomorphism between $\mathfrak{H}^1$ and $\mathrm{Qsym}$.

 Let $e$ be the function sending $t_i$ to $\frac{1}{i}$. 
 Moreover, define $\rho_N:\mathfrak{H}^{1}\to\mathbb{R}$ by $\rho_N=e\phi_N$.
 For a composition ${\pmb \gamma}$, we have 
\[
 \rho_N\phi^{-1}(M_{\pmb \gamma})=\zeta^{N}({\pmb \gamma}), \quad 
 \rho_N\phi^{-1}(E_{\pmb \gamma})=\zeta^{N\star}({\pmb \gamma}). 
\]
 Here, the second formula follows from the first equations of \eqref{for:zsNzN} and \eqref{for:EM}. 
 Define the map $\rho:\mathfrak{H}^1\to\mathbb{R}^{\mathbb{N}}$ by 
 $\rho(w)=(\rho_N(w))_{N\in\mathbb{N}}$ for $w\in\mathfrak{H}^{1}$. 
 Notice that if $w\in \mathfrak{H}^0=\mathbb{Z}1+y\mathfrak{H}x$, which is a subalgebra of $\mathfrak{H}^1$,
 then we may understand that $\rho(w)=\lim_{N\to\infty}\rho_N(w)\in\mathbb{R}$.
 In particular, for a composition ${\pmb \gamma}=(\gamma_1,\gamma_2,\ldots,\gamma_n)$ with $\gamma_n\ge 2$,
 we have 
\begin{equation}
\label{for:rho}
 \rho\phi^{-1}(M_{\pmb \gamma})
=\zeta({\pmb \gamma}), \quad
 \rho\phi^{-1}(E_{\pmb \gamma})
=\zeta^{\star}({\pmb \gamma}).
\end{equation}

\subsection{Schur type quasi-symmetric functions}

 Now, one easily reaches the definition of the following {\it Schur type quasi-symmetric functions} (of skew type). 
 For partitions $\lambda,\mu$ satisfying $\mu\subset \lambda\subset (s^r)$ and ${\pmb \gamma}=(\gamma_{ij})\in T(\lambda/\mu,\mathbb{N})$, 
 define  
\[
 S_{\lambda/\mu}({\pmb \gamma})
=\sum_{(m_{ij})\in\mathrm{SSYT}(\lambda)}\prod_{(i,j)\in D(\lambda/\mu)}t^{\gamma_{ij}}_{m_{ij}},
\]
 which is actually in $\mathrm{Qsym}$.
 Clearly we have 
\begin{align*}
\ytableausetup{boxsize=normal,aligntableaux=center}
 S_{(1^n)} 
\left(~ 
\begin{ytableau}
 \gamma_1 \\
 \lower2pt\vdots \\
 \gamma_{n}
\end{ytableau} 
~\right)
=M_{(\gamma_1,\ldots,\gamma_n)},
 \quad
 S_{(n)} 
\left(~
\begin{ytableau}
 \gamma_1 & \cdots & \gamma_{n}
\end{ytableau} 
~\right) 
=E_{(\gamma_1,\ldots,\gamma_n)}.
\end{align*} 
 Hence $S_{\lambda/\mu}({\pmb \gamma})$ interpolates both the monomial and essential quasi-symmetric functions.
 Moreover, one sees that this is the quasi-symmetric function
 corresponding to the Schur multiple zeta value in the sense of \eqref{for:rho}. 
 
\begin{lemma}
 Let 
\[
 I_{\lambda/\mu}
=
\left\{{\pmb \gamma}=(\gamma_{ij})\in T(\lambda/\mu,\mathbb{N})\,\left|\,\text{$\gamma_{ij}\ge 2$ for all $(i,j)\in C(\lambda/\mu)$}
\right.
\right\}.
\]
 Then, for ${\pmb \gamma}\in I_{\lambda/\mu}$,
 we have 
\[
 \rho\phi^{-1}(S_{\lambda/\mu}({\pmb \gamma}))
=\zeta_{\lambda/\mu}({\pmb \gamma}).
\]
\end{lemma}
\begin{proof}
 This follows from one of the following expressions 
\begin{equation}
\label{for:SMorSE}
 S_{\lambda/\mu}({\pmb \gamma})
=\sum_{{\pmb u} \,\preceq\, {\pmb \gamma}}M_{\pmb u}, \quad
 S_{\lambda/\mu}({\pmb \gamma})
=\sum_{{\pmb u} \,\preceq\, {\pmb \gamma}'}(-1)^{|\lambda/\mu|-\ell({\pmb u})}E_{\pmb u},
\end{equation}
 similarly obtained as \eqref{for:SchurtoMZVskew},
 together with \eqref{for:SchurtoMZVskew} and \eqref{for:rho}.
\end{proof}

\begin{rmk}
 There is another important class of quasi-symmetric functions called the {\it fundamental} or {\it ribbon quasi-symmetric function} defined by 
 $F_{\pmb \gamma}=\sum_{{\pmb \delta} \,\succeq\, {\pmb \gamma}}M_{\pmb \delta}$ for a composition ${\pmb \gamma}$.
 We remark that they are not in the class of Schur type quasi-symmetric functions.
\end{rmk}

 We again concentrate on the case ${\pmb \gamma}=(\gamma_{ij})\in T^{\mathrm{diag}}(\lambda/\mu,\mathbb{N})$.
 Write $c_{k}=\gamma_{i,i+k}$ for $k\in\mathbb{Z}$ (and for any $i\in\mathbb{N}$).
 Then, from the tableau expression \eqref{for:TableauRepresentation} of the ninth variation of the Schur function $S^{(r)}_{\lambda/\mu}(U)$,
 if we put $u_k^{(i)}=t^{c_{i-r}}_k$, then we have $u^{(r-i+j)}_{m_{ij}}=t^{\gamma_{ij}}_{m_{ij}}$ 
 and hence   
\begin{align*}
 S^{(r)}_{\lambda/\mu}(U)
&=\sum_{(m_{ij})\in \mathrm{SSYT}_N(\lambda/\mu)}\prod_{(i,j)\in D(\lambda/\mu)}u^{(r-i+j)}_{m_{ij}}\\
&=\sum_{(m_{ij})\in \mathrm{SSYT}_N(\lambda/\mu)}\prod_{(i,j)\in D(\lambda/\mu)}t^{\gamma_{ij}}_{m_{ij}}\\
&=\phi_N\phi^{-1}(S_{\lambda/\mu}({\pmb \gamma})).
\end{align*}
 This shows that, when ${\pmb \gamma}\in T^{\mathrm{diag}}(\lambda/\mu,\mathbb{N})$,
 the Schur type quasi-symmetric function $S_{\lambda/\mu}({\pmb \gamma})$
 is also realized as (the limit of) a specialization of the ninth variation of the Schur functions,
 whence we can similarly obtain the Jacobi-Trudi, Giambelli and dual Cauchy formulas
 for such quasi-symmetric functions.
 Notice that the following formulas actually hold in the algebra of formal power series,
 which means that we do not need any further assumptions on variables such as appeared in the corresponding results
 in the previous section for SMZFs. 
 
\begin{thm}
 Assume that ${\pmb \gamma}=(\gamma_{ij})\in T^{\mathrm{diag}}(\lambda/\mu,\mathbb{N})$
 and write $c_{k}=\gamma_{i,i+k}$ for $k\in\mathbb{Z}$. 
\begin{itemize}
\item[$(1)$] 
 We have
\begin{equation}
\label{for:JTquasiE}
 S_{\lambda/\mu}({\pmb \gamma})
=\det\left[E_{(c_{\mu_j-j+1},c_{\mu_j-j+2},\ldots,c_{\mu_j-j+(\lambda_{i}-\mu_j-i+j)})}\right]_{1\le i,j\le r\,}.
\end{equation}
 Here, we understand that $E_{(\,\cdots)}=1$ if $\lambda_i-\mu_j-i+j=0$ and $0$ if $\lambda_i-\mu_j-i+j<0$.
\item[$(2)$]
 We have
\begin{equation}
\label{for:JTquasiM}
 S_{\lambda/\mu}({\pmb \gamma})
=\det\left[M_{(c_{-\mu'_j+j-1},c_{-\mu'_j+j-2},\ldots,c_{-\mu'_j+j-(\lambda'_{i}-\mu'_j-i+j)})}\right]_{1\le i,j\le s}.
\end{equation}
 Here, we understand that $M_{(\,\cdots)}=1$ if $\lambda'_i-\mu'_j-i+j=0$ and $0$ if $\lambda'_i-\mu'_j-i+j<0$.
\end{itemize}
\end{thm}

\begin{thm}
 Let $\lambda=(p_1-1,\ldots,p_t-1\,|\,q_1,\ldots,q_t)$ be a partition written
 in the Frobenius notation (see Section~\ref{subsec:Giambelli}).
 Assume that ${\pmb \gamma}=(\gamma_{ij})\in T^{\mathrm{diag}}(\lambda,\mathbb{N})$
 and write $c_{k}=\gamma_{i,i+k}$ for $k\in\mathbb{Z}$.
 Then, we have
\[
 S_{\lambda}({\pmb \gamma}) 
=\det\left[S_{(p_i,1^{q_j})}({\pmb \gamma}_{i,j})\right]_{1 \le i,j \le t},
\]
 where \!
${\pmb \gamma}_{i,j}=\,
\ytableausetup{boxsize=25pt,aligntableaux=center}
\begin{ytableau}
 c_0 & c_1 & c_2 & \cdots & c_{p_i-1} \\
 c_{-1} \\
 \vdots \\
 c_{-q_j} 
\end{ytableau}
\in T((p_i,1^{q_j}),\mathbb{N})
$.
\end{thm}

\begin{thm}
 Assume that ${\pmb \gamma}=(\gamma_{ij})\in T^{\mathrm{diag}}((s^r),\mathbb{N})$ and
 ${\pmb \delta}=(\delta_{ij})\in T^{\mathrm{diag}}((r^s),\mathbb{N})$
 with $c_{k}=\gamma_{i,i+k}$ and $d_{k}=\delta_{i,i+k}$ for $k\in\mathbb{Z}$.
 Write $\eta=r+s$.
 Then, we have 
\begin{align}
\label{for:dualCauchyqsym}
 \sum_{\lambda\subset (s^r)}&(-1)^{|\lambda|}
S_{\lambda}\left({\pmb \gamma}|_{\lambda}\right)S_{\lambda^{\ast}}\left({\pmb \delta}|_{\lambda^{\ast}}\right)\\
&=\det
\left[
\begin{array}{ccccccc}
 1 & E_{(c_{1-r})} & E_{(c_{1-r},c_{2-r})} & \cdots & E_{(c_{1-r},\ldots,c_{0})} & \cdots & E_{(c_{1-r},\ldots,c_{\eta-1-r})} \\
 0 & 1 & E_{(c_{2-r})} & \cdots & E_{(c_{2-r},\ldots,c_{0})} & \cdots & E_{(c_{2-r},\ldots,c_{\eta-1-r})} \\
 \vdots & \ddots & \ddots & \ddots & \vdots & & \vdots \\
 0 & \cdots & 0 & 1 & E_{(c_0)} & \cdots &  E_{(c_0,\ldots,c_{\eta-1-r})} \\
\hline 
 1 & E_{(d_{1-s})} & E_{(d_{1-s},d_{2-s})} & \cdots & E_{(d_{1-s},\ldots,d_{0})} & \cdots & E_{(d_{1-s},\ldots,d_{\eta-1-s})} \\
 0 & 1 & E_{(d_{2-s})} & \cdots & E_{(d_{2-s},\ldots,d_{0})} & \cdots & E_{(d_{2-s},\ldots,d_{\eta-1-s})} \\
 \vdots & \ddots & \ddots & \ddots & \vdots & & \vdots \\
 0 & \cdots & 0 & 1 & E_{(d_0)} & \cdots & E_{(d_0,\ldots,d_{\eta-1-s})} 
\end{array}
\right].
\nonumber
\end{align}
 Here, ${\pmb \gamma}|_{\lambda}\in T^{\mathrm{diag}}(\lambda,\mathbb{N})$ and ${\pmb \delta}|_{\lambda^{\ast}}\in T^{\mathrm{diag}}(\lambda^{\ast},\mathbb{N})$
 are the shape restriction of ${\pmb \gamma}$ and ${\pmb \delta}$ to $\lambda$ and $\lambda^{\ast}$, respectively.
\end{thm}

\begin{rmk}
 In \cite{{MalvenutoReutenauer1998}},
 a more general type of quasi-symmetric function is defined by a set of equality and inequality conditions. 
 One can see that this includes both the Schur type quasi-symmetric functions and the fundamental quasi-symmetric functions as special cases
 and actually leads a generalized multiple zeta function via $\rho\phi^{-1}$.
 However, because it is too complicated in general,
 it seems to be difficult to expect that such generalized quasi-symmetric and multiple zeta functions
 satisfy the similar kind of determinant formulas as above.
\end{rmk}

 We know that $\mathrm{Qsym}$ has a commutative Hopf algebra structure (see \cite{H2,Kassel1995,MilnorMoore1965,Sweedler1969}).
 The antipode $S$, which is an automorphism of $\mathrm{Qsym}$ satisfying $S^2=\mathrm{id}$,
 is explicitly given as follows.

\begin{thm}
[{\cite[Theorem~3.1]{H2}}]
\label{thm:Hoffman} 
 For a composition ${\pmb \gamma}=(\gamma_1,\gamma_2,\ldots,\gamma_n)$, we have 
\begin{itemize}
\item[$(1)$]
 $\displaystyle{S(M_{{\pmb \gamma}})
=\sum_{{\pmb \gamma}_1 \,\sqcup\, {\pmb \gamma}_2 \,\sqcup\, \cdots \,\sqcup\, {\pmb \gamma}_m\,=\,{\pmb \gamma}}(-1)^m M_{{\pmb \gamma}_1}M_{{\pmb \gamma}_2}\cdots M_{{\pmb \gamma}_m}}$.
\item[$(2)$]
 $S(M_{{\pmb \gamma}})
=(-1)^{n}E_{\overline{{\pmb \gamma}}}$.
\end{itemize}
 Here, 
 ${\pmb \gamma}_1\sqcup {\pmb \gamma}_2\sqcup \cdots \sqcup {\pmb \gamma}_m$ is just the juxtaposition of non-empty compositions 
 ${\pmb \gamma}_1,{\pmb \gamma}_2,\ldots,{\pmb \gamma}_m$ and 
 $\overline{\pmb \gamma}=(\gamma_n,\gamma_{n-1},\ldots,\gamma_1)$.
\end{thm}

 Combining these formulas, 
 we reach the expressions  
\begin{align}
\label{for:MsumE}
 M_{{\pmb \gamma}}
&=\sum_{{\pmb \gamma}_1 \,\sqcup\, {\pmb \gamma}_2 \,\sqcup\, \cdots \,\sqcup\, {\pmb \gamma}_m\,=\,\overline{\pmb \gamma}}(-1)^{n-m}
 E_{{\pmb \gamma}_1}E_{{\pmb \gamma}_2}\cdots E_{{\pmb \gamma}_m},\\
\label{for:EsumM}
 E_{{\pmb \gamma}}
&=\sum_{{\pmb \gamma}_1 \,\sqcup\, {\pmb \gamma}_2 \,\sqcup\, \cdots \,\sqcup\, {\pmb \gamma}_m\,=\,\overline{\pmb \gamma}}(-1)^{n-m}
M_{{\pmb \gamma}_1}M_{{\pmb \gamma}_2}\cdots M_{{\pmb \gamma}_m}.
\end{align}
 One sees by induction on $n$ that  
 \eqref{for:MsumE} and \eqref{for:EsumM} are respectively equivalent to the formulas 
\begin{align*}
 M_{(\gamma_1,\ldots,\gamma_n)}
&=
\left|
\begin{array}{cccccc}
 E_{(\gamma_1)} & E_{(\gamma_2,\gamma_1)} & & \cdots & \cdots & E_{(\gamma_n,\ldots,\gamma_2,\gamma_1)} \\
 1 & E_{(\gamma_2)} & & \cdots & \cdots & E_{(\gamma_n,\ldots,\gamma_2)} \\
 & 1 & \ddots & & & \vdots \\
 & & \ddots & 1& E_{(\gamma_{n-1})} & E_{(\gamma_{n},\gamma_{n-1})} \\
 \scalebox{2}{$0$} & & & & 1 & E_{(\gamma_{n})}
\end{array}
\right|,\\[5pt]
 E_{(\gamma_1,\ldots,\gamma_n)}
&=
\left|
\begin{array}{cccccc}
 M_{(\gamma_1)} & M_{(\gamma_2,\gamma_1)} & & \cdots & \cdots & M_{(\gamma_n,\ldots,\gamma_2,\gamma_1)} \\
 1 & M_{(\gamma_2)} & & \cdots & \cdots & M_{(\gamma_n,\ldots,\gamma_2)} \\
  & 1 & \ddots & & & \vdots \\
 & & \ddots & 1& M_{(\gamma_{n-1})} & M_{(\gamma_{n},\gamma_{n-1})} \\
 \scalebox{2}{$0$} & & & & 1 & M_{(\gamma_{n})}
\end{array}
\right|,
\end{align*}
 which are obtained from the Jacobi-Trudi formulas \eqref{for:JTquasiE} and \eqref{for:JTquasiM}, respectively.

\begin{exam}
 When $n=3$, we have 
\begin{align*}
 M_{(\gamma_1,\gamma_2,\gamma_3)}
&=E_{(\gamma_3,\gamma_2,\gamma_1)}
-E_{(\gamma_3,\gamma_2)}E_{(\gamma_1)}-E_{(\gamma_3)}E_{(\gamma_2,\gamma_1)}
+E_{(\gamma_3)}E_{(\gamma_2)}E_{(\gamma_1)}\\
&=
\left|
\begin{array}{ccc}
 E_{(\gamma_1)} & E_{(\gamma_2,\gamma_1)} & E_{(\gamma_3,\gamma_2,\gamma_1)} \\
 1 & E_{(\gamma_2)} & E_{(\gamma_3,\gamma_2)} \\
 0 & 1 & E_{(\gamma_{3})}
\end{array}
\right|,\\
 E_{(\gamma_1,\gamma_2,\gamma_3)}
&=M_{(\gamma_3,\gamma_2,\gamma_1)}
-M_{(\gamma_3,\gamma_2)}M_{(\gamma_1)}-M_{(\gamma_3)}M_{(\gamma_2,\gamma_1)}
+M_{(\gamma_3)}M_{(\gamma_2)}M_{(\gamma_1)}\\
&=
\left|
\begin{array}{ccc}
 M_{(\gamma_1)} & M_{(\gamma_2,\gamma_1)} & M_{(\gamma_3,\gamma_2,\gamma_1)} \\
 1 & M_{(\gamma_2)} & M_{(\gamma_3,\gamma_2)} \\
 0 & 1 & M_{(\gamma_{3})}
\end{array}
\right|.
\end{align*} 
\end{exam}

 For a skew Young diagram $\nu$, 
 we denote by $\nu^{\#}$ the transpose of $\nu$ with respect to the anti-diagonal.
 Similarly, the anti-diagonal transpose of a skew Young tableaux $T\in T(\nu,X)$ is denoted by $T^{\#}\in T(\nu^{\#},X)$. 
 In the following discussion,
 we also encounter $(T^{\#})'\in T((\nu^{\#})',X)$, the conjugate of $T^{\#}$. 
 For example,
\[
\ytableausetup{boxsize=normal,aligntableaux=center}
\begin{ytableau}
 \gamma_{11} & \gamma_{12} & \gamma_{13} \\
 \gamma_{21} 
\end{ytableau}^{\,\#}
=
\,
\begin{ytableau}
 \none & \gamma_{13} \\
 \none & \gamma_{12} \\
 \gamma_{21} & \gamma_{11}
\end{ytableau}
\,,\quad
\left(~
\begin{ytableau}
 \gamma_{11} & \gamma_{12} & \gamma_{13} \\
 \gamma_{21} 
\end{ytableau}^{\,\#}
~\right)'
=
\,
\begin{ytableau}
 \none & \none & \gamma_{21} \\ 
 \gamma_{13} & \gamma_{12} & \gamma_{11} 
\end{ytableau}
\]
 Namely, $(T^{\#})'$ is just the rotation of $T$ by $\pi$ around the center of $\nu$.
 Now, the image of the Schur type quasi-symmetric functions by the antipode $S$ is explicitly calculated as follows.

\begin{thm} 
\label{thm:SSchurquasi}
 For a skew Young diagram $\nu$,
 we have 
\begin{equation}
\label{for:SSnu}
 S(S_{\nu}({\pmb \gamma}))
=(-1)^{|\nu|}S_{\nu^{\#}}({\pmb \gamma}^{\#}).
\end{equation}
 Moreover, when ${\pmb \gamma}\in T^{\mathrm{diag}}(\nu,\mathbb{N})$, we have 
\begin{align}
\label{for:SSnuH}
 S(S_{\nu}({\pmb \gamma}))
&=(-1)^{|\nu|}\sum_{\Theta=(\theta_1,\theta_2,\ldots,\theta_r)\in \mathrm{Rim}^{\nu^{\#}}_H}\varepsilon_H(\Theta) 
E_{\theta_1({\pmb \gamma}^{\#})}E_{\theta_2({\pmb \gamma}^{\#})} \cdots E_{\theta_r({\pmb \gamma}^{\#})},\\
\label{for:SSnuE}
 S(S_{\nu}({\pmb \gamma}))
&=(-1)^{|\nu|}\sum_{\Theta=(\theta_1,\theta_2,\ldots,\theta_s)\in \mathrm{Rim}^{\nu^{\#}}_E}\varepsilon_E(\Theta) 
M_{\theta_1({\pmb \gamma}^{\#})}M_{\theta_2({\pmb \gamma}^{\#})} \cdots M_{\theta_s({\pmb \gamma}^{\#})}.
\end{align}
\end{thm}
\begin{proof}
 From \eqref{for:SMorSE} and Theorem~\ref{thm:Hoffman} (2),
 we have 
\begin{align*}
 S(S_{\nu}({\pmb \gamma}))
&=\sum_{{\pmb u} \,\preceq\, {\pmb \gamma}}S(M_{\pmb u})\\
&=\sum_{\overline{\pmb u} \,\preceq\, {\pmb \gamma}}(-1)^{\ell({\pmb u})}E_{{\pmb u}}\\
&=(-1)^{|\nu|}\sum_{{\pmb u} \,\preceq\, ({\pmb \gamma}^{\#})'}(-1)^{|\nu|-\ell({\pmb u})}E_{{\pmb u}}\\
&=(-1)^{|\nu|}S_{\nu^{\#}}({\pmb \gamma}^{\#}).
\end{align*}  
 Notice that, in the third equality,
 we have used the fact that 
 $\overline{\pmb u} \preceq {\pmb \gamma}$ if and only if ${\pmb u} \preceq ({\pmb \gamma}^{\#})'$,
 which can be verified directly.
 This shows \eqref{for:SSnu}.
 Now, the rest of assertions are immediately obtained from  
\begin{align*}
 S_{\nu}({\pmb \gamma})
&=\sum_{\Theta=(\theta_1,\theta_2,\ldots,\theta_r)\in \mathrm{Rim}^{\nu}_H}
\varepsilon_H(\Theta)
 E_{\theta_1({\pmb \gamma})}E_{\theta_2({\pmb \gamma})}\cdots E_{\theta_r({\pmb \gamma})},\\
 S_{\nu}({\pmb \gamma})
&=\sum_{\Theta=(\theta_1,\theta_2,\ldots,\theta_s)\in \mathrm{Rim}^{\nu}_E}
\varepsilon_E(\Theta)
 M_{\theta_1({\pmb \gamma})}M_{\theta_2({\pmb \gamma})}\cdots M_{\theta_s({\pmb \gamma})},
\end{align*}
 which are similarly obtained as \eqref{for:keyexpressionHrim} and \eqref{for:keyexpressionErim}
 (hence we need the assumption ${\pmb \gamma}\in T^{\mathrm{diag}}(\nu,\mathbb{N})$) 
 and lead the Jacobi-Trudi formulas \eqref{for:JTquasiE} and \eqref{for:JTquasiM}
 for the Schur type quasi-symmetric functions.
 This completes the proof.
\end{proof}

\begin{rmk}
 The formula \eqref{for:SSnuE} with $\nu=(1^n)$ is nothing but the one in Theorem~\ref{thm:Hoffman} (1).
\end{rmk}

\begin{exam}
 When $\nu=(3,1)$, we have from \eqref{for:SSnu} 
\begin{align*}
\ytableausetup{boxsize=normal,aligntableaux=center}
 S
\left(
 S_{(3,1)}
\left(~
\begin{ytableau}
 \gamma_{11} & \gamma_{12} & \gamma_{13} \\
 \gamma_{21} 
\end{ytableau}
~\right)
\right)
&=S_{(2,2,2)/(1,1)}
\left(~
\begin{ytableau}
 \none & \gamma_{13} \\
 \none & \gamma_{12} \\
 \gamma_{21} & \gamma_{11} 
\end{ytableau}
~\right)
\\
&=E_{(\gamma_{21},\gamma_{13},\gamma_{12},\gamma_{11})}-E_{(\gamma_{21}+\gamma_{13},\gamma_{12},\gamma_{11})}-E_{(\gamma_{21},\gamma_{13}+\gamma_{12},\gamma_{11})}\\
&\ \ \ -E_{(\gamma_{21},\gamma_{13},\gamma_{12}+\gamma_{11})}+E_{(\gamma_{21}+\gamma_{13}+\gamma_{12},\gamma_{11})}+E_{(\gamma_{21}+\gamma_{13},\gamma_{12}+\gamma_{11})}\\
&\ \ \ +E_{(\gamma_{21},\gamma_{13}+\gamma_{12}+\gamma_{11})}+E_{(\gamma_{13},\gamma_{21},\gamma_{12},\gamma_{11})}-E_{(\gamma_{13},\gamma_{21},\gamma_{12}+\gamma_{11})}\\
&\ \ \ +E_{(\gamma_{13},\gamma_{12},\gamma_{21},\gamma_{11})}-E_{(\gamma_{13}+\gamma_{12},\gamma_{21},\gamma_{11})}-E_{(\gamma_{13},\gamma_{12}+\gamma_{21},\gamma_{11})}\\
&=M_{(\gamma_{21},\gamma_{13},\gamma_{12},\gamma_{11})}+M_{(\gamma_{21}+\gamma_{13},\gamma_{12},\gamma_{11})}+M_{(\gamma_{13},\gamma_{21}+\gamma_{12},\gamma_{11})}\\
&\ \ \ +M_{(\gamma_{13},\gamma_{21},\gamma_{12},\gamma_{11})}+M_{(\gamma_{13},\gamma_{12},\gamma_{21}+\gamma_{11})}+M_{(\gamma_{13},\gamma_{21}+\gamma_{12},\gamma_{11})}.
\end{align*} 
 Here, the second and third equations are similarly obtained as in Example~\ref{ex:skewSMZ}. 
 On the other hand, 
 we have from \eqref{for:SSnuH} 
\begin{align*}
\ytableausetup{boxsize=normal,aligntableaux=center}
 S
\left(
 S_{(3,1)}
\left(~
\begin{ytableau}
 \gamma_{11} & \gamma_{12} & \gamma_{13} \\
 \gamma_{21} 
\end{ytableau}
~\right)
\right)
&=E_{(\gamma_{13})}E_{(\gamma_{12})}E_{(\gamma_{21},\gamma_{11})}
-E_{(\gamma_{12},\gamma_{13})}E_{(\gamma_{21},\gamma_{11})}\\
&\ \ \ -E_{(\gamma_{13})}E_{(\gamma_{21},\gamma_{11},\gamma_{12})}
+E_{(\gamma_{21},\gamma_{11},\gamma_{12},\gamma_{13})}
\end{align*}
 where each term corresponds to the $H$-rim decomposition 
 \ytableausetup{mathmode,boxsize=10pt,aligntableaux=center} 
\!\!\!\!\!
$\begin{ytableau}
 \none & 1 \\
 \none & 2 \\
 3 & 3 
\end{ytableau}
$\,, 
$\begin{ytableau}
 \none & 2 \\
 \none & 2 \\
 3 & 3 
\end{ytableau}
$\,,
$\begin{ytableau}
 \none & 1 \\
 \none & 3 \\
 3 & 3 
\end{ytableau}
$\,
 and 
$\begin{ytableau}
 \none & 3 \\
 \none & 3 \\
 3 & 3 
\end{ytableau}
$\, of $(3,1)^{\#}=(2,2,2)/(1,1)$, respectively,
 and from \eqref{for:SSnuE} 
\begin{align*}
\ytableausetup{boxsize=normal,aligntableaux=center}
 S
\left(
 S_{(3,1)}
\left(~
\begin{ytableau}
 \gamma_{11} & \gamma_{12} & \gamma_{13} \\
 \gamma_{21} 
\end{ytableau}
~\right)
\right)
&=M_{(\gamma_{21})}M_{(\gamma_{13},\gamma_{12},\gamma_{11})}-M_{(\gamma_{13},\gamma_{12},\gamma_{11},\gamma_{21})},
\end{align*}
 where each term to the $E$-rim decomposition 
 \ytableausetup{mathmode,boxsize=10pt,aligntableaux=center} 
\!\!\!\!\!
$\begin{ytableau}
 \none & 2 \\
 \none & 2 \\
 1 & 2 
\end{ytableau}
$\,
 and 
$\begin{ytableau}
 \none & 2 \\
 \none & 2 \\
 2 & 2 
\end{ytableau}
$\,, respectively.
\end{exam}

\begin{rmk}
 The equation \eqref{for:SSnu} is essentially obtained by
 Malvenuto and Reutenauer \cite[Theorem~3.1]{MalvenutoReutenauer1998} for their quasi-symmetric functions.
 Notice that $\nu^{\#}$ is called the conjugate of $\nu$ in their notion.
 If Jacobi-Trudi formulas are obtained for such quasi-symmetric functions, 
 then one may also establish the similar kind of expressions like \eqref{for:SSnuH} and \eqref{for:SSnuE} for them. 
\end{rmk}

 Using Theorem~\ref{thm:SSchurquasi},
 one automatically gets another relation from a given relation among quasi-symmetric functions by mapping it by the antipode $S$.
 For instance, from \eqref{for:dualCauchyqsym}, 
 we obtain the following equation.
 
\begin{cor}
 Assume that ${\pmb \gamma}=(\gamma_{ij})\in T^{\mathrm{diag}}((s^r),\mathbb{N})$ and
 ${\pmb \delta}=(\delta_{ij})\in T^{\mathrm{diag}}((r^s),\mathbb{N})$
 with $c_{k}=\gamma_{i,i+k}$ and $d_{k}=\delta_{i,i+k}$ for $k\in\mathbb{Z}$.
 Write $\eta=r+s$.
 Then, we have 
\begin{align*}
&\sum_{\lambda\subset (r^s)}(-1)^{|\lambda|}
S_{(r^s)/\lambda}\left(({\pmb \gamma}|_{\lambda^{\ast}})^{\#}\right)S_{(s^r)/\lambda^{\ast}}\left(({\pmb \delta}|_{\lambda})^{\#}\right)\\
&=
\det
\left[
\begin{array}{ccccccc}
 1 & -M_{(c_{1-r})} & M_{(c_{2-r},c_{1-r})} & \cdots & (-1)^{r}M_{(c_{0},\ldots,c_{1-r})} & \cdots & (-1)^{\eta-1}M_{(c_{\eta-1-r},\ldots,c_{1-r})} \\
 0 & 1 & -M_{(c_{2-r})} & \cdots & (-1)^{r-1}M_{(c_{0},\ldots,c_{2-r})} & \cdots & (-1)^{\eta-2}M_{(c_{\eta-1-r},\ldots,c_{2-r})} \\
 \vdots & \ddots & \ddots & \ddots & \vdots & & \vdots \\
 0 & \cdots & 0 & 1 & -M_{(c_0)} & \cdots & (-1)^{\eta-r}M_{(c_{\eta-1-r},\ldots,c_{0})} \\
\hline 
 1 & -M_{(d_{1-s})} & M_{(d_{2-s},d_{1-s})} & \cdots & (-1)^{s}M_{(d_{0},\ldots,d_{1-s})} & \cdots & (-1)^{\eta-1}M_{(d_{\eta-1-s},\ldots,d_{1-s})} \\
 0 & 1 & -M_{(d_{2-s})} & \cdots & (-1)^{s-1}M_{(d_{0},\ldots,d_{2-s})} & \cdots & (-1)^{\eta-2}M_{(d_{\eta-1-s},\ldots,d_{2-s})} \\
 \vdots & \ddots & \ddots & \ddots & \vdots & & \vdots \\
 0 & \cdots & 0 & 1 & -M_{(d_0)} & \cdots & (-1)^{\eta-s}M_{(d_{\eta-1-s},\ldots,d_{0})} 
\end{array}
\right]
.
\end{align*}
 Here, ${\pmb \gamma}|_{\lambda^{\ast}}\in T^{\mathrm{diag}}(\lambda^{\ast},\mathbb{N})$ and ${\pmb \delta}|_{\lambda}\in T^{\mathrm{diag}}(\lambda,\mathbb{N})$
 are the shape restriction of ${\pmb \gamma}$ and ${\pmb \delta}$ to $\lambda^{\ast}$ and $\lambda$, respectively.
\end{cor}

 We remark that mapping this equation by $\rho\phi^{-1}$ under suitable convergence assumptions,
 one obtains the corresponding relation among the Schur multiple zeta values.


\section{Integral representations of Schur multiple zeta values}
 
 We finally show that the Schur multiple zeta value (SMZV for short) has an iterated integral representation
 when it is of ribbon type. 
 
\subsection{Integral representations} 

 For nonnegative integers $p_1,q_1,\ldots,p_r,q_r$,
 we denote by
\[
 \mathrm{rib}(h_1,\ldots,h_r)
=\mathrm{rib}(p_1,q_1\,|\,\cdots\,|\,p_r,q_r)
\]
 the ribbon
 obtained by connecting hooks $h_1=(p_1,1^{q_1}),\ldots,h_r=(p_r,1^{q_r})$ from the top right to the bottom left.
 For example, 
\[
\ytableausetup{mathmode,boxsize=10pt,aligntableaux=center} 
 \mathrm{rib}(3,3)
\,
=
\,
\ydiagram{3,1,1,1}\,,
\quad
 \mathrm{rib}(0,3\,|\,4,0)
\,
=
\,
\ydiagram{3+1,3+1,3+1,4}\,,
\quad
 \mathrm{rib}(4,1\,|\,2,3)
\,
=
\,
\ydiagram{1+4,1+1,2,1,1,1}\,.
\]
 Notice that  
 $\mathrm{rib}(p,0)=(p)$,
 $\mathrm{rib}(0,q)=(1^q)$,
 $\mathrm{rib}(p,q)=(p,1^{q})$ is a hook and 
 $\mathrm{rib}(0,q\,|\,p,0)=(p^{q+1})/((p-1)^q)$ is an anti-hook. 
 To guarantee the uniqueness of such expressions,
 we choose the minimum $r$ in the above expression. 
 For example,
 one cannot write $(p)=\mathrm{rib}(p-1,0\,|\,1,0)$ or $(1^q)=\mathrm{rib}(0,q-1\,|\,1,0)$, and so on.
 We remark that $p_i\ne 0,q_i=0$ and $p_{i+1}\ne 0$ may occur for some $i$, however,
 $q_i\ne 0,p_{i+1}=0$ and $q_{i+1}\ne 0$ never does for any $i$.
 For a ribbon $\nu$, let
\[
 I_{\nu}
=\left\{{\pmb \gamma}=(\gamma_{ij})\in T(\nu,\mathbb{N})\,\left|\,\text{$\gamma_{ij}\ge 2$ for all $(i,j)\in C(\nu)$}\right.\right\},
\]
 which also appeared in the previous section.
 For ${\pmb \gamma}=(\gamma_{ij})\in I_{\nu}$, we put $|{\pmb \gamma}|=\sum_{(i,j)\in D(\nu)}\gamma_{ij}$.
 
 It is well known that $\zeta_{(1^q)}({\pmb \beta})$, that is,
 the multiple zeta value (MZV for short), 
 has the following iterated integral representation (see, e.g., \cite{Za});
\begin{equation}
\label{for:intMZV}
\ytableausetup{boxsize=normal,aligntableaux=center}
\begin{ytableau}
 \beta_1 \\
 \lower2pt\vdots \\
 \beta_{q}
\end{ytableau} 
\,
=
\,
\int_{\Delta({\pmb \beta})}
\prod^{q}_{l=1}
\left(\frac{dy_{\beta_1+\cdots+\beta_{l-1}+1}}{1-y_{\beta_1+\cdots+\beta_{l-1}+1}}
\prod^{\beta_1+\cdots+\beta_{l-1}+\beta_l}_{j=\beta_1+\cdots+\beta_{l-1}+2}\frac{dy_j}{y_j}\right),
\end{equation}
 where ${\pmb \beta}\in I_{(1^q)}$ is the Young tableau in the left hand side of \eqref{for:intMZV} 
 and 
\[
 \Delta({\pmb \beta})
=\left\{{\pmb y}=(y_1,\ldots,y_{|{\pmb \beta}|})\in[0,1]^{|{\pmb \beta}|}\,\left|\,Q_{\pmb \beta}({\pmb y})\right.\right\}
\]
 with the condition 
\[
 Q_{\pmb \beta}({\pmb y})\,:\,y_1<\cdots<y_{|{\pmb \beta}|}.
\]
 Notice that the empty sum and product should be taken to be $0$ and $1$, respectively.
 Moreover, it is shown in \cite{Yamamoto} that 
 $\zeta_{(p)}({\pmb \alpha})$, that is, the multiple zeta-star value (MZSV for short), 
 also has a similar integral expression as
\begin{equation}
\label{for:intMZSV}
\ytableausetup{boxsize=normal,aligntableaux=center}
\begin{ytableau}
 \alpha_1 & \cdots & \alpha_{p}
\end{ytableau} 
\,
=
\,
\int_{\Delta({\pmb \alpha})}
\prod^{p}_{k=1}
\left(\frac{dx_{\alpha_p+\cdots+\alpha_{p+2-k}+1}}{1-x_{\alpha_p+\cdots+\alpha_{p+2-k}+1}}
\prod^{\alpha_p+\cdots+\alpha_{p+2-k}+\alpha_{p+1-k}}_{j=\alpha_p+\cdots+\alpha_{p+2-k}+2}\frac{dx_j}{x_j}\right),
\end{equation}
 where ${\pmb \alpha}\in I_{(p)}$ is the Young tableau in the left hand side of \eqref{for:intMZSV} and  
\[
 \Delta({\pmb \alpha})
=\left\{{\pmb x}=(x_1,\ldots,x_{|{\pmb \alpha}|})\in[0,1]^{|{\pmb \alpha}|}\,\left|\,P_{\pmb \alpha}({\pmb x})\right.\right\} 
\]
 with the condition 
\[
 P_{\pmb \alpha}({\pmb x})\,:\,
\left\{
\begin{array}{l}
 \text{$x_j<x_{j+1}$ if $j\notin\{\alpha_p,\alpha_p+\alpha_{p-1},\ldots,\alpha_p+\cdots+\alpha_2\}$}, \\[3pt]
 \text{$x_j>x_{j+1}$ if $j\in\{\alpha_p,\alpha_p+\alpha_{p-1},\ldots,\alpha_p+\cdots+\alpha_2\}$}.
\end{array}
\right.
\] 
 Furthermore,
 in \cite[Theorem~4.1]{KanekoYamamoto},
 the following integral expression for the SMZV of anti-hook type,
 which is denoted by $\zeta(\mu({\pmb k},{\pmb l}))$ in \cite{KanekoYamamoto}, is obtained;
\begin{align}
\label{for:intSMZantihook}
\ytableausetup{boxsize=normal,aligntableaux=center}
\begin{ytableau}
 \none & \none & \beta_1 \\
 \none & \none & \lower2pt\vdots \\
 \none & \none & \beta_q \\ 
 \alpha_1 & \cdots & \alpha_{p}
\end{ytableau} 
\,
&=
\,
\int_{\Delta({\pmb \delta})}
\prod^{q}_{l=1}
\left(\frac{dy_{\beta_1+\cdots+\beta_{l-1}+1}}{1-y_{\beta_1+\cdots+\beta_{l-1}+1}}
\prod^{\beta_1+\cdots+\beta_{l-1}+\beta_l}_{j=\beta_1+\cdots+\beta_{l-1}+2}\frac{dy_j}{y_j}\right)\\
\nonumber
&\ \ \ \times
\prod^{p}_{k=1}
\left(\frac{dx_{\alpha_p+\cdots+\alpha_{p+2-k}+1}}{1-x_{\alpha_p+\cdots+\alpha_{p+2-k}+1}}
\prod^{\alpha_p+\cdots+\alpha_{p+2-k}+\alpha_{p+1-k}}_{j=\alpha_p+\cdots+\alpha_{p+2-k}+2}\frac{dx_j}{x_j}\right),
\end{align}
 where ${\pmb \delta}={\pmb \delta}({\pmb \beta},{\pmb \alpha})\in I_{\mathrm{rib}(0,q\,|\,p,0)}$ is the Young tableau in the left hand side of \eqref{for:intSMZantihook},
 ${\pmb \alpha},{\pmb \beta}$ are the previous tableaux with a relaxed condition $\beta_q\ge 1$ if $p\ne 0$ and  
\[
 \Delta({\pmb \delta})
=\left\{({\pmb y},{\pmb x})=(y_1,\ldots,y_{|{\pmb \beta}|},x_1,\ldots,x_{|{\pmb \alpha}|})\in[0,1]^{|{\pmb \beta}|+|{\pmb \alpha}|}\,\left|\,
\text{$Q_{\pmb \beta}({\pmb y})$, $P_{\pmb \alpha}({\pmb x})$ and $y_{|{\pmb \beta}|}<x_1$}\right.\right\}. 
\]
 By the same idea,
 one obtains the formula for that of hook type;
\begin{align}
\label{for:intSMZhook}
\ytableausetup{boxsize=normal,aligntableaux=center}
\begin{ytableau}
 \alpha_1 & \cdots & \alpha_{p} \\
 \beta_1 \\
 \lower2pt\vdots \\ 
 \beta_q
\end{ytableau} 
\,
&=
\,
\int_{\Delta({\pmb \gamma})}
\prod^{p}_{k=1}
\left(\frac{dx_{\alpha_p+\cdots+\alpha_{p+2-k}+1}}{1-x_{\alpha_p+\cdots+\alpha_{p+2-k}+1}}
\prod^{\alpha_p+\cdots+\alpha_{p+2-k}+\alpha_{p+1-k}}_{j=\alpha_p+\cdots+\alpha_{p+2-k}+2}\frac{dx_j}{x_j}\right)\\
\nonumber
&\ \ \ \times
\prod^{q}_{l=1}
\left(\frac{dy_{\beta_1+\cdots+\beta_{l-1}+1}}{1-y_{\beta_1+\cdots+\beta_{l-1}+1}}
\prod^{\beta_1+\cdots+\beta_{l-1}+\beta_l}_{j=\beta_1+\cdots+\beta_{l-1}+2}\frac{dy_j}{y_j}\right),
\end{align}
 where ${\pmb \gamma}={\pmb \gamma}({\pmb \alpha},{\pmb \beta})\in I_{\mathrm{rib}(p,q)}$ is the Young tableau in the left hand side of \eqref{for:intSMZhook}
 and
\[
 \Delta({\pmb \gamma})
=\left\{({\pmb x},{\pmb y})=(x_1,\ldots,x_{|{\pmb \alpha}|},y_1,\ldots,y_{|{\pmb \beta}|})\in[0,1]^{|{\pmb \alpha}|+|{\pmb \beta}|}\,\left|\,
\text{$P_{\pmb \alpha}({\pmb x})$, $Q_{\pmb \beta}({\pmb y})$ and $x_{|{\pmb \alpha}|}<y_1$}\right.\right\}. 
\]

\begin{exam} 
 We have
\begin{align}
\ytableausetup{boxsize=normal,aligntableaux=center}
\begin{ytableau}
 2 \\
 1 \\
 3
\end{ytableau} 
\,
&=
\,
\int_{y_1<y_2<y_3<y_4<y_5<y_6}
\frac{dy_1}{1-y_1}\frac{dy_2}{y_2}\frac{dy_3}{1-y_3}\frac{dy_4}{1-y_4}\frac{dy_5}{y_5}\frac{dy_6}{y_6}, \nonumber\\
\begin{ytableau}
 3 & 1 & 2 
\end{ytableau} 
\,
&=
\,
\int_{x_1<x_2>x_3>x_4<x_5<x_6}
\frac{dx_1}{1-x_1}\frac{dx_2}{x_2}\frac{dx_3}{1-x_3}\frac{dx_4}{1-x_4}\frac{dx_5}{x_5}\frac{dx_6}{x_6}, \nonumber\\
\label{for:intA}
\begin{ytableau}
 \none & 3 \\
 \none & 1 \\
 2 & 2
\end{ytableau}
\,
&=
\,
\int_{y_1<y_2<y_3<y_4<x_1<x_2>x_3<x_4}
\frac{dy_1}{1-y_1}\frac{dy_2}{y_2}\frac{dy_3}{y_3}\frac{dy_4}{1-y_4}\frac{dx_1}{1-x_1}\frac{dx_2}{x_2}\frac{dx_3}{1-x_3}\frac{dx_4}{x_4},\\
\label{for:intB}
\begin{ytableau}
 3 & 2 \\
 1 \\
 2
\end{ytableau}
\,
&=
\,
\int_{x_1<x_2>x_3<x_4<x_5<y_1<y_2<y_3}
\frac{dx_1}{1-x_1}\frac{dx_2}{x_2}\frac{dx_3}{1-x_3}\frac{dx_4}{x_4}\frac{dx_5}{x_5}\frac{dy_1}{1-y_1}\frac{dy_2}{1-y_2}\frac{dy_3}{y_3}.
\end{align} 
 Notice that we omit the condition $0<x_i,y_i<1$ from the notations.
\end{exam}
 
 Now, one can easily generalize these results to SMZVs of ribbon type as follows.

\begin{thm}
\label{thm:intrep}
 Let $p_1,q_1,\ldots,p_r,q_r$ be nonnegative integers.
 For $1\le i\le r$, let $h_i=(p_i,1^{q_i})$ be a hook and
\[
\ytableausetup{boxsize=20pt,aligntableaux=center}
 {\pmb \alpha}_i
=\,
\begin{ytableau}
 \alpha^{(i)}_1 & \cdots & \alpha^{(i)}_{p_i}
\end{ytableau} 
\in I_{(p_i)},
\quad
 {\pmb \beta}_i
=\,
\begin{ytableau}
 \beta^{(i)}_1 \\
 \lower2pt\vdots \\ 
 \beta^{(i)}_{q_i}
\end{ytableau} 
\in I_{(1^{q_i})},
\quad
 {\pmb \gamma}_i
={\pmb \gamma}({\pmb \alpha}_i,{\pmb \beta_i})
=
\,
\begin{ytableau}
 \alpha^{(i)}_1 & \cdots & \alpha^{(i)}_{p_i} \\
 \beta^{(i)}_1 \\
 \lower2pt\vdots \\ 
 \beta^{(i)}_{q_i}
\end{ytableau} 
\in I_{h_i},
\] 
 with relaxed conditions $\beta^{(i)}_{q_i}\ge 1$ for $1\le i\le r-1$.
 Define ${\pmb \gamma}={\pmb \gamma}_1\sqcup\cdots\sqcup {\pmb \gamma}_r\in I_{\mathrm{rib}(h_1,\ldots,h_r)}$
 by connecting ${\pmb \gamma}_1,\ldots,{\pmb \gamma}_r$ from the top right to the bottom left.
 Then, it holds that 
\begin{align*}
 \zeta_{\mathrm{rib}(h_1,\ldots,h_r)}({\pmb \gamma})
&=
\int_{\Delta({\pmb \gamma})}
\prod^{r}_{i=1}
\left[
\prod^{p_i}_{k=1}
\left(\frac{dx^{(i)}_{\alpha^{(i)}_{p_i}+\cdots+\alpha^{(i)}_{{p_i}+2-k}+1}}{1-x^{(i)}_{\alpha^{(i)}_{p_i}+\cdots+\alpha^{(i)}_{{p_i}+2-k}+1}}
\prod^{\alpha^{(i)}_{p_i}+\cdots+\alpha^{(i)}_{{p_i}+2-k}+\alpha^{(i)}_{{p_i}+1-k}}_{j=\alpha^{(i)}_{p_i}+\cdots+\alpha^{(i)}_{{p_i}+2-k}+2}\frac{dx^{(i)}_j}{x^{(i)}_j}\right)
\right.
\\
\nonumber
&\ \ \ \left.
\times
\prod^{q_i}_{l=1}
\left(\frac{dy^{(i)}_{\beta^{(i)}_1+\cdots+\beta^{(i)}_{l-1}+1}}{1-y^{(i)}_{\beta^{(i)}_1+\cdots+\beta^{(i)}_{l-1}+1}}
\prod^{\beta^{(i)}_1+\cdots+\beta^{(i)}_{l-1}+\beta^{(i)}_l}_{j=\beta^{(i)}_1+\cdots+\beta^{(i)}_{l-1}+2}\frac{dy^{(i)}_j}{y^{(i)}_j}\right)
\right],
\end{align*} 
 where
\begin{align*}
 \Delta({\pmb \gamma})
=\left\{({\pmb x}_1,{\pmb y}_1,\ldots,{\pmb x}_r,{\pmb y}_r)\in[0,1]^{\sum^{r}_{i=1}(|{\pmb \alpha}_i|+|{\pmb \beta}_i|)}\,\left|\,
\begin{array}{l}
\text{$P_{{\pmb \alpha}_i}({\pmb x}_i)$, $Q_{{\pmb \beta}_i}({\pmb y}_i)$ \ $(1\le i\le r)$}, \\[5pt]
\text{$x^{(i)}_{|{\pmb \alpha}_i|}<
\begin{cases}
y^{(i)}_1 & q_i\ne 0 \\[3pt]
x^{(i+1)}_1 & q_i=0 \\
\end{cases}
$ \ $(1\le i\le r)$}, \\[15pt]
\text{$y^{(i)}_{|{\pmb \beta}_i|}<x^{(i+1)}_1$ \ $(1\le i\le r-1)$} 
\end{array}
\right.\right\}
.
\end{align*}
 Here, we write ${\pmb x}_i=(x^{(i)}_1,\ldots,x^{(i)}_{|{\pmb \alpha}_i|})$ and ${\pmb y}_i=(y^{(i)}_1,\ldots,y^{(i)}_{|{\pmb \beta}_i|})$ for $1\le i\le r$.
\end{thm}
\begin{proof}
 This is direct. 
 One can understand the general case by the following example: 
\begin{equation}
\label{for:intexample}
\ytableausetup{boxsize=normal,aligntableaux=center}
\begin{ytableau}
 \none & 1 & 2 \\
 \none & 1 \\
 2 & 2 
\end{ytableau}
\,
=
\,
\int_{x_1<x_2>x_3<y_1<z_1<z_2>z_3<z_4}
\frac{dx_1}{1-x_1}\frac{dx_2}{x_2}\frac{dx_3}{1-x_3}\frac{dy_1}{1-y_1}
\frac{dz_1}{1-z_1}\frac{dz_2}{z_2}\frac{dz_3}{1-z_3}\frac{dz_4}{z_4},
\end{equation}
 which is the case of $h_1=(2,1)$ and $h_2=(2)$.
 Actually, we have 
\begin{align*}
\ytableausetup{boxsize=normal,aligntableaux=center}
 (\text{RHS of \eqref{for:intexample}}) 
&=\sum^{\infty}_{l=1}\frac{1}{l}\int_{x_2>x_3<y_1<z_1<z_2>z_3<z_4}
x^{l-1}_2dx_2\frac{dx_3}{1-x_3}\frac{dy_1}{1-y_1}
\frac{dz_1}{1-z_1}\frac{dz_2}{z_2}\frac{dz_3}{1-z_3}\frac{dz_4}{z_4}\\
&=\sum^{\infty}_{l=1}\frac{1}{l^2}\int_{x_3<y_1<z_1<z_2>z_3<z_4}
\frac{1-x_3^{l}}{1-x_3}dx_3\frac{dy_1}{1-y_1}
\frac{dz_1}{1-z_1}\frac{dz_2}{z_2}\frac{dz_3}{1-z_3}\frac{dz_4}{z_4}\\
&=\sum^{\infty}_{l=1}\sum^{l}_{k=1}\frac{1}{l^2}\frac{1}{k}\int_{y_1<z_1<z_2>z_3<z_4}
\frac{y_1^k}{1-y_1}dy_1
\frac{dz_1}{1-z_1}\frac{dz_2}{z_2}\frac{dz_3}{1-z_3}\frac{dz_4}{z_4}\\
&=\sum^{\infty}_{l=1}\sum^{l}_{k=1}\sum^{\infty}_{m=1}\frac{1}{l^2}\frac{1}{k}\frac{1}{k+m}\int_{z_1<z_2>z_3<z_4}
\frac{z_1^{k+m}}{1-z_1}dz_1\frac{dz_2}{z_2}\frac{dz_3}{1-z_3}\frac{dz_4}{z_4}\\
&=\sum^{\infty}_{l=1}\sum^{l}_{k=1}\sum^{\infty}_{m=1}\sum^{\infty}_{n=1}\frac{1}{l^2}\frac{1}{k}\frac{1}{k+m}\frac{1}{k+m+n}\int_{z_2>z_3<z_4}
z_2^{k+m+n-1}dz_2\frac{dz_3}{1-z_3}\frac{dz_4}{z_4}\\
&=\sum^{\infty}_{l=1}\sum^{l}_{k=1}\sum^{\infty}_{m=1}\sum^{\infty}_{n=1}
\frac{1}{l^2}\frac{1}{k}\frac{1}{k+m}\frac{1}{(k+m+n)^2}\int_{z_3<z_4}
\frac{1-z^{k+m+n}_3}{1-z_3}dz_3\frac{dz_4}{z_4}\\
&=\sum^{\infty}_{l=1}\sum^{l}_{k=1}\sum^{\infty}_{m=1}\sum^{\infty}_{n=1}\sum^{k+m+n}_{j=1}
\frac{1}{l^2}\frac{1}{k}\frac{1}{k+m}\frac{1}{(k+m+n)^2}\frac{1}{j}\int^{1}_{0}
z^{j-1}_4dz_4\\
&=\sum^{\infty}_{l=1}\sum^{l}_{k=1}\sum^{\infty}_{m=1}\sum^{\infty}_{n=1}\sum^{k+m+n}_{j=1}
\frac{1}{l^2}\frac{1}{k}\frac{1}{k+m}\frac{1}{(k+m+n)^2}\frac{1}{j^2} \displaybreak[3] \\
&=\sum_{a\le b,\,a<c<e,\,d\le e}
\frac{1}{b^2}\frac{1}{a}\frac{1}{c}\frac{1}{e^2}\frac{1}{d^2}\\[3pt]
&=(\text{LHS of \eqref{for:intexample}}).
\end{align*}
\end{proof}

\begin{rmk}
 It seems to be difficult to express a general (that is, non-ribbon type) SMZVs as a single iterated integral as above. 
 Notice that, since we have the expressions \eqref{for:SchurtoMZVskew},
 every SMZV can be written as a {\it sum} of such integrals. 
\end{rmk}

\subsection{A duality}

 From \eqref{for:intexample}, we have 
\begin{align*}
\ytableausetup{boxsize=normal,aligntableaux=center}
\begin{ytableau}
 \none & 1 & 2 \\
 \none & 1 \\
 2 & 2 
\end{ytableau}
\,
&=
\,
\int_{t_1<t_2>t_3<t_4<t_5<t_6>t_7<t_8}
\frac{dt_1}{1-t_1}\frac{dt_2}{t_2}\frac{dt_3}{1-t_3}\frac{dt_4}{1-t_4}
\frac{dt_5}{1-t_5}\frac{dt_6}{t_6}\frac{dt_7}{1-t_7}\frac{dt_8}{t_8}\\
&=
\int_{t'_1<t'_2<t'_3>t'_4<t'_5<t'_6>t'_7<t'_8}
\frac{dt'_1}{1-t'_1}\frac{dt'_2}{t'_2}\frac{dt'_3}{1-t'_3}\frac{dt'_4}{t'_4}
\frac{dt'_5}{t'_5}\frac{dt'_6}{t'_6}\frac{dt'_7}{1-t'_7}\frac{dt'_8}{t'_8}\nonumber\\[3pt]
&=
\,
\begin{ytableau}
 2 & 4 & 2  
\end{ytableau}
\,
.
\nonumber
\end{align*} 
 Here, in the second equality, we have made a change of variables $t'_i=1-t_{9-i}$ for $1\le i\le 8$.
 Such a kind of relation is called a {\it duality}.
 The duality for MZVs is well-known (see \cite{Za}).
 On the other hand, the duality for MZSVs has not been obtained yet (see \cite{KanekoOhno2010,Li2012} for another kind of duality for MZSVs).
 Theorem~\ref{thm:intrep} immediately implies that there exists a duality for SMZVs of ribbon type, however,
 one should remark that in general the dual (in the above sense) of a SMZV of ribbon type is not of ribbon type again.  
 In fact, one sees that, if ${\pmb \gamma}=(\gamma_{ij})\in I_{\mathrm{rib}(h_1,\ldots,h_r)}$ has an entry $\gamma_{ij}=1$ 
 where $(i,j)\in D(\mathrm{rib}(h_1,\ldots,h_r))$ is a horizontal and non-vertical entry
 or the entry where the ribbon ends,  
 then the dual of $\zeta_{\mathrm{rib}(h_1,\ldots,h_r)}({\pmb \gamma})$ is not of ribbon type.

\begin{exam}
 One easily sees the following dualities   
\[
\ytableausetup{boxsize=normal,aligntableaux=center}
\begin{ytableau}
 \none & \none & 2 \\
 \none & 1 & 2 \\
 2 & 2  
\end{ytableau}
\,
=
\,
\begin{ytableau}
 2 & 3 & 2 \\
 2  
\end{ytableau}
\,,
\quad
\begin{ytableau}
 \none & \none & \none & 3 & 2 \\
 \none & \none & \none & 1 \\
 \none & 1 & 2 & 2 \\
 1 & 2 \\
 2  
\end{ytableau}
\,
=
\,
\begin{ytableau}
 \none & 3 & 2 & 3 & 3 \\
 \none & 1 \\
 2 & 2  
\end{ytableau}
\] 
 On the other hand, we have
\begin{align*}
\begin{ytableau}
 \none & \none & 1 \\
 2 & 1 & 2 
\end{ytableau}
\,
&=\int_{t_1<t_2<t_3>t_4>t_5<t_6}
\frac{dt_1}{1-t_1}\frac{dt_2}{1-t_2}\frac{dt_3}{t_3}\frac{dt_4}{1-t_4}\frac{dt_5}{1-t_5}\frac{dt_6}{t_6}\\
&=\int_{t'_1<t'_2>t'_3>t'_4<t'_5<t'_6}
\frac{dt'_1}{1-t'_1}\frac{dt'_2}{t'_2}\frac{dt'_3}{t'_3}\frac{dt'_4}{1-t'_4}\frac{dt'_5}{t'_5}\frac{dt'_6}{t'_6}
\end{align*}
 and see that the rightmost hand side above can not be realized as a single SMZV of ribbon type
 because of the $(2,2)$ entry $1$ of the Young tableau in the left hand side. 
 Notice that there are "self-dual" SMZVs.
 For example, the dual of  
 \ytableausetup{mathmode,boxsize=10pt,aligntableaux=center} 
\!\!\!\!\!
$\begin{ytableau}
 \none & 2 & 2 \\
 2 & 2 
\end{ytableau}
$\,
is itself.
\end{exam}

 When MZSV has no $1$ entries, its dual can be explicitly written as follows.

\begin{cor}
 Let $\alpha_1,\ldots,\alpha_p\ge 2$ be positive integers.
 Then, it holds that 
\[
\ytableausetup{boxsize=normal,aligntableaux=center}
\begin{ytableau}
 \alpha_1 & \cdots & \alpha_{p}
\end{ytableau} 
\,
=
\,
\begin{ytableau}
 \none & \none & \none & \none & \none & 1 \\
 \none & \none & \none & \none & \none & \lower2pt\vdots  \\
 \none & \none & \none & \none & \none & 1 \\
 \none & \none & \none & \none & 1 & 2 \\
 \none & \none & \none & \none & \lower2pt\vdots  \\
 \none & \none & \none & \none & 1 \\
 \none & \none & \none & 1 & 2 \\
 \none & \none & \none & 1 \\
 \none & \none & \none[\reflectbox{$\ddots$}] \\
 \none & 1 \\ 
 1 & 2 \\ 
 \lower2pt\vdots  \\
 1 \\
 2 
\end{ytableau} 
,
\]
 where the ribbon in the right hand side has $p$ columns and $\alpha_{p+1-j}-1$ boxes in the $j$th column for $1\le j\le p$.
\end{cor} 

\begin{rmk}
 If one obtains a duality for the SMZVs, then, from \eqref{for:SchurtoMZVskew},
 one may be able to get a linear relation for MZVs and MZSVs. 
 For example, we can check the duality 
\[
\ytableausetup{boxsize=normal,aligntableaux=center}
\begin{ytableau}
 3 & 2
\end{ytableau}
\,
=
\,
\begin{ytableau}
 \none & 1 \\
 2 & 2
\end{ytableau}
\]
 and, from \eqref{for:SchurtoMZVskew}, obtain the relation 
\[
 \zeta(3,2)+\zeta(5)
=\zeta(1,4)+\zeta(1,2,2)+\zeta(3,2)+\zeta(2,1,2).
\] 
 On the other hand, from \eqref{for:intA} and \eqref{for:intB}, we have 
\[
\ytableausetup{boxsize=normal,aligntableaux=center}
\begin{ytableau}
 \none & 3 \\
 \none & 1 \\
 2 & 2
\end{ytableau}
\,
=
\,
\begin{ytableau} 
 3 & 2 \\
 1 \\
 2
\end{ytableau}
\,,
\]
 however, \eqref{for:SchurtoMZVskew} does not yield any relations.
\end{rmk}

\begin{rmk}
 It is proved in \cite[(8) or, more generally, Theorem~A]{Chen2017} that 
\begin{equation}
\label{for:Chen}
\ytableausetup{boxsize=normal,aligntableaux=center}
\begin{ytableau}
 \none & \none & \none & 1 \\
 \none & \none & \none & \lower2pt\vdots \\
 \none & \none & \none & 1 \\
 1 & \cdots & 1 & 2
\end{ytableau}
\,
=\binom{p+q}{p}\zeta(p+q+1),
\end{equation}
 where the shape of the anti-hook in the lefthand side of \eqref{for:Chen} is $((q+1)^p)/(q^{p-1})$,
 that is, $p-1$ and $q$ are the numbers of $1$'s in the horizontal and vertical entries of the anti-hook, respectively. 
 We notice that in \cite{Chen2017}
 the lefthand side of \eqref{for:Chen} is expressed as 
\[
 \sum_{1\le n_1<\cdots<n_p}
\frac{P_q(H^{(1)}_{n_p},\ldots,H^{(q)}_{n_p})}{n_1\cdots n_{p-1}n_p^2},
\]
 where $P_q(x_1,\ldots,x_q)$ is the modified Bell polynomial
 and $H^{(k)}_{n}=\sum^{n}_{m=1}\frac{1}{m^k}$
 is the generalized harmonic number.
 One can also prove \eqref{for:Chen} via the integral representations
 and their duality in the above sense.
 Actually, from \eqref{for:intSMZantihook}, we have 
\begin{align*}
 (\text{LHS of \eqref{for:Chen}}) 
&=\int_{y_1<\cdots<y_p<z>x_q>\cdots>x_1}
\frac{dy_1}{1-y_1}\cdots \frac{dy_p}{1-y_p}\frac{dz}{z}\frac{dx_q}{1-x_q}\cdots \frac{dx_1}{1-x_1}\\
&=\int_{x'_1>\cdots>x'_q>z'<y'_p<\cdots<y'_1}
\frac{dx'_1}{x'_1}\cdots \frac{dx'_q}{x'_q}\frac{dz'}{1-z'}\frac{dy'_p}{y'_p}\cdots \frac{dy'_1}{y'_1}\\
&=\binom{p+q}{p}\int_{z'<w_1<\cdots<w_{p+q}}\frac{dz'}{1-z'}\frac{dw_1}{w_1}\cdots \frac{dw_{p+q}}{w_{p+1}}\\
&=\binom{p+q}{p}\zeta(p+q+1)\\
&=(\text{RHS of \eqref{for:Chen}}).
\end{align*} 
\end{rmk}


\section*{Acknowledgement}
 We would like to express our appreciation to all those who gave us valuable advice for this article: 
 Prof. Masatoshi Noumi who provided expertise that greatly helped us to prove the results on Schur multiple zeta functions in Section~\ref{sec:Macdonald}, 
 Prof. Masanobu Kaneko who gave guidance in quasi-symmetric functions and inspired us to establish the generalized result for such functions,
 Prof. Takeshi Ikeda who gave meaningful suggestion for our work and 
 Prof. Michael E. Hoffman who pointed out a mistake in Example~\ref{ex:n4}.
 We also would like to thank Prof. Hiroshi Naruse, Prof. Takashi Nakamura, Prof. Soichi Okada and Prof. Yasuo Ohno for their useful comments in many aspects.
 Moreover, the third-named author is very grateful to the Max-Planck-Institut f\"ur Mathematik in Bonn
 for the hospitality and support during his research stay at the institute.
 Finally, the authors thanks the referees for many comments which improve the paper.



\bigskip

\noindent
\textsc{Maki Nakasuji}\\
 Department of Information and Communication Science, Faculty of Science, \\
 Sophia University, Tokyo, Japan \\
 \texttt{nakasuji@sophia.ac.jp}

\medskip

\noindent
\textsc{Ouamporn Phuksuwan}\\
 Department of Mathematics and Computer Science, Faculty of Science, \\
 Chulalongkorn University, Bangkok, Thailand \\
 \texttt{ouamporn.p@chula.ac.th}

\medskip

\noindent
\textsc{Yoshinori Yamasaki}\\
 Graduate School of Science and Engineering, \\
 Ehime University, Ehime, Japan \\
 \texttt{yamasaki@math.sci.ehime-u.ac.jp}
 
\end{document}